\documentclass[preprint,nopreprintline,10.5pt]{elsarticle}

\usepackage{setspace}

\usepackage[top=1in, bottom=1in, left=1in, right=1in]{geometry}

\usepackage{graphicx}
\usepackage{epstopdf}
\usepackage{amsmath}
\usepackage{booktabs}
\usepackage{multirow}
\usepackage{natbib}
\usepackage{amssymb}
\usepackage{algorithm}
\usepackage{algcompatible}
\usepackage{algpseudocode}
\usepackage{amsthm}
\usepackage{enumitem}

\usepackage{subcaption}
\usepackage[flushleft]{threeparttable}

\usepackage{lineno}
\usepackage{xcolor}
\usepackage[colorinlistoftodos]{todonotes}

\usepackage{enumitem}
\newlist{todolist}{itemize}{2}
\setlist[todolist]{label=$\square$}
\usepackage{pifont}

\usepackage{xfrac}
\usepackage{acronym}
\acrodef{gan}[GAN]{\emph{GAN}}
\acrodef{bc}[BC]{\emph{BC}}
\acrodef{rbc}[RBC]{\emph{RBC}}

\biboptions{sort,compress} 

\usepackage{xcolor}

\newcommand\norm[1]{\left\lVert#1\right\rVert}

\newcommand*\patchAmsMathEnvironmentForLineno[1]{%
  \expandafter\let\csname old#1\expandafter\endcsname\csname #1\endcsname
  \expandafter\let\csname oldend#1\expandafter\endcsname\csname end#1\endcsname
  \renewenvironment{#1}%
     {\linenomath\csname old#1\endcsname}%
     {\csname oldend#1\endcsname\endlinenomath}}%
\newcommand*\patchBothAmsMathEnvironmentsForLineno[1]{%
  \patchAmsMathEnvironmentForLineno{#1}%
  \patchAmsMathEnvironmentForLineno{#1*}}%
\AtBeginDocument{%
\patchBothAmsMathEnvironmentsForLineno{equation}%
\patchBothAmsMathEnvironmentsForLineno{align}%
\patchBothAmsMathEnvironmentsForLineno{flalign}%
\patchBothAmsMathEnvironmentsForLineno{alignat}%
\patchBothAmsMathEnvironmentsForLineno{gather}%
\patchBothAmsMathEnvironmentsForLineno{multline}%
}

\usepackage{graphicx}
\usepackage{amssymb}
\usepackage{amsthm}
\usepackage{bbm}
\usepackage{bm}
\usepackage{lineno}
\usepackage{url}
\usepackage{listings}
\usepackage[colorlinks=true]{hyperref}
\newtheorem{theorem}{Theorem}[]
\newtheorem{remark}{Remark}[section]
\newtheorem{lemma}[theorem]{Lemma}
\newtheorem{proposition}[theorem]{Proposition}
\newtheorem{Observation}{Key Observation}[section]

\usepackage{color,soul}
\usepackage{mathtools}

\definecolor{lightblue}{rgb}{.90,.95,1}
\definecolor{darkgreen}{rgb}{0,.5,0.5}

\definecolor{lightgreen}{rgb}{.90,1,0.90}

\usepackage{gensymb}
\usepackage{array}
\usepackage{changes}
\usepackage{multirow}
\usepackage{enumerate}
\newcolumntype{P}[1]{>{\centering\arraybackslash}m{#1}}

\newcolumntype{L}[1]{>{\raggedright\let\newline\\\arraybackslash\hspace{0pt}}m{#1}}
\newcolumntype{C}[1]{>{\centering\let\newline\\\arraybackslash\hspace{0pt}}m{#1}}
\newcolumntype{R}[1]{>{\raggedleft\let\newline\\\arraybackslash\hspace{0pt}}m{#1}}

\newcommand{\vect}[1]{\boldsymbol{#1}}

\graphicspath{ {./figs/} }

\usepackage{changes}
\definechangesauthor[name={Reviewer 1}, color = blue]{R1}
\definechangesauthor[name={Reviewer 2}, color = red]{R2}
\definechangesauthor[name={All reviewers}, color = brown]{All}
\definechangesauthor[name={Editor}, color = purple]{Editor}
\definechangesauthor[name={Author}, color = olive]{Author}

\usepackage{accents}

\makeatletter
  \renewcommand{\ALG@name}{Key Procedures}
\makeatother

\begin{document}

\begin{frontmatter}

\title{High-order Nodal Space-time Flux Reconstruction Methods for Hyperbolic Conservation Laws on Curvilinear Moving Grids}
\author[umbc]{Meilin Yu\corref{mycor}}

\cortext[mycor]{Corresponding author}
\ead{mlyu@umbc.edu}

\address[umbc]{Department of Mechanical Engineering, \\
University of Maryland, Baltimore County (UMBC), Baltimore, MD 21250}

\begin{abstract}
High-order nodal space-time flux reconstruction (STFR) methods have been developed to solve hyperbolic conservation laws on curvilinear moving grids. Unlike the method-of-lines approach for moving domain simulation, the grid velocity is implicitly embedded into the curvilinear geometric representation of space-time elements. Several key issues in moving domain simulation, including the discrete geometric conservation law (GCL), solution and flux approximation, and aliasing error control, are discussed in the context of the nodal STFR framework. Conditions and the corresponding numerical strategies to reduce aliasing errors due to the curvilinear space-time representation of moving domain problems, including the discrete GCL errors (i.e. one type of aliasing errors in the space-time framework), are then explained and examined. Since a space-time tensor product is used to construct the FR formulation in this study, all space-time schemes show the temporal superconvergence property, similar to that presented by the implicit Runge–Kutta discontinuous Galerkin (IRK-DG) schemes, in moving domain simulation. Specifically, a nominal $k$th order scheme can achieve a ($2k-1$)th order superconvergence rate when solutions on $k$ Gauss--Legendre points are used to construct polynomials in the time dimension.  The robustness of temporal superconvergence in the existence of aliasing errors induced by the curvilinear space-time representation, and upon de-aliasing operations based on polynomial filtering, has been examined with numerical experiments.  
      
\end{abstract}

\end{frontmatter}

\setcounter{page}{1}


    
\section{Introduction}
High-order computational fluid dynamics (CFD) methods~\cite{Wang_EtAl2013} have achieved remarkable progress in recent years, enabling accurate simulation of multi-scale and multi-physics flow phenomena across a wide range of applications, including atmospheric dynamics, turbomachinery, plasma flows, and bio-inspired aerodynamics. 
Numerical simulation of fluid flows over moving domains presents significant challenges in the development of high-order CFD methods based on the arbitrary Lagrangian-Eulerian (ALE) formulation~\cite{DONEA_EtAl_CMAME_1982}. 
Extensive research has been conducted to improve high-order CFD methods for moving domain simulations~\cite{FARHAT_CMAME_2004,PerssonEtAl_CMAME_2009,MAVRIPLIS_JCP_2011,Yu2011,LiangEtAl_JCP2014,KOPRIVA_EtAl_CF_2016,Yu2016,ABE_EtAl_CF_2016}. A central focus in these efforts is to enforce the discrete geometric conservation law (GCL) in mapping from the time-dependent domain to a fixed reference domain in the ALE method. 
Therein, high-order accuracy is typically limited to spatial discretization, and the method-of-lines approach is then adopted for time integration. 


The space-time formulation provides a consistent numerical treatment of both spatial and temporal discretizations, and can resolve GCL automatically to the level of the numerical resolution of the schemes used in the simulation. 
Although some research has been devoted to advancing space-time methods, such as the space-time discontinuous Galerkin (DG) method~\cite{VANDERVEGT_JCP_2002,KLAIJ_EtAl_JCP_2006,Petersen_EtAl_IJNME_2009,Feistauer_Cesenek_2011,Wang_Persson_CF_2015,Corrigan_EtAl_IJNMF_2019,LUO_EtAl_JCP_2021} and the space-time streamline-upwind/Petrov-Galerkin (SUPG) method~\cite{HUGHES_CMAME_1988, TEZDUYAR_EtAl_CMAME_1992, KocherBause_JSC_2014}, further work is required to fully exploit the numerical potential of space-time methods for moving-domain simulations. As is well known, the efficiency of explicit Runge–Kutta time-discretization schemes suffers from Butcher barriers (see \textbf{Theorem} 324A--C in~\cite{Butcher2002}) when the order of accuracy exceeds four. Even some popular implicit Runge–Kutta (IRK) schemes, such as the explicit first stage, single diagonally implicit
Runge–Kutta (ESDIRK) method \cite{BIJL_EtAl_JCP_2002}, need to use more stages than the order in their design. Moreover, non-trivial efforts are needed to preserve their formal order of accuracy when extending time integration methods established on stationary grids to moving domain simulation~\cite{FARHAT_CMAME_2004}. Thus, understanding the convergence rate property of space-time methods, especially in the time dimension for moving domain simulation with large grid deformation, has practical importance. 
Another interesting feature of the space-time method for moving domain simulation is that the grid velocity, which is needed in the ALE formulation, does not explicitly show up in space-time schemes. Instead, it becomes the slope of the space-time element along its time dimension. Thus, the moving domain simulation problem becomes that of space-time simulation using general curvilinear elements, and discrete GCL errors become aliasing errors associated with the curvilinear space-time geometric representation.
Therefore, we need to understand how the general curvilinear space-time element representation of moving grids affects simulation accuracy, and how to effectively control the associated numerical errors, especially aliasing errors. Some previous works~\cite{PerssonEtAl_CMAME_2009,YU201470,AbeEtAl_JCP_2015,CICCHINO_EtAl_JCP_2022} on high-order simulation with curvilinear grids in the context of the method of lines have provided insights on possible solutions.  
 
In this study, a  high-order nodal tensor-product space-time flux reconstruction (STFR) method is developed for moving domain simulation. 
We mention that an implicit space-time method with right Radau points in the time dimension has been developed by Huynh~\cite{Huynh_AIAA_2013} for conservation laws. Recently, Huynh~\cite{Huynh_JSC_2023} showed the equivalence between the DG-type discretization methods of ordinary differential equations (ODEs) and several IRK methods, including Radau IA, Radau IIA~\cite{Ehle_SJMA_1973}, and DG-Gauss~\cite{Bottasso_ANM_97,Tang_Sun_AMC_12} with the assistance of the flux reconstruction (FR) concept. Note that FR was originally developed by Huynh~\cite{huynh2007,huynh2009} for compact spatial discretization of partial differential equations (PDEs), and is a generalization of various types of discontinuous finite/spectral element methods, such as DG~\cite{Cockburn_DG_1989,bassi1997,Hesthaven_Warburton_08}, spectral difference~\cite{Kopriva:1996,Wang_SD_2006}, and spectral volume~\cite{Wang_SV_2002}. The FR method for spatial discretization has been well established; see some fundamental contributions in~\cite{wang2009,vincent2011} and a review~\cite{Huynh_EtAl_CF_14} that summarized FR-related numerical algorithm developments and applications by 2014. In~\cite{Yu_Space_Time_2017}, we developed a nodal STFR method based on Gauss--Legendre points to solve hyperbolic conservation laws on stationary grids. Temporal superconvergence with a convergence rate of $2k-1$ for a degree $k-1$ (i.e. a nominal $k$th local order of accuracy) polynomial construction in the time dimension was observed. 
According to~\cite{Huynh_JSC_2023}, in the time dimension the STFR scheme based on Gauss--Legendre points is equivalent to the IRK DG-Gauss scheme, which also has a superconvergence rate of $2k-1$ for a $k$ stage construction. In~\cite{Yu_Space_Time_Moving_2024}, we extended the STFR method for moving domain simulation with linear space-time grids, and discussed its connection with the dynamic grid formulation based on the method of lines. In this work, we developed a rigorous mathematical foundation for the nodal tensor-product STFR method for moving domain simulation with general curvilinear space-time grids.

The remainder of the paper is organized as follows. In \textbf{Sect.}~\ref{sec:num_method}, the basic idea of the nodal tensor-product STFR method on general curvilinear moving grids is explained. In \textbf{Sect.}~\ref{sec:GCL}, the connection between the general curvilinear space-time element representation of moving/deformable grids and GCL is discussed. It reveals the space-time numerical resolution requirement to enforce GCL automatically. Following that, requirements on the solution and flux approximation for curvilinear space-time grids are discussed in \textbf{Sect.}~\ref{sec:SFA}. Therein, different solution and flux approximation methods are discussed, and a new prospective is presented on how to design numerical schemes that satisfy the discrete GCL on general curvilinear space-time grids. Space-time projection-based polynomial filtering is then introduced to control aliasing errors associated with the  STFR schemes. In \textbf{Sect.}~\ref{sec:results}, numerical experiments are conducted to examine key numerical features of the nodal tensor-product STFR method for moving domain simulation with curvilinear space-time grids. Finally, \textbf{Sect.}~\ref{sec:con_FW} concludes this study, and points out potential future research directions.

\section{Numerical Methods}
\label{sec:num_method}
Consider the general hyperbolic conservation law in conservation form,
\begin{equation} \label{eq:GovEq}
   \frac{\partial \vect{Q}}{\partial t} + \nabla \cdot \mathbf{F} (\vect{Q}) = 0,
\end{equation}
defined on $\Omega \times [0,T)$ with the spatial domain $\Omega$ bounded by $\partial \Omega$, and $T$ as the physical time covered by numerical simulation. 
Herein,
$\vect{Q}$ is the vector of conserved variables, $\mathbf{F}$ is the spatial flux tensor, and $\nabla$ is the gradient operator in the physical domain. For a $d$-dimensional spatial system with $d=1, 2, \text{or} \ 3$, $\vect{Q}$ is represented as $\vect{Q} = \left(Q_1, \ldots, Q_{N_v} \right)$, where $N_v$ is the number of solution variables (i.e. conserved variables $Q_i$, $i=1,\ldots,N_v$, in this study).
The flux tensor is written as 
$ \mathbf{F} = \left[\vect{F}_1 \ldots \ \vect{F}_d \right]$, where $\vect{F}_i$, $i=1, 
\ldots, d$, is the spatial flux in the $i$-th dimension.
%
%
%
Let $\vect{x}=(x_1, \ldots, x_d)$ be the spatial coordinates. 
Now we introduce the space-time domain $\Omega^{\text{st}}=\Omega \times [0,T)$ with the boundary $\partial \Omega^{\text{st}}$,  and the space-time gradient operator $\nabla^{\text{st}}=\left(
\partial_t, \partial_{x_1}, \ldots, \partial_{x_d} 
\right)$. Eq.~\eqref{eq:GovEq} can then be written in the space-time domain as
\begin{equation} \label{eq:GovEq_ST}
     \nabla^{\text{st}} \cdot \mathbf{F}^{\text{st}} (\vect{Q}) = 0,
\end{equation}
where $\mathbf{F}^{\text{st}}$ is the space-time flux tensor, defined as
$ \mathbf{F}^{\text{st}} = \left[\vect{Q} \ \vect{F}_1 \ldots \ \vect{F}_d \right]$. Note that as a naming convention, bold non-italic symbols, such as $\mathbf{F}^{\text{st}}$, denote tensors of dimension greater than one, bold italic symbols, such as $\vect{Q}$, denote vectors, and italic non-bold symbols, such as $Q_i$, denote scalars.

To numerically solve Eq.~\eqref{eq:GovEq_ST}, the space-time domain $\Omega^{\text{st}}$ can be divided into $N_e^{\text{st}}$ non-overlapping space-time elements $\Omega_e^{\text{st}}$, $e=1,  \ldots, N_e^{\text{st}}$, with the element boundary denoted as $\partial \Omega_e^{\text{st}}$. 
Two approaches, namely the so-called physical domain approach and reference domain approach (see a complete discussion in \textbf{Sect.}~\ref{subsec:Implementation}), can then been used to implement Eq.~\eqref{eq:GovEq_ST}. 
Next we use the reference domain approach to further explain the FR formulation for three-dimensional (3D) tensor-product space-time elements with two dimensions in space and one dimension in time. 

\subsection{Tensor-product space-time FR} \label{subsubsec:STFR}
Consider the hyperbolic conservation law in 2D Cartesian space with $\vect{F}_1 = \vect{F}$ and $\vect{F}_2 = \vect{G}$, 
\begin{equation} \label{eq:GovEq_2D}
    \frac{\partial \vect{Q}}{\partial t} + \frac{\partial \vect{F}}{\partial x} + \frac{\partial \vect{G}}{\partial y} = 0,
\end{equation}
and transform the governing equation from the physical domain $(t,x,y)$ to the reference (computational) domain $(\tau,\xi,\eta)$ with $\tau, \xi,\eta \in [-1,1]$ as
\begin{equation} \label{eq:GovEq_2D_Comp}
    \frac{\partial \widetilde{\vect{Q}}}{\partial \tau} + 
    \frac{\partial \widetilde{\vect{F}}}{\partial \xi} 
    +
    \frac{\partial \widetilde{\vect{G}}}{\partial \eta} = 0,
\end{equation}
with
\begin{equation}
\Biggl\{
\begin{array}{lcl}
\widetilde{\vect{Q}} = |J| \tau_t \vect{Q} + |J| \tau_x \vect{F} + |J| \tau_y \vect{G} \\
\widetilde{\vect{F}} = |J| \xi_t \vect{Q} + |J| \xi_x \vect{F} + |J| \xi_y \vect{G} \\
\widetilde{\vect{G}} = |J| \eta_t \vect{Q} + |J| \eta_x \vect{F} + |J| \eta_y \vect{G}
\end{array}. 
\label{eq:ST_Var_ref}
\end{equation}
Herein, $J$ is the Jacobian matrix for the space-time coordinate transformation, which takes the form:
\begin{eqnarray}
J=\frac{\partial(t,x,y)}{\partial(\tau,\xi,\eta)} =
\left (
\begin{array}{ccc}
      t_\tau & t_\xi & t_\eta  \\
      x_\tau & x_\xi & x_\eta \\
      y_\tau & y_\xi & y_\eta 
\end{array} \right ).
\label{eq:Jacob}
\end{eqnarray}
The so-called Jacobian $|J|$ is the determinant of the Jacobian matrix $J$.
The inverse transformation of Eq.~\eqref{eq:Jacob} must exist for a non-singular transformation, which can be related to the Jacobian matrix $J$ as:
\begin{eqnarray}
J^{-1}=\frac{\partial(\tau,\xi,\eta)}{\partial(t,x,y)} =
\left (
\begin{array}{ccc}
      \tau_t & \tau_x & \tau_y  \\
      \xi_t  & \xi_x  & \xi_y \\
      \eta_t & \eta_x & \eta_y 
\end{array} \right ).
\label{eq:Jacob_inv}
\end{eqnarray}
The metrics $t_\tau$, $t_\xi$, $t_\eta$, $x_\tau$, $x_\xi$, $x_\eta$, $y_\tau$, $y_\xi$, and $y_\eta$ within each space-time element $\Omega_e^{\text{st}}$ can be calculated when a specific coordinate transformation $t(\tau,\xi,\eta)$, $x(\tau,\xi,\eta)$, and $y(\tau,\xi,\eta)$ is given, and $\tau_t$, $\tau_x$, $\tau_y$, $\xi_t$, $\xi_x$, $\xi_y$, $\eta_t$, $\eta_x$, and $\eta_y$ can then be determined by evaluating $J^{-1}$ correspondingly. 
Two examples of high-order curvilinear tensor-product space-time elements for grids with a general motion and the associated coordinate transformations are presented in Figure~\ref{fig:ST_Elements}. Note that the curvilinear space-time elements are represented with polynomials along each dimension. To facilitate implementation, it is required that the polynomial degrees along each spatial dimension are the same, and they can be different from the polynomial degree along the temporal dimension. Since the geometric features of the space-time elements are always known (e.g. grid locations at any time are known when a prescribed motion is used  or grid motion needs to be specified and iterated when two-way coupled fluid-structure interaction is performed), Gauss-Lobatto points along each dimension are used to describe the geometry of the high-order curvilinear space-time elements. 

\begin{figure}[!htbp]
  \centering
  \subfloat[2D Space-Time Element]{\includegraphics[width=0.45\textwidth]{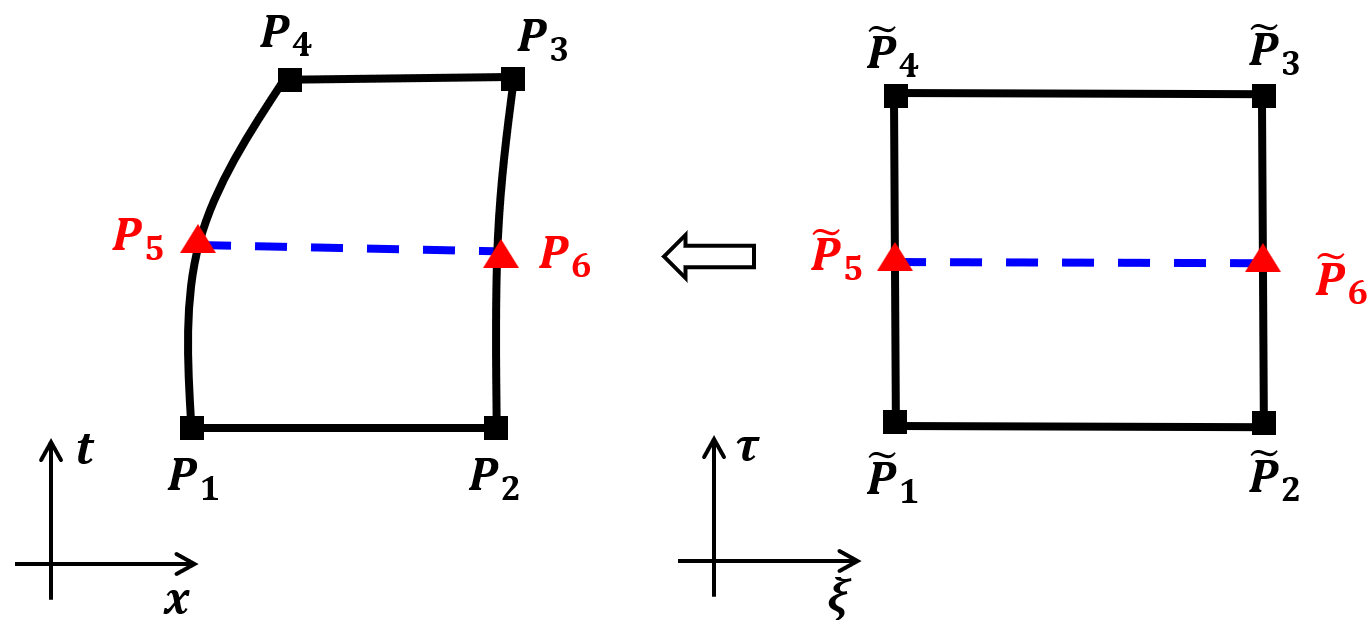}
  \label{fig:ST_Elements_2D}
  }\hspace{0.2em}
\hspace{0.2em}
  \subfloat[3D Space-Time Element]{\includegraphics[width=0.45\textwidth]{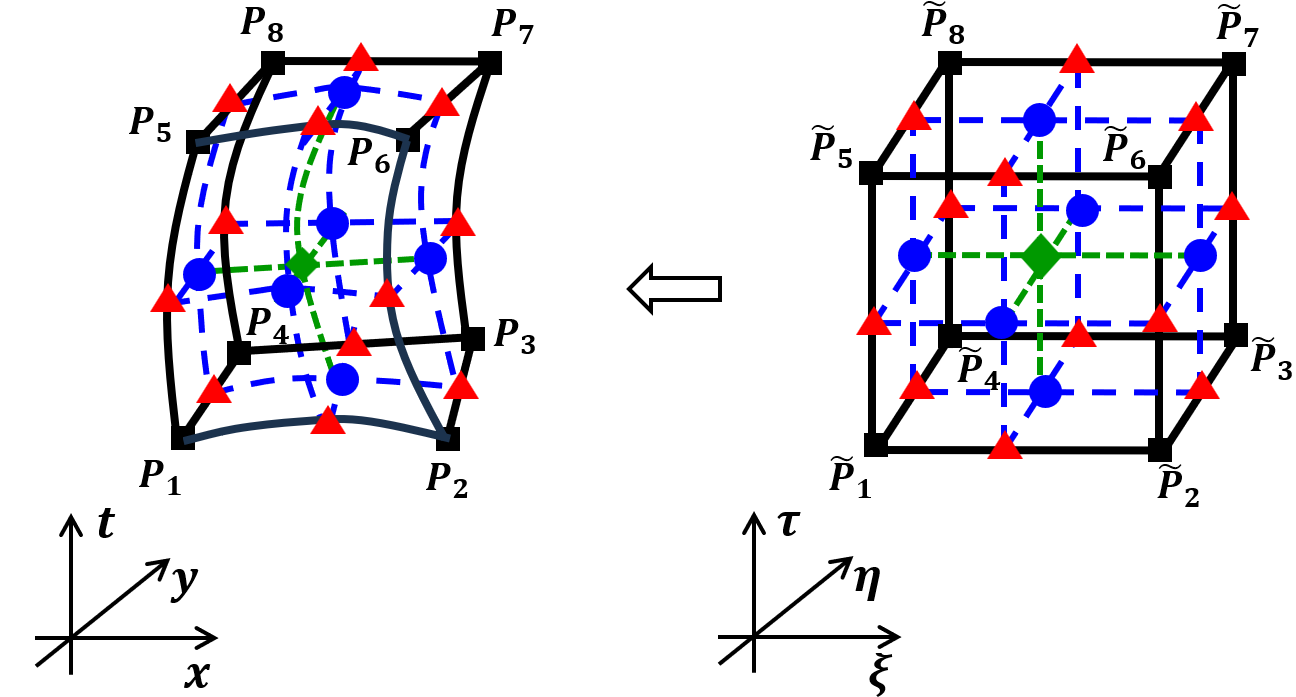}}
    \caption{High-order curvilinear (a) 2D and (b) 3D tensor-product space-time elements. For the 2D space-time element in (a), the $P^1$ polynomial is used to construct the 1D spatial element, and the $P^2$ polynomial is used to construct the curve along the time dimension. For the 3D space-time element in (b), $P^2$ polynomials are used to construct the spatial and temporal curves along each dimension. Note that the edge, surface, and volume points for the 3D space-time element are not numbered in (b).}
  \label{fig:ST_Elements}
\end{figure}

In space-time FR, each flux vector within a space-time element $\Omega_e^{\text{st}}$ is decomposed into a local component $(*)^{\text{loc}}$, constructed from the solution inside the element, and a correction component $(*)^{\text{cor}}$, formed using information from the element interfaces. The correction fluxes account for the discrepancies between the local fluxes and the common (numerical) fluxes imposed on the element boundaries. 
The total space-time flux is thus given by
\begin{equation} 
\Biggl\{
\begin{array}{lcl}
\widetilde{\vect{Q}} (\tau,\xi,\eta) = \widetilde{\vect{Q}}^{\text{loc}} (\tau,\xi,\eta) +
\widetilde{\vect{Q}}^{\text{cor}} (\tau,\xi,\eta)\\
\widetilde{\vect{F}} (\tau,\xi,\eta) =  
\widetilde{\vect{F}}^{\text{loc}} (\tau,\xi,\eta) +
\widetilde{\vect{F}}^{\text{cor}} (\tau,\xi,\eta) \\
\widetilde{\vect{G}} (\tau,\xi,\eta) = 
\widetilde{\vect{G}}^{\text{loc}} (\tau,\xi,\eta) +
\widetilde{\vect{G}}^{\text{cor}} (\tau,\xi,\eta)
\end{array}. 
\label{eq:flux_total}
\end{equation}
For a tensor-product space consisting of $(\tau,\xi,\eta)$, we have
\begin{equation}
\begin{array}{l}
\widetilde{\vect{Q}}^{\text{cor}}  (\tau, \xi, \eta) =
\left(
    \widetilde{\vect{Q}}_L^{\text{num}}  (\xi, \eta)
    - \widetilde{\vect{Q}}^{\text{loc}} (-1, \xi, \eta)
\right) g_L (\tau) +
\left(
    \widetilde{\vect{Q}}_R^{\text{num}}  (\xi, \eta)
    - \widetilde{\vect{Q}}^{\text{loc}} (1, \xi, \eta) 
\right) g_R (\tau), \\
\widetilde{\vect{F}}^{\text{cor}}  (\tau, \xi, \eta) =
\left(
    \widetilde{\vect{F}}_L^{\text{num}}  (\tau, \eta)
    - \widetilde{\vect{F}}^{\text{loc}}  (\tau, -1, \eta)
\right) g_L (\xi) +
\left(
    \widetilde{\vect{F}}_R^{\text{num}}  (\tau, \eta)
    - \widetilde{\vect{F}}^{\text{loc}}  (\tau, 1, \eta)
\right) g_R (\xi),\\
\widetilde{\vect{G}}^{\text{cor}}  (\tau, \xi, \eta) =
\left(
    \widetilde{\vect{G}}_L^{\text{num}}  (\tau, \xi)
    - \widetilde{\vect{G}}^{\text{loc}}  (\tau, \xi, -1)
\right)  g_L (\eta) +
\left(
    \widetilde{\vect{G}}_R^{\text{num}}  (\tau, \xi)
    - \widetilde{\vect{G}}^{\text{loc}}  (\tau, \xi, 1)
 \right)  g_R (\eta).
\end{array}
\label{eq:flux_cor}
\end{equation}
Herein, 
$g_{\rm L/R} (*)$ are the correction functions, which are required to satisfy the following conditions:
$g_L(-1) = g_R(1) = 1$ and $g_L(1) = g_R(-1) = 0$.
It is clear that the correction functions map the differences between the numerical fluxes $(*)^{\text{num}}$  and local fluxes $(*)^{\text{loc}}$ on the space-time element boundaries to the entire element. 
Different correction functions can recover different numerical schemes~\cite{huynh2007}. For example,
the right and left Radau polynomials are used as correction functions to recover the DG scheme in this study.
In Eq.~\eqref{eq:flux_cor},
$\widetilde{\vect{Q}}_{L/R}^{\text{num}}$, $\widetilde{\vect{F}}_{L/R}^{\text{num}}$, and $\widetilde{\vect{G}}_{L/R}^{\text{num}}$  are the numerical fluxes on the left or right boundaries of the standard space-time element in the $\tau$-, $\xi$- and $\eta$-direction, respectively. They are evaluated from the common normal fluxes $\vect{F}_n^{\text{st,com}}$ on the corresponding element surfaces in the physical domain as follows:
\begin{equation} \label{eq:ST_Flux_Com}
\Biggl \{
\begin{array}{l}
\widetilde{\vect{Q}}^{\text{num}} = |J| \norm{\nabla^{\text{st}} \tau}_2 \ \text{sign} (\vect{n}^{\text{st}} \cdot \nabla^{\text{st}} \tau) \vect{F}_n^{\text{st,com}} \\
\widetilde{\vect{F}}^{\text{num}} = |J| \norm{\nabla^{\text{\text{st}}} \xi}_2 \ \text{sign} (\vect{n}^{\text{st}} \cdot \nabla^{\text{st}} \xi) \vect{F}_n^{\text{st,com}} \\
\widetilde{\vect{G}}^{\text{num}} = |J| \norm{\nabla^{\text{st}} \eta}_2 \ \text{sign} (\vect{n}^{\text{st}} \cdot \nabla^{\text{st}} \eta) \vect{F}_n^{\text{st,com}}
\end{array},
\end{equation}
where the function $\text{sign} (*)$ takes the sign of the input, and $\norm{*}_2$ takes the $\ell^2$-norm of a vector input. 
Note that the normal space-time flux vector $\vect{F}_n^{\text{st}}$ is defined as $\vect{F}_n^{\text{st}} = \vect{n}^{\text{st}} \cdot \mathbf{F}^{\text{st}}$,
where $\vect{n}^{\text{st}}$ is the outward-going unit normal vector of the space-time boundary $\partial \Omega_e^{\text{st}}$ for any space-time element $\Omega_e^{\text{st}}$. On any surface of the space-time element, there exist two sets of values of $\vect{F}_n^{\text{st}}$ constructed separately from the two elements (including ghost elements over the domain boundaries) sharing the surface. Therefore, common normal fluxes $\vect{F}_n^{\text{st,com}}$ are calculated from the two sets of of values of $\vect{F}_n^{\text{st}}$ to enforce the mass, momentum, and energy conservation laws at the element boundaries.
In this study, the common inviscid fluxes at the element interfaces in the physical domain are calculated using the local Lax-Friedrichs (or Rusanov) approximate Riemann solver~\cite{RUSANOV1962}. 

%
After substituting Eq.~\eqref{eq:flux_total} into Eq.~\eqref{eq:GovEq_2D_Comp}, the governing equations become
\begin{equation}
\frac{\partial \widetilde{\vect{Q}}^{\text{loc}}}{\partial \tau} +
\frac{\partial
\widetilde{\vect{F}}^{\text{loc}}}{\partial \xi} +
\frac{\partial
\widetilde{\vect{G}}^{\text{loc}}}{\partial \eta}+ 
\widetilde{\vect{\delta}}^{\text{cor,st}} = 0,
\label{eq:GovEq_2D_Comp_FR}
\end{equation}
where
\begin{equation}
\widetilde{\vect{\delta}}^{\text{cor,st}} =
\frac{\partial \widetilde{\vect{Q}}^{\text{cor}}}{\partial \tau} +
\frac{\partial
\widetilde{\vect{F}}^{\text{cor}}}{\partial \xi} +
\frac{\partial
\widetilde{\vect{G}}^{\text{cor}}}{\partial \eta}.
\label{eq:Correction_Field_comp}
\end{equation}
Note that $\widetilde{\vect{\delta}}^{\text{cor}}$ is called the correction field in standard FR, and with the superscript `st', it is the space-time correction field. $\widetilde{\vect{\delta}}^{\text{cor,st}}$ can be directly evaluated with the derivative of the correction functions, that is, $g^{\prime}_{L/R} (*)$, due to the special structure of the correction fluxes $(*)^{\text{cor}}$ shown in Eq.~\eqref{eq:flux_cor}. 
%

Typical distributions of solution and flux points in a standard space-time element used to formulate STFR are shown in Figure~\ref{fig:ST_SF}. Therein, Gauss--Legendre points are used as solution points in each dimension, and flux points are defined as the intersections of the element boundaries with the coordinate-aligned lines drawn from each solution point. This should be distinguished from the Gauss-Lobatto points used to represent the curvilinear geometry of the space-time element displayed in Figure~\ref{fig:ST_Elements}. We mention that the accuracy features of different treatments of solutions and fluxes can vary much, especially for curvilinear elements.  A formal discussion of the polynomial approximation of solutions and fluxes, and their order of accuracy measure will be presented in \textbf{Sect.}~\ref{sec:SFA}.

\begin{figure}
  \centering
  \subfloat[2D Space-Time Solution and Flux Points]{\includegraphics[width=0.35\textwidth]{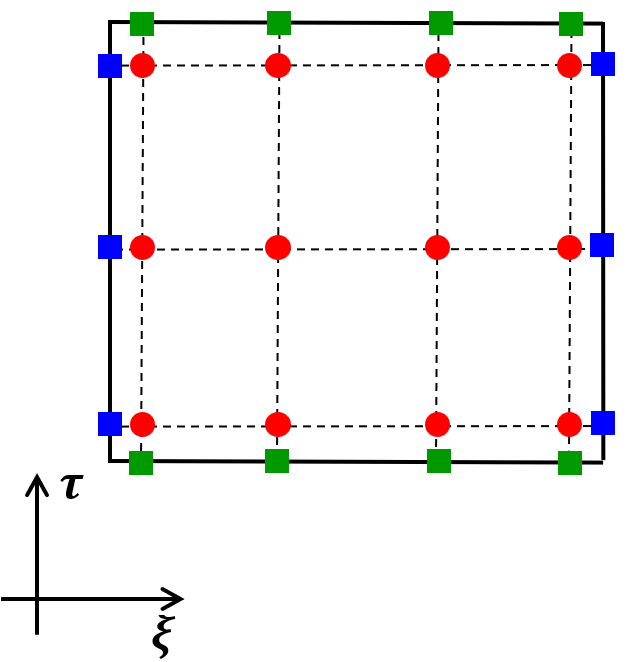}
  \label{fig:ST_SF_2D}
  }\hspace{0.2em}
\hspace{0.2em}
  \subfloat[3D Space-Time Solution and Flux Points]{\includegraphics[width=0.35\textwidth]{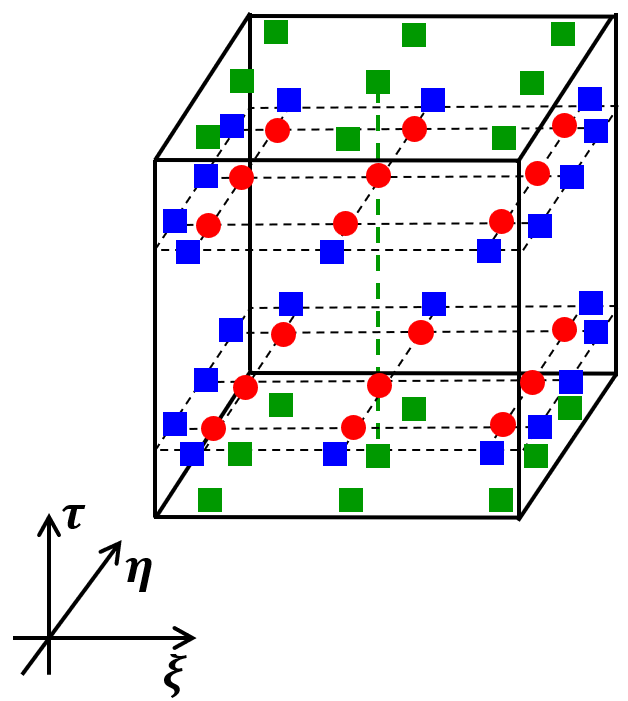}}
    \caption{Distribution of solution points (circle) and flux points (square) in (a) 2D and (b) 3D tensor-product space-time elements. Gauss-Legendre points are used in each dimension to construct solution and flux polynomials. For the 2D space-time element in (a), the $P^3$ polynomial is used to construct the solution along the spatial dimension, and the $P^2$ polynomial is used to construct the solution along the time dimension. For the 3D space-time element in (b), $P^2$ polynomials are used to construct the solution along each (i.e., $\xi$- and $\eta$-) spatial dimension, and the $P^1$ polynomial is used to construct the solution along the time dimension.}
  \label{fig:ST_SF}
\end{figure}



 
To solve the STFR formulation~\eqref{eq:GovEq_2D_Comp_FR}, we augment it with a pseudo-time derivative term $\partial \widetilde{\vect{Q}} / \partial \widetilde{t}$, where $\widetilde{t}$ is the pseudo-time. Define the residual 
vector $\widetilde{\vect{R}}(\widetilde{\vect{Q}})$ as the oppostite of the left-hand side of Eq.~\eqref{eq:GovEq_2D_Comp_FR}.
Thus, Eq.~\eqref{eq:GovEq_2D_Comp_FR} can be written as
\begin{equation}
\frac{\partial \widetilde{\vect{Q}}}{\partial \widetilde{t}} =\widetilde{\vect{R}}(\widetilde{\vect{Q}}).
\label{eq:GovEq_Semi}
\end{equation}

Both explicit and implicit pseudo-time marching methods can be used to solve Eq.~\eqref{eq:GovEq_Semi}. In this study, local pseudo-time stepping with the explicit two-stage, second-order strong stability preserving Runge-Kutta (SSPRK) method~\cite{Gottlieb_EtAl_SIAM_2001} is used to accelerate convergence. We mention that implicit pseudo-time marching methods~\cite{Wang_Yu_JSC_2020}, such as the backward Euler and second-order backward differentiation formula
(BDF2), can also be used to solve Eq.~\eqref{eq:GovEq_Semi} when the physical time step is sufficiently large. However, this is beyond the scope of the current work.


\section{Discussions of Moving Grid Simulations} \label{sec:GCL}
In the space-time coordinate transformation from Eq.~\eqref{eq:GovEq_2D} to Eq.~\eqref{eq:GovEq_2D_Comp}, the following geometric conservation laws are used:
\begin{equation}
\frac{\partial }{\partial \tau}
\left (
\begin{array}{c}
      |J| \tau_t \\
      |J| \tau_x \\
      |J| \tau_y
\end{array} \right )  +
\frac{\partial }{\partial \xi}
\left (
\begin{array}{c}
      |J| \xi_t \\
      |J| \xi_x \\
      |J| \xi_y
\end{array} \right )  +
\frac{\partial }{\partial \eta}
\left (
\begin{array}{c}
      |J| \eta_t \\
      |J| \eta_x \\
      |J| \eta_y
\end{array} \right ) = 0.
\label{eq:GCL}
\end{equation}

\subsection{Method of lines}
When the method of lines is used for moving grid simulation, the metrics $\tau_x$ and $\tau_y$ are always zeros, i.e. $\tau (t)$ is only a function of $t$. Therefore, only the first equation of the GCL~\eqref{eq:GCL} is time-dependent. Moreover, the grid velocity $\vect{v}_g \equiv d \vect{x}/d t$ can be expressed with the metrics $\tau_t$, $x_\tau$, and $y_\tau$ through $(v_{g,x}, v_{g,y}) = (x_\tau \tau_t, y_\tau \tau_t)$, where $v_{g,x}$ and $v_{g,y}$ are the $x$- and $y$-component of $\vect{v}_g$, respectively. Therefore, the metrics $\xi_t$ and $\eta_t$ in Eq.~\eqref{eq:GCL} can be calculated from the grid velocity as:
\begin{equation}
\biggl\{
\begin{array}{l}
\xi_t = - \vect{v}_g \cdot \nabla \xi\\
\eta_t = - \vect{v}_g \cdot \nabla \eta
\end{array}.
\end{equation}

Since the temporal discretization of $|J| \tau_t$  is usually not consistent with the spatial discretization of $|J| \xi_t$ and $|J| \eta_t$ in the first equation of the GCL~\eqref{eq:GCL}, the so-called discrete GCL errors can contaminate moving grid simulation. In general, there are two common approaches to control GCL errors in high-order numerical simulations with the method of lines. In the work by Persson \textit{et al.}~\cite{PerssonEtAl_CMAME_2009}, the first GCL formula in Eq.~\eqref{eq:GCL} was solved using the same numerical time integration scheme as that for the physical laws~\eqref{eq:GovEq_2D} or~\eqref{eq:GovEq_2D_Comp}. The analytical $|J|$ in $\widetilde{\vect{Q}} $ is then substituted with the numerical $|J|$ calculated from the first GCL to cancel the GCL error. Another approach to eliminate the GCL error is to replace the term $\partial (|J| \tau_t) / \partial \tau$ in Eq.~\eqref{eq:GovEq_2D_Comp} directly with $- \left( \partial (|J| \xi_t) / \partial \xi + \partial (|J| \eta_t) / \partial \eta \right)$ in the scheme implementation. As a result, all GCL-related terms only have errors related to spatial discretization, which is high-order and consistent between relevant terms. A local source term $-\vect{v}_g \cdot \nabla \vect{Q}$ is thus added to the governing equation to cancel the GCL error~\cite{Yu2011,Yu2016}.

To summarize, when the method of lines is used for moving grid simulation, the grid velocity needs to be calculated, and the time-dependent GCL needs to be enforced to ensure accuracy.

\subsection{Space-time method} \label{subsec:GCL_ST}
In the space-time method, the GCL~\eqref{eq:GCL} is automatically enforced, to the level of the numerical resolution, by the geometry of the space-time elements as shown in Figure~\ref{fig:ST_Elements}. 
As will be discussed in \textbf{Sect.}~\ref{sec:SFA}, GCL can be either enforced by using special metric construction strategies~\cite{Kopriva_JSC_2006,AbeEtAl_JCP_2015} or satisfied by embedding them exactly into the numerical schemes (see \textbf{Sect.}~\ref{subsec:new_FR}).

A special attention should be paid to the metrics $t_{\xi}$ and $t_{\eta}$ when using the tensor-product space-time curvilinear elements. In general, they are function of space and time; this indicates that the spatial grid points are allowed to move freely in the time dimension (see the space-time unstructured grid setup in~\cite{Behr_IJNMF_2008,Nishikawa_Padway_AIAAAva_2020,FRONTIN_EtAl_2021_ANM}). Here we use the concept of space-time slabs $[t_j, t_j + \Delta t_j]$ with piece-wise constant time steps $\Delta t_j$, and all spatial grid points are only allowed to march with the same distance in the temporal direction within any space-time curvilinear element. As a result, the physical time $t$ is only a (nonlinear) function of the reference time $\tau$. The proof is straightforward as for a full set of Lagrange polynomials $\phi_i (\xi)$, $i=1, \ldots, i_{max}$, we have $ \sum_{i=1}^{i_{max}} \phi_i (\xi) \equiv 1 $. Therefore, $\widetilde{\vect{Q}}$ in Eq.~\eqref{eq:ST_Var_ref} is expressed as $\widetilde{\vect{Q}} = |J| \tau_t \vect{Q}$. The inverse of the Jacobian matrix $J^{-1}$ (2D space plus time) in Eq.~\eqref{eq:Jacob_inv} can then be calculated from the so-called non-conservative formulation:
\begin{eqnarray}
|J| J^{-1} =
|J| \left (
\begin{array}{ccc}
      \tau_t & \tau_x & \tau_y  \\
      \xi_t  & \xi_x  & \xi_y \\
      \eta_t & \eta_x & \eta_y 
\end{array} \right )
=
\left (
\begin{array}{ccc}
      x_\xi y_\eta - x_\eta y_\xi & 0 & 0  \\
      -x_\tau y_\eta + y_\tau x_\eta  & t_\tau y_\eta  & -t_\tau x_\eta \\
      x_\tau y_\xi - y_\tau x_\xi  & -t_\tau y_\xi & t_\tau x_\xi 
\end{array} \right ).
\label{eq:Jacob_inv_2D}
\end{eqnarray}

We then have the following conclusion:
\begin{lemma}
The GCL~\eqref{eq:GCL} is automatically satisfied numerically when the geometric features of a curvilinear element can be completely represented by polynomials in a numerical scheme.
\label{lem:GCL_num}
\end{lemma}

\begin{proof}
Take the second formula, i.e.
\[
\frac{\partial }{\partial \tau}
\left(|J| \tau_x \right) +
\frac{\partial }{\partial \xi}
\left(|J| \xi_x \right) +
\frac{\partial }{\partial \eta}
\left(|J| \eta_x  \right) = 0
\]
as an example.
Similar approaches can be used to prove other formulas in Eq.~\eqref{eq:GCL}.

In the space-time slab setup, we have $\tau_x \equiv 0$. The coordinates $x(\xi,\eta,\tau)$, $y(\xi,\eta,\tau)$, and $ t(\xi,\eta,\tau) \equiv t(\tau)$ of a curvilinear element can, in general, be written as
\begin{equation*}
\biggl\{
\begin{array}{l}
x(\xi,\eta,\tau) = 
\sum_{i_1} \sum_{i_2} \sum_{i_3} L_{i_1,i_2,i_3}(\xi,\eta,\tau) x_{i_1,i_2,i_3} = 
\sum_{i_1}  \sum_{i_2} \sum_{i_3} \phi_{i_1}(\xi) \phi_{i_2}(\eta) \phi_{i_3}(\tau)
x_{i_1,i_2,i_3}\\
y(\xi,\eta,\tau) = 
\sum_{i_1} \sum_{i_2} \sum_{i_3} L_{i_1,i_2,i_3}(\xi,\eta,\tau) y_{i_1,i_2,i_3} = 
\sum_{i_1}  \sum_{i_2} \sum_{i_3} \phi_{i_1}(\xi) \phi_{i_2}(\eta) \phi_{i_3}(\tau) y_{i_1,i_2,i_3}\\
t(\tau) = \sum_{i_3} \phi_{i_3}(\tau) t_{i_3} 
\end{array},
\end{equation*}
where $\phi_i (*)$, $* = \xi, \eta, \text{or}~\tau \in [-1,1]$ are 1D Lagrange polynomials, and $L_{i_1,i_2,i_3}(\xi,\eta,\tau)$ is their tensor product.

Since the coordinates are fully represented by polynomials, the term $\partial (|J| \xi_x)/\partial \xi$ can be calculated analytically from
\begin{equation*}
\frac{\partial |J| \xi_x}{\partial \xi} 
\stackrel{t_{\xi *} \equiv 0}{=}  t_\tau y_{\xi \eta}
= \left(\sum_{i_3} \frac{d \phi_{i_3} (\tau)}{d \tau} t_{i_3} \right)
  \left(\sum_{i_1} \sum_{i_2} \sum_{i_3} 
  \frac{d \phi_{i_1} (\xi)}{d \xi}
  \frac{d \phi_{i_2} (\eta)}{d \eta}
  \phi_{i_3} (\tau)
   y_{i_1,i_2,i_3} \right).
\end{equation*}
Following the same procedure, it can be shown that $\partial (|J| \eta_x)/\partial \eta$  is exactly $- \partial (|J| \xi_x)/\partial \xi$. This ensures that the second formula in Eq.~\eqref{eq:GCL} always holds numerically. 
\end{proof}

\section{Discussions of Solution and Flux Approximation} \label{sec:SFA}

In this section, we first introduce two popular ways to construct solution and flux polynomials for curvilinear elements in static domains, and discuss their numerical features. To facilitate discussion, only the spatial dimensions are considered. Based on that understanding, we then 
develop a new approach to approximate solutions and fluxes as polynomials over general curvilinear space-time elements.

\subsection{Two popular solution and flux approximation approaches} \label{subsec:Implementation}

\subsubsection{Physical domain approach} \label{subsubsec:phys_dom}
Divide the spatial domain $\Omega$ into $N_e$ non-overlapping elements $\Omega_e$, $e=1,  \ldots, N_e$, with the element boundary denoted as $\partial \Omega_e$. In many DG-based constructions~\cite{YOU_EtAl_CPC_23}, each component $Q_i$, $i=1,\ldots,N_v$, of the solution vector $\vect{Q}$ is approximated as a polynomial $Q_{h,i} (\vect{x},t)$, which sits in the degree $k$ polynomial space $P^k (\Omega_e)$ based on the physical element $\Omega_e$.
Let $w_{j}(\vect{x})$, $j=1, \ldots, N(k;\Omega_e)$, denote the orthogonal bases of the polynomial space.   
We have
\begin{equation*}
Q_{h,i} (\vect{x},t) = \sum_j^{N(k)} q_{h,i,j}(t) w_{j}(\vect{x}),
\end{equation*}
where $N(k)$ depends on the element type and the dimension of the polynomial space. For example, for a tetrahedral element, $N(k) = (k+1)(k+2)(k+3)/6$; and for a $d$-dimensional tensor-product element, $N(k) = (k+1)^d$. 
The strong form of Eq.~\eqref{eq:GovEq} for each spatial element $\Omega_e$ bounded by $\partial \Omega_e$ finally becomes
\begin{equation} \label{eq:WRes_Strong}
 \int_{\Omega_e} \frac{\partial \vect{Q_h}}{\partial t} \otimes \vect{w} dV  + 
 \int_{\Omega_e}  \nabla \cdot \mathbf{F}^{\text{loc}} \otimes \vect{w} dV + 
 \int_{\partial \Omega_e} \left( \vect{F}_n^{\text{com}} - \vect{F}_n^{\text{loc}} \right) \otimes \vect{w} dS = 0.
\end{equation}
Herein, $\vect{Q_h}$ is the solution vector consisting of all $Q_{h,i}(\vect{x},t)$, $\vect{w}$ is a vector consisting of all bases $w_{j}(\vect{x})$, $\vect{F}_n^{\text{loc}}$ and $\vect{F}_n^{\text{com}}$ are the local and common normal fluxes with $\vect{F}_n = \vect{n} \cdot \mathbf{F}$, where $\vect{n}$ is the outward-going unit normal vector of the physical element boundary $\partial \Omega_e$, and $\otimes$ is the tensor product operator.
Numerical quadrature can then be performed to calculate each component of the mass matrix, and to carry out the volumetric and surface integration for flux-related terms. Over-integration can be used to enhance the integration accuracy by reducing aliasing errors from the flux-related terms.

To facilitate implementation, $Q_i$ can be approximated as a polynomial $Q_{h,i}^r (\vect{\xi},t) \in P^k (\Omega^r)$, where $\Omega^r$ is the standard element in the reference domain, and $\vect{\xi} = (\xi_1, \ldots, \xi_d)$, e.g. $\vect{\xi} = (\xi,\eta)$ for the two-dimensional governing equations. Since the mapping between the reference and physical domains is one-to-one within any physical element $\Omega_e$ (see Figure~\ref{fig:ST_Elements}), the inverse map exists and $\vect{\xi} = \vect{\xi}(\vect{x})$.  Analogous to the nomenclature used for the physical element, we can define the solution vector $\vect{Q_h}^r (\vect{\xi})$ for the reference domain, and the vector $\vect{w}^r (\vect{\xi})$ consisting of all bases of the polynomial space $P^k (\Omega^r)$. 
Thus, the integration of Eq.~\eqref{eq:GovEq} in the physical element $\Omega_e$ can be calculated from the reference element $\Omega^r$ as

\begin{equation} \label{eq:Transform}
 \begin{gathered}
 \int_{\Omega_e}  \left( 
 \frac{\partial \vect{Q_h}^{r}(\vect{\xi} (\vect{x}),t)}{\partial t} 
 + \nabla \cdot \mathbf{F}^{r}(\vect{\xi} (\vect{x}),t)
\right) \otimes \vect{w}^{r} (\vect{\xi} (\vect{x})) dV  = \\
 \int_{\Omega^r}  \left( 
 \frac{\partial \vect{Q_h}^r(\vect{\xi},t)}{\partial t}
 +\sum_{i=1}^d \nabla \xi_{i} \cdot \frac{\partial \mathbf{F}^r(\vect{\xi},t)}{\partial \xi_{i}}
 \right) \otimes \vect{w}^r(\vect{\xi}) |J|(\vect{\xi};\Omega_e) dV^r.
 \end{gathered}
\end{equation} 
Note that although the original integration concept (i.e. the left-hand side of the equation) is based on the physical domain, the implementation (i.e. the right-hand side of the equation), including the approximation of solutions and fluxes, Jacobian and metric calculation, and numerical integration, is based on the reference domain. 
Therefore, this approach can  be termed as a hybrid reference-physics domain approach, which will be further discussed and developed in following sections.

\begin{remark}
In the analytical governing equations, the value of $\vect{Q}^r(\vect{\xi},t)$ is equal to that of $\vect{Q}(\vect{x},t)$ at the point pair $\vect{\xi}$ and $\vect{x}$. However, the polynomial approximations $\vect{Q_h} (\vect{x},t)$ in the physical domain do not equal to the polynomial approximations $\vect{Q_h}^r (\vect{\xi},t)$ in the reference domain at the point pair $\vect{\xi}$ and $\vect{x}$ when the transformation between $\vect{x}$ and $\vect{\xi}$ is nonlinear.  
\end{remark}

\begin{remark}
Over-integration is usually needed to enhance the integration accuracy of the flux-related terms in Eq.~\eqref{eq:Transform} for curvilinear elements.  
\end{remark}

\subsubsection{Reference domain approach} \label{subsubsec:RefD}
Without loss of generality, we use the 2D governing equations to explain solution and flux approximations here. When the method of lines is used to solve Eq.~\eqref{eq:GovEq_2D} in a stationary domain, the solution and flux vectors take the following forms in the reference domain:
\begin{equation} \label{eq:Var_trans_2D}
\Biggl\{
\begin{array}{l}
\widetilde{\vect{Q}} = |J| \vect{Q} \\
\widetilde{\vect{F}} = |J| \xi_x \vect{F} + |J| \xi_y \vect{G} \\
\widetilde{\vect{G}} = |J| \eta_x \vect{F} + |J| \eta_y \vect{G}
\end{array}, 
\end{equation}
with $\tau \equiv t$, and the governing equations read like Eq.~\eqref{eq:GovEq_2D_Comp} but with different definitions of $\widetilde{\vect{Q}}$, $\widetilde{\vect{F}}$, and $\widetilde{\vect{G}}$. 

In the reference domain approach, each component $\widetilde{Q}_i$ of the solution vector $\widetilde{\vect{Q}}$ is approximated as a polynomial $\widetilde{Q}_{h,i} (\vect{\xi},t)$, which sits in the degree $k$ polynomial space $P^k (\Omega^r)$ based on the standard element $\Omega^r$.
Let $w_{j}^r(\vect{\xi})$, $j=1, \ldots, N(k;\Omega^r)$, denote the orthogonal bases of the polynomial space $P^k (\Omega^r)$.
As a result, $\widetilde{Q}_{h,i} (\vect{\xi},t)$ can be expressed as
\begin{equation*}
\widetilde{Q}_{h,i} (\vect{\xi},t) = \sum_j^{N(k)} \widetilde{q}_{h,i,j}(t) w_{j}^r(\vect{\xi}).
\end{equation*}

Let $\widetilde{\vect{Q}}_{\vect{h}}$ be the solution vector consisting of all $\widetilde{Q}_{h,i} (\vect{\xi},t)$, and $\vect{w}^r$ be the vector consisting of all bases $w_{j}^r(\vect{\xi})$. The DG scheme in the strong form for the governing equation on a stationary grid in the reference domain $\Omega^r$ bounded by $\partial \Omega^r$ can then be written as
\begin{equation} \label{eq:DG_ref}
 \int_{\Omega^r} \frac{\partial \widetilde{\vect{Q}}_{\vect{h}}}{\partial t} \otimes \vect{w}^r dV^r  + 
 \int_{\Omega^r}  \nabla^r \cdot \widetilde{\mathbf{F}}^{\text{loc}} \otimes \vect{w}^r dV^r + 
 \int_{\partial \Omega^r} \left( \widetilde{\vect{F}}_{n^r}^{\text{num}} - \widetilde{\vect{F}}_{n^r}^{\text{loc}} \right) \otimes \vect{w}^r d S^r = 0.
\end{equation}
Herein, $\nabla^r$ is the gradient operator in the reference domain, $\widetilde{\mathbf{F}} = \left(\widetilde{\vect{F}},  \widetilde{\vect{G}}\right)$ in 2D, 
and $\widetilde{\vect{F}}_{n^r} = \vect{n}^r \cdot \widetilde{\mathbf{F}}$, where $n^r$ is the outward-going unit normal vector of the standard element $\Omega^r$. 
%


The accuracy of the reference domain approach can suffer on curvilinear elements due to the associated geometric nonlinearity. 
Moreover, the freestream flow may not be preserved by the reference domain approach due to the violation of the GCL~\eqref{eq:GCL} in its discrete form.    
To restore the accuracy of the reference domain approach for curvilinear elements, several conservative metric construction methods have been developed; see~\cite{Kopriva_JSC_2006,AbeEtAl_JCP_2015}. 
If the non-conservative metric calculation (see Eq.~\eqref{eq:Jacob_inv_2D} as an example) is used, the coordinate transformation concept presented in Eq.~\eqref{eq:Transform} can be used to enhance the accuracy of the implementation of Eq.~\eqref{eq:DG_ref} on curvilinear elements. With the correction procedure via reconstruction (CPR) concept~\cite{wang2009}, Eq.~\eqref{eq:Transform} can be recast into the differential format for any physical element $\Omega_e$ as
\begin{equation} \label{eq:FR_Diff}
 \frac{\partial \vect{Q_h}^r(\vect{\xi},t)}{\partial t} +
 \nabla \cdot \mathbf{F}^{r,\text{loc}}(\vect{\xi},t) +
 \vect{\delta}^{r,\text{cor}}(\vect{\xi},t) = 0,
\end{equation}
with
\begin{equation*}
\int_{\Omega^r}   \vect{\delta}^{r,\text{cor}}(\vect{\xi},t) \otimes \vect{w}^r(\vect{\xi}) |J|(\vect{\xi};\Omega_e) d V^r
=
\int_{\partial \Omega^r} \left( \widetilde{\vect{F}}_{n^r}^{\text{num}}(\vect{\xi},t) - \widetilde{\vect{F}}_{n^r}^{\text{loc}}(\vect{\xi},t) \right) \otimes \vect{w}^r(\vect{\xi}) d S^r.
\end{equation*}
Herein, $|J|(\vect{\xi};\Omega_e) \vect{\delta}^{r,\text{cor}}$ becomes the spatial correction field in the FR formula  \eqref{eq:GovEq_2D_Comp_FR} in a stationary domain.

Note that the numerical implementation of Eq.~\eqref{eq:FR_Diff} can preserve the freestream flow on curvilinear meshes~\cite{YU201470}, but does not necessarily hold the discrete GCL. As a result, although Eq.~\eqref{eq:DG_ref} and Eq.~\eqref{eq:FR_Diff} are analytically identical to each other, they are numerically different. 

\subsection{A new perspective} \label{subsec:new_FR}
As discussed in \textbf{Sect.}~\ref{subsubsec:RefD}, when $\widetilde{\vect{Q}}$ is approximated as a polynomial, the accuracy of the solution $\vect{Q}$ can suffer due to the geometric nonlinearity of the curvilinear elements. The fix via CPR through Eq.~\eqref{eq:FR_Diff} does not necessarily hold discrete GCL, which is not desirable. Here we introduce a new perspective to approximate $\widetilde{\vect{Q}}$ as a polynomial by approximating $\vect{Q}$ and metrics as polynomials separately.  

\subsubsection{Stationary curvilinear elements} \label{subsubsec:Stat_CE}

Similar to \textbf{Sect.}~\ref{subsubsec:RefD}, only 2D governing equations on stationary grids and solved with the method of lines are used to explain the concept in this subsection. Moreover, the tensor product polynomial space $\mathbb{Q}(\Omega^r)$ defined on the standard element $\Omega^r$ will be considered in this study. Thus, we can explicitly write $\mathbb{Q}^k(\Omega^r)$ with the polynomial degree at most $k$ in each variable defined on the 2D standard quadrilateral element $\Omega^r$ as $\mathbb{Q}^k(\xi,\eta) = P^k(\xi) \otimes P^k(\eta)$. 

Let $Q_{h,i} (\vect{\xi};t) \in \mathbb{Q}^k (\xi, \eta)$ at a fixed time $t$ be the polynomial approximation of each component $Q_i$ of the solution vector $\vect{Q}$. For curvilinear elements, the physical domain coordinates $(x,y)$ can be represented as polynomials, i.e. $x,y \in \mathbb{Q}^l (\xi, \eta)$. For 2D coordinate transformation between the stationary physical domain $(x,y)$ and the reference domain $(\xi, \eta)$, we summarize the polynomial space where the Jacobian and metrics sit in Table~\ref{Tab:metrics}. 
Therefore, $\widetilde{Q}_{h,i} (\vect{\xi}) = |J| Q_{h,i} \in \mathbb{Q}^{k+2l-1} (\xi, \eta) $.

\begin{remark} \label{re:poly_degree_2D}
Note that the coefficients of the polynomials used to represent the Jacobian and metrics are known (see \textbf{Sect.}~\ref{subsec:GCL_ST}). Therefore, in the degree $k+2l-1$ polynomial approximation of $\widetilde{Q}_{h,i}$, only $k+1$ coefficients in each dimension are independent.
\end{remark}

\begin{table}[!htbp] 
\centering
\begin{tabular}{|c|c|c|c|}
\hline
                  & $|J|$ & $|J|\xi_x$, $|J|\xi_y$ & $|J|\eta_x$, $|J|\eta_y$ \\
\hline
PS &  $\mathbb{Q}^{2l-1} (\xi,\eta)$ & $P^l(\xi) \otimes P^{l-1}(\eta)$ & $P^{l-1}(\xi) \otimes P^l(\eta)$ \\ 
\hline
\end{tabular}
\caption{A list of polynomial space occupied by the Jacobian and metrics of a stationary curvilinear element with $x,y \in \mathbb{Q}^l (\xi, \eta)$. `PS' stands for polynomial space. Note that $l \geq 1$. When $l = 1$, there exists a special case when the element is a proper rectangle. Therein, the Jacobian $J$ and all metrics are constant. }
\label{Tab:metrics}
\end{table}

Based on Eq.~\eqref{eq:Var_trans_2D}, to match the polynomial degree of $\widetilde{Q}_{h,i}$ in the differential form of the governing equation, each component of the flux vector $\widetilde{\vect{F}}$ in the reference domain can be approximated as a polynomial sitting in the space $P^{k+2l}(\xi) \otimes P^{k+2l-1}(\eta)$, and each component of $\widetilde{\vect{G}}$ can be approximated as a polynomial sitting in the space $P^{k+2l-1}(\xi) \otimes P^{k+2l}(\eta)$. 
Thus, we have $\partial \widetilde{F}_i / \partial \xi \in \mathbb{Q}^{k+2l-1} (\xi, \eta)$, and $\partial \widetilde{G}_i / \partial \eta \in \mathbb{Q}^{k+2l-1} (\xi, \eta)$, where $\widetilde{F}_i$ and $\widetilde{G}_i$ are the $i$-th component of $\widetilde{\vect{F}}$ and $\widetilde{\vect{G}}$, respectively. 
Following the FR formulation~\eqref{eq:flux_total}, we can split $\widetilde{\vect{F}}$ and $\widetilde{\vect{G}}$ into local fluxes $\widetilde{\vect{F}}^{\text{loc}} (\xi,\eta)$, $\widetilde{\vect{G}}^{\text{loc}} (\xi,\eta)$, and correction fluxes $\widetilde{\vect{F}}^{\text{cor}} (\xi,\eta)$, $\widetilde{\vect{G}}^{\text{cor}} (\xi,\eta)$. Similar to Eq.~\eqref{eq:flux_cor}, the correction fluxes can be approximated with high-order correction functions as 
\begin{equation}
\Biggl\{
\begin{array}{l}
\widetilde{\vect{F}}^{\text{cor}}  (\xi, \eta) =
\left(
    \widetilde{\vect{F}}_L^{\text{num}}  (\eta)
    - \widetilde{\vect{F}}^{\text{loc}}  (-1, \eta)
\right) g_L (\xi) +
\left(
    \widetilde{\vect{F}}_R^{\text{num}}  (\eta)
    - \widetilde{\vect{F}}^{\text{loc}}  (1, \eta)
\right) g_R (\xi),\\
\widetilde{\vect{G}}^{\text{cor}}  (\xi, \eta) =
\left(
    \widetilde{\vect{G}}_L^{\text{num}}  (\xi)
    - \widetilde{\vect{G}}^{\text{loc}}  (\xi, -1)
\right)  g_L (\eta) +
\left(
    \widetilde{\vect{G}}_R^{\text{num}}  (\xi)
    - \widetilde{\vect{G}}^{\text{loc}}  (\xi, 1)
 \right)  g_R (\eta).
\end{array}
\label{eq:flux_cor_2D}
\end{equation}
Herein, $g_{L/R} (\xi) \in P^{k+2l}(\xi)$, and $g_{L/R} (\eta) \in P^{k+2l}(\eta)$. 

Now we examine polynomial space that can be used to approximate local fluxes $\widetilde{\vect{F}}^{\text{loc}} (\xi,\eta)$ and $\widetilde{\vect{G}}^{\text{loc}} (\xi,\eta)$. Note that the piece-wise continuous total fluxes $\widetilde{\vect{F}}$ and $\widetilde{\vect{G}}$ sit in the same polynomial space as that of $\widetilde{\vect{F}}^{\text{cor}}$ and $\widetilde{\vect{G}}^{\text{cor}}$. 
According to the structure of Eq.~\eqref{eq:flux_cor_2D}, $\widetilde{F}_i^{\text{loc}} (*,\eta) \in P^{k+2l-1} (\eta)$ along the $\eta$ direction, and $\widetilde{G}_i^{\text{loc}} (\xi,*) \in P^{k+2l-1} (\xi)$ along the $\xi$ direction, where $\widetilde{F}_i^{\text{loc}}$ and $\widetilde{G}_i^{\text{loc}}$ are the $i$-th component of $\widetilde{\vect{F}}^{\text{loc}}$ and $\widetilde{\vect{G}}^{\text{loc}}$, respectively. This gives flexibility to approximate $\widetilde{\vect{F}}^{\text{loc}}$ along the $\xi$ direction, and $\widetilde{\vect{G}}^{\text{loc}}$ along the $\eta$ direction. We study this issue by checking the DG formulation~\eqref{eq:DG_ref} in the integral form. 
%

Let $\vect{w}^r = \left( w_1^r, \ldots, w_j^r, \ldots, w_{(k+2l) \times (k+2l)}^r \right)$, and $w_j^r (\xi,\eta) \in \mathbb{Q}^{k+2l-1} (\xi,\eta)$, and $\nabla^r = (\partial/\partial \xi, \partial/\partial \eta)$ in 2D. When the Gauss--Legendre quadrature rule with $k+2l$ quadrature points in each dimension is used to carry out integration, it can integrate the degree $2(k+2l)-1$ polynomial exactly. Therefore, to ensure that all integrals in Eq.~\eqref{eq:DG_ref} can be integrated exactly, one choice is to approximate $\widetilde{\mathbf{F}}^{\text{loc}} = (\widetilde{\vect{F}}^{\text{loc}}, 
 \widetilde{\vect{G}}^{\text{loc}})$ as
 $\widetilde{F}_i^{\text{loc}} \in P^{k+2l}(\xi) \otimes P^{k+2l-1}(\eta)$, and $\widetilde{G}_i^{\text{loc}} \in P^{k+2l-1}(\xi) \otimes P^{k+2l}(\eta)$. Note that this indicates that $F_i^{\text{loc}}, G_i^{\text{loc}} \in \mathbb{Q}^{k+l} (\xi, \eta)$. We can further define the correction field $\widetilde{\vect{\delta}}^{\text{cor}} = |J| \vect{\delta}^{\text{cor}}$ with each of its component $\widetilde{\delta}_i^{\text{cor}} \in \mathbb{Q}^{k+2l-1} (\xi,\eta)$ or equivalently $\delta_i^{\text{cor}} \in \mathbb{Q}^k (\xi,\eta)$ as 
 \begin{equation}  \label{eq:field_cor_ref_integ}
 \int_{\Omega^r}   \widetilde{\vect{\delta}}^{\text{cor}} \otimes \vect{w}^r d V^r  = 
 \int_{\partial \Omega^r} \left( \widetilde{\vect{F}}_{n^r}^{\text{num}} - \widetilde{\vect{F}}_{n^r}^{\text{loc}} \right) \otimes \vect{w}^r d S^r.
\end{equation}
According to~\cite{YuWang_JSC_2013}, $\widetilde{\vect{\delta}}^{\text{cor}}$ defined in Eq.~\eqref{eq:field_cor_ref_integ} is equivalent to the following definition with $\widetilde{\vect{F}}^{\text{cor}}$ and $\widetilde{\vect{G}}^{\text{cor}}$ in 2D:
 \begin{equation}  \label{eq:field_cor_ref_diff}
 \widetilde{\vect{\delta}}^{\text{cor}} = 
 \frac{\partial \widetilde{\vect{F}}^{\text{cor}}}{\partial \xi} +
 \frac{\partial \widetilde{\vect{G}}^{\text{cor}}}{\partial \eta} .
\end{equation}
As a result, we can exactly convert the DG formulation~\eqref{eq:DG_ref} in the integral form into the FR formulation in the differential form numerically:
\begin{equation} \label{eq:2DFR_new_ref}
\frac{\partial \widetilde{\vect{Q}}_{\vect{h}}}{\partial t} +
\frac{\partial
\widetilde{\vect{F}}^{\text{loc}}}{\partial \xi} +
\frac{\partial
\widetilde{\vect{G}}^{\text{loc}}}{\partial \eta}+ 
\widetilde{\vect{\delta}}^{\text{cor}} = 0.
\end{equation}

Based on the aforementioned analyses, we summarize the polynomial degrees of solutions and fluxes used in the above analysis in Table~\ref{Tab:Phy_Variables}. 

\begin{table}[!htbp]
\centering
\begin{tabular}{|c|c|c|c|c|}
\hline
   & $\widetilde{\vect{Q}}$  
   & $\vect{Q}$ 
   & $\vect{F}^{\text{loc}}$ & $\vect{G}^{\text{loc}}$ \\ \hline
PS & $\mathbb{Q}^{k+2l-1} (\xi,\eta)$ 
   & $\mathbb{Q}^{k} (\xi,\eta)$ 
   & $\mathbb{Q}^{k+l} (\xi,\eta)$ 
   & $\mathbb{Q}^{k+l} (\xi,\eta)$  \\ \hline
   \hline
   & $\widetilde{\vect{\delta}}^{\text{cor}}$
   & $\vect{\delta}^{\text{cor}}$
   & $\widetilde{\vect{F}}^{\text{cor}}$, $\widetilde{\vect{F}}^{\text{loc}}$
   & $\widetilde{\vect{G}}^{\text{cor}}$, $\widetilde{\vect{G}}^{\text{loc}}$ \\ \hline
PS & $\mathbb{Q}^{k+2l-1} (\xi,\eta)$ 
   & $\mathbb{Q}^{k} (\xi,\eta)$ 
   & $P^{k+2l}(\xi) \otimes P^{k+2l-1}(\eta)$ 
   & $P^{k+2l-1}(\xi) \otimes P^{k+2l}(\eta)$ \\ \hline
\end{tabular}
\caption{A list of polynomial space where the FR solution, flux, and correction field variables in the physical domain and reference domain sit at any specific time.}
\label{Tab:Phy_Variables}
\end{table}

\begin{theorem}
The FR scheme~\eqref{eq:2DFR_new_ref}  with solution, flux and correction field polynomial degrees following those listed in Table~\ref{Tab:Phy_Variables}
numerically satisfies the following GCL (i.e. the GCL~\eqref{eq:GCL} in the stationary domain)
\begin{equation}
\frac{\partial }{\partial \xi}
\left (
\begin{array}{c}
      |J| \xi_x \\
      |J| \xi_y
\end{array} \right )  +
\frac{\partial }{\partial \eta}
\left (
\begin{array}{c}
      |J| \eta_x \\
      |J| \eta_y
\end{array} \right ) = 0,
\label{eq:GCL_Sta}
\end{equation}
no matter the metrics are calculated with the non-conservative or conservative approaches~\cite{AbeEtAl_JCP_2015}.  
\label{pro:sta}
\end{theorem}
\begin{proof}
Consider the freestream with constant $\vect{Q}$. The fluxes $\vect{F}$ and $\vect{G}$ are also constant. Note that in a stationary domain, the Jacobian $|J|$ is not a function of time. Therefore, when $\vect{Q}$ is constant, the term $\partial \widetilde{\vect{Q}}_{\vect{h}}/\partial t$ in Eq.~\eqref{eq:2DFR_new_ref} is zero. In numerical discretization, values that are analytically zero can only be represented approximately, appearing as machine zeros due to finite precision. However, these values are typically several tens of orders of magnitude smaller than non-trivial numerical errors, such as those arising from GCL violations (see examples in Ref.~\cite{Yu2011}). Therefore, machine zeros are neglected in the following discussions to facilitate a cleaner and more transparent proof. We then have from Eq.~\eqref{eq:2DFR_new_ref} that
\begin{equation}
\begin{aligned}
 &\vect{F}^{\text{loc}} \left(
   \sum_j (|J| \xi_x)_j^{\text{SP}} \frac{\partial w_{j}^r(\vect{\xi})}{\partial \xi}
  +\sum_j (|J| \eta_x)_j^{\text{SP}} \frac{\partial w_{j}^r(\vect{\xi})}{\partial \eta}
  \right) \\ 
+ \ & \vect{G}^{\text{loc}} \left(
      \sum_j (|J| \xi_y)_j^{\text{SP}} \frac{\partial w_{j}^r(\vect{\xi})}{\partial \xi}
      +\sum_j (|J| \eta_y)_j^{\text{SP}} \frac{\partial w_{j}^r(\vect{\xi})}{\partial \eta}
      \right) + \widetilde{\vect{\delta}}^{\text{cor}} = 0,
\end{aligned}
\label{eq:GCL_dis}
\end{equation}
with $w_j^r (\vect{\xi}) = w_j^r (\xi,\eta) \in \mathbb{Q}^{k+2l-1} (\xi,\eta)$, which is the basis function constructed via tensor product of 1D Lagrange polynomials rooted on solution points, and 
\begin{equation*}
\widetilde{\vect{F}}^{\text{loc}} (\vect{\xi}) = \sum_j (\widetilde{\vect{F}}^{\text{loc}})_{j}^{\text{SP}} w_{j}^r(\vect{\xi}), \
\widetilde{\vect{G}}^{\text{loc}} (\vect{\xi}) = \sum_j (\widetilde{\vect{G}}^{\text{loc}})_{j}^{\text{SP}} w_{j}^r(\vect{\xi}).
\end{equation*}
The superscript `SP' in Eq.~\eqref{eq:GCL_dis} indicates that the function is calculated at solution points.

We first prove that $\widetilde{\vect{\delta}}^{\text{cor}} = 0$.
Since $\vect{F}$ and $\vect{G}$ are constant, we have $\vect{F}_n^{\text{com}} = \vect{F}_n^{\text{loc}}$. 
Without loss of generality, check the term $\widetilde{\vect{F}}_L^{\text{num}}  (\eta)
- \widetilde{\vect{F}}^{\text{loc}}  (-1, \eta)$ in $\widetilde{\vect{F}}^{\text{cor}}  (\xi, \eta)$ from Eq.~\eqref{eq:flux_cor_2D}. It can be written as
\[
\begin{aligned}
\widetilde{\vect{F}}_L^{\text{num}}  (\eta)
- \widetilde{\vect{F}}^{\text{loc}}  (-1, \eta) & =
\vect{F}^{\text{loc}} (-1, \eta) \left[ (|J| \xi_x)^{\text{FP}}(-1, \eta) - (|J| \xi_x)^{\text{SP,I}}(-1,\eta) \right] \\ & +
\vect{G}^{\text{loc}} (-1, \eta) \left[ (|J| \xi_y)^{\text{FP}}(-1, \eta) - (|J| \xi_y)^{\text{SP,I}}(-1,\eta) \right].
\end{aligned}
\]
Herein, the superscript `FP' indicates that the function is calculated at flux points, and `SP,I' indicates that the function is first calculated at solution points and then interpolated from solution points to the corresponding flux points. 

Based on Table~\ref{Tab:metrics}, we have $|J| \xi_x, |J| \xi_y \in P^l(\xi) \otimes P^{l-1}(\eta)$. To construct the FR scheme~\eqref{eq:2DFR_new_ref}, the polynomial space $\mathbb{Q}^{k+2l-1} (\xi,\eta)$ with $k+2l$ Gauss--Legendre quadrature points (i.e. solution points in FR) in each dimension is used. Therefore, $|J| \xi_x$ and $|J| \xi_y$ can be exactly evaluated at any point within the curvilinear element when they are constructed as polynomials analytically. As a result, we have
$(|J| \xi_{x/y})^{\text{FP}} = (|J| \xi_{x/y})^{\text{SP,I}}$, and $\widetilde{\vect{F}}_L^{\text{num}}  (\eta)
- \widetilde{\vect{F}}^{\text{loc}}  (-1, \eta) = 0$. Furthermore, we have $\widetilde{\vect{F}}^{\text{cor}} = 0$, $\widetilde{\vect{G}}^{\text{cor}} = 0$, and thus, $\widetilde{\vect{\delta}}^{\text{cor}} = 0$.

Now Eq.~\eqref{eq:GCL_dis} is reduced to
\begin{equation*}
\begin{aligned}
 &\vect{F}^{\text{loc}} \left(
   \sum_j (|J| \xi_x)_j^{\text{SP}} \frac{\partial w_{j}(\vect{\xi})}{\partial \xi}
  +\sum_j (|J| \eta_x)_j^{\text{SP}} \frac{\partial w_{j}(\vect{\xi})}{\partial \eta}
  \right) \\ 
+ \ & \vect{G}^{\text{loc}} \left(
      \sum_j (|J| \xi_y)_j^{\text{SP}} \frac{\partial w_{j}(\vect{\xi})}{\partial \xi}
      +\sum_j (|J| \eta_y)_j^{\text{SP}} \frac{\partial w_{j}(\vect{\xi})}{\partial \eta}
      \right) = 0.
\end{aligned}
\end{equation*}

Since $|J| \xi_x, |J| \xi_y \in P^l(\xi) \otimes P^{l-1}(\eta)$ and $|J| \eta_x, |J| \eta_y \in P^{l-1}(\xi) \otimes P^l(\eta)$, they can be fully represented with $w_j^r (\vect \xi)$. Therefore, taking $|J| \xi_x$ as an example, we have
\begin{equation}
(|J| \xi_x) (\vect{\xi}) = \sum_j (|J| \xi_x)_j^{\text{SP}} w_{j}^r(\vect{\xi}),
\label{eq:J_xi}
\end{equation}
and
\begin{equation}
\frac{\partial (|J| \xi_x) (\vect{\xi})}{\partial \xi} = \sum_j (|J| \xi_x)_j^{\text{SP}} \frac{\partial w_{j}^r(\vect{\xi})}{\partial \xi}.
\label{eq:dJ_dxi}
\end{equation}
Note that $(|J| \xi_{x})_j^{\text{SP}}$ used in Eq.~\eqref{eq:J_xi} and~\eqref{eq:dJ_dxi} can be calculated either with the non-conservative approach (see Eq.~\eqref{eq:Jacob_inv_2D}) or with the conservative approach (see Ref.~\cite{AbeEtAl_JCP_2015}). Similar results apply to $|J| \xi_y$, $|J| \eta_x$, and $|J| \eta_y$.  

On plugging Eq.~\eqref{eq:J_xi} and~\eqref{eq:dJ_dxi} in Eq.~\eqref{eq:GCL_dis}, we have
\begin{equation}
 \vect{F}^{\text{loc}} \left(
    \frac{\partial |J| \xi_x}{\partial \xi}
  + \frac{\partial |J| \eta_x}{\partial \eta}
  \right)
+ \vect{G}^{\text{loc}} \left(
      \frac{\partial|J| \xi_y}{\partial \xi}
      +\frac{\partial |J| \eta_y}{\partial \eta}
      \right) = 0.
\end{equation}

Since the geometric features of the curvilinear element are fully represented by polynomials in the FR scheme, based on \textbf{Lemma}~\ref{lem:GCL_num}, we have $\partial |J| \xi_x/\partial \xi + \partial |J| \eta_x/\partial \eta= 0$, and $\partial |J| \xi_y/\partial \xi + \partial |J| \eta_y/\partial \eta= 0$.

This proves that for freestream the GCL~\eqref{eq:GCL_Sta} held numerically. 

\end{proof}

With \textbf{Theorem}~\ref{pro:sta}, the FR scheme~\eqref{eq:2DFR_new_ref} can be numerically converted back to its physical domain format as 
\begin{equation} \label{eq:2DFR_new_phy}
\frac{\partial \vect{Q}_{\vect{h}}}{\partial t} +
\frac{\partial
\vect{F}^{\text{loc}}}{\partial x} +
\frac{\partial
\vect{G}^{\text{loc}}}{\partial y}+ 
\vect{\delta}^{\text{cor}} = 0.
\end{equation}
%


Key procedures to implement Eq.~\eqref{eq:2DFR_new_phy} are summarized in \textbf{Key Procedures}~\ref{alg:2DFR_key_new}.
%
Note that the FR scheme~\eqref{eq:2DFR_new_phy}  is different from Eq.~\eqref{eq:FR_Diff} discussed in \textbf{Sect.}~\ref{subsubsec:RefD} as in Eq.~\eqref{eq:2DFR_new_phy} variables describing flow physics and those describing geometry are approximated by separate polynomials with $\widetilde{\vect{Q}}$ served as the hidden working variable. 


%

\begin{algorithm}[!htbp]
\caption{A Hybrid Reference-Physics Domain Approach}
\label{alg:2DFR_key_new}
\begin{algorithmic} [1]
\State  
The terms $\partial \vect{F}^{\text{loc}} / \partial x $ and $\partial \vect{G}^{\text{loc}} / \partial y $ on solution points in each dimension can be calculated from
\begin{equation*}
\frac{\partial \vect{F}^{\text{loc}}}{\partial x} =
\frac{\partial \vect{F}^{\text{loc}}}{\partial \xi} \xi_x + 
\frac{\partial \vect{F}^{\text{loc}}}{\partial \eta} \eta_x, \quad
\frac{\partial \vect{G}^{\text{loc}}}{\partial y} =
\frac{\partial \vect{G}^{\text{loc}}}{\partial \xi} \xi_y + 
\frac{\partial \vect{G}^{\text{loc}}}{\partial \eta} \eta_y.
\end{equation*}

\noindent \Comment{Although $\partial \xi_i / \partial x_j$ are not polynomials, $|J| \partial \xi_i / \partial x_j$ are polynomials sitting either in $P^l(\xi) \otimes P^{l-1}(\eta)$ or in $P^{l-1}(\xi) \otimes P^l(\eta)$.} This numerically ensures that
\begin{equation*}
|J| \left(
\frac{\partial \vect{F}^{\text{loc}}}{\partial x} +
\frac{\partial \vect{G}^{\text{loc}}}{\partial y} 
\right) =
\frac{\partial \widetilde{\vect{F}}^{\text{loc}}}{\partial \xi} +
\frac{\partial \widetilde{\vect{G}}^{\text{loc}}}{\partial \eta}
\end{equation*}
holds on solution points in each dimension.

\State
Use the univariate correction functions $g_{L/R} (*) $ for $\widetilde{\vect{F}}^{\text{cor}}$ and $\widetilde{\vect{G}}^{\text{cor}}$ in Eq.~\eqref{eq:flux_cor_2D}, and evaluate $\widetilde{\vect{\delta}}^{\text{cor}}$ from Eq.~\eqref{eq:field_cor_ref_diff}. Then $\vect{\delta}^{\text{cor}}$ can be calculated from $\vect{\delta}^{\text{cor}} = \widetilde{\vect{\delta}}^{\text{cor}}/|J|$ on the solution points in each dimension.

\noindent \Comment{Only the values of $g^{\prime}_{L/R} (*)$ are needed on 1D solution points. These values can be stored in pre-processing.}

\end{algorithmic}
\end{algorithm}

Before ending this section, we briefly compare the consistent symmetric conservative metric (CSC) construction approach in~\cite{AbeEtAl_JCP_2015} and the new approach developed above. Both approaches have the global conservation and freestream preservation properties. The differences are summarized as follows: 
\begin{itemize}
    \item Different from the CSC approach, metrics in the new FR approach can be calculated with either the non-conservative formulation or the conservative one.
    \item The metrics in the CSC approach are reconstructed from their exact values to match the order of the accuracy of physical variables. However, the metrics in the new FR approach here are exact, and their polynomial order is independent from that of physical variables.
\end{itemize}

\subsubsection{Space-time moving curvilinear elements} \label{subsubsec:ST_MCE}
For space-time curvilinear elements over a moving domain, let $x,y \in \mathbb{Q}^l (\xi, \eta) \otimes P^n (\tau)$, $t \in P^n (\tau)$. Thus, the Jacobian $|J| = t_{\tau} (x_{\xi} y_{\eta} - x_{\eta} y_{\xi}) \in \mathbb{Q}^{2l-1} (\xi, \eta) \otimes P^{3n-1} (\tau)$. The polynomial space where the Jacobian and metrics sit for a tensor-product curvilinear space-time element are summarized in Table~\ref{Tab:ST_metrics}.

\begin{table}[!htbp] 
\centering
\begin{tabular}{|c|c|c|c|}
\hline
    & $|J|$ & $|J|\xi_x$, $|J|\xi_y$ & $|J|\eta_x$, $|J|\eta_y$ \\
\hline
PS &  $\mathbb{Q}^{2l-1} (\xi,\eta) \otimes P^{3n-1} 
      (\tau)$ 
   & $P^l(\xi) \otimes P^{l-1}(\eta) \otimes P^{2n-1}(\tau)$ 
   & $P^{l-1}(\xi) \otimes P^l(\eta) \otimes P^{2n-1}(\tau)$ \\ 
\hline
\hline
    & $|J| \tau_t$ & $|J|\xi_t$ & $|J|\eta_t$ \\
\hline
PS &  $\mathbb{Q}^{2l-1} (\xi,\eta) \otimes P^{2n} 
      (\tau)$ 
   & $P^{2l}(\xi) \otimes P^{2l-1}(\eta) \otimes P^{2n-1}(\tau)$ 
   & $P^{2l-1}(\xi) \otimes P^{2l}(\eta) \otimes P^{2n-1}(\tau)$ \\ 
\hline
\end{tabular}
\caption{A list of polynomial space occupied by the Jacobian and metrics of a space-time moving curvilinear element with $x,y \in \mathbb{Q}^l (\xi, \eta) \otimes P^n (\tau)$ and $t \in P^n (\tau)$. Note that $l \geq 1$ and $n \geq 1$.}
\label{Tab:ST_metrics}
\end{table}

To numerically solve Eq.~\eqref{eq:GovEq_2D_Comp}, a pseudo-time term $\partial |J| \vect{Q}_{\vect{h}} / \partial \widetilde{t} $ is augmented to the governing equation. As a result, the STFR scheme can be written as
\begin{equation} \label{eq:STFR_new_ref}
\frac{\partial |J| \vect{Q}_{\vect{h}}}{\partial \widetilde{t}} +
\frac{\partial \widetilde{\vect{Q}}^{\text{loc}}}{\partial \tau} +
\frac{\partial \widetilde{\vect{F}}^{\text{loc}}}{\partial \xi} +
\frac{\partial \widetilde{\vect{G}}^{\text{loc}}}{\partial \eta}+ 
\widetilde{\vect{\delta}}^{\text{cor,st}} = 0,
\end{equation}
with
 \begin{equation}  \label{eq:ST_field_cor_ref_diff}
 \widetilde{\vect{\delta}}^{\text{cor,st}} = 
 \frac{\partial \widetilde{\vect{Q}}^{\text{cor}}}{\partial \tau} +
 \frac{\partial \widetilde{\vect{F}}^{\text{cor}}}{\partial \xi} +
 \frac{\partial \widetilde{\vect{G}}^{\text{cor}}}{\partial \eta} .
\end{equation}
Since $t$ is only a function of $\tau$, we have from Eq.~\eqref{eq:ST_Var_ref} that $\widetilde{\vect{Q}}^{\text{loc}} = |J| \tau_t \vect{Q}^{\text{loc}}$. Define $Q_{h,i} (\xi, \eta, \tau) \in \mathbb{Q}^k (\xi, \eta) \otimes P^m (\tau)$. Then following the procedures explained in \textbf{Sect.}~\ref{subsubsec:Stat_CE}, the STFR scheme~\eqref{eq:STFR_new_ref} can be converted back to its physical domain format as 
\begin{equation} \label{eq:STFR_new_phy}
\frac{\partial \vect{Q}_{\vect{h}}}{\partial \widetilde{t}} +
\frac{\partial \vect{Q}^{\text{loc}}}{\partial t} +
\frac{\partial \vect{F}^{\text{loc}}}{\partial x} +
\frac{\partial \vect{G}^{\text{loc}}}{\partial y} + 
\vect{\delta}^{\text{cor,st}} = 0.
\end{equation}
Similar to Table~\ref{Tab:Phy_Variables}, the polynomial degrees of solutions, fluxes and correction field used in the implementation of Eq.~\eqref{eq:STFR_new_phy} are summarized in Table~\ref{Tab:ST_Variables}.

\begin{table}[!htbp]
\caption{A list of polynomial space occupied by the Jacobian and metrics of a space-time moving curvilinear element with $x,y \in \mathbb{Q}^l (\xi, \eta) \otimes P^n (\tau)$ and $t \in P^n (\tau)$.}
\label{Tab:ST_Variables}
\centering
\begin{threeparttable}
\begin{tabular}{|c|c|c|}
\hline
   & $\vect{Q}_{\vect{h}}$ & 
     $|J| \vect{Q}_{\vect{h}}$\\ \hline
PS & $\mathbb{Q}^k (\xi, \eta) \otimes P^m (\tau)$ &
     $\mathbb{Q}^{k+2l-1} (\xi, \eta) \otimes P^{m+3n-1} (\tau)$ \\ \hline
\hline
   & $\vect{\delta}^{\text{cor,st}}$ &
     $\widetilde{\vect{\delta}}^{\text{cor,st}}$ \\ \hline
PS & $\mathbb{Q}^k (\xi, \eta) \otimes P^m (\tau)$ &
     $\mathbb{Q}^{k+2l-1} (\xi, \eta) \otimes P^{m+3n-1} (\tau)$ \\ \hline
\hline
   & $\vect{Q}^{\text{loc}}$ &
   $\widetilde{\vect{Q}}^{\text{loc}},
   \widetilde{\vect{Q}}^{\text{cor}}$\\ \hline
PS & $\mathbb{Q}^k (\xi, \eta) \otimes P^{m+n} (\tau)$\tnote{*} &
     $\mathbb{Q}^{k+2l-1} (\xi, \eta) \otimes P^{m+3n} (\tau)$\\ \hline
\hline
   & $\vect{F}^{\text{loc}}$ &
   $\widetilde{\vect{F}}^{\text{loc}},
   \widetilde{\vect{F}}^{\text{cor}}$\\ \hline
PS & $\mathbb{Q}^{k+l} (\xi, \eta) \otimes P^{m+n} (\tau)$ &
     $P^{k+2l} (\xi) \otimes P^{k+2l-1} (\eta) \otimes P^{m+3n-1} (\tau)$\\ \hline
\hline
   & $\vect{G}^{\text{loc}}$ &
   $\widetilde{\vect{G}}^{\text{loc}},
   \widetilde{\vect{G}}^{\text{cor}}$\\ \hline
PS & $\mathbb{Q}^{k+l} (\xi, \eta) \otimes P^{m+n} (\tau)$ &
     $P^{k+2l-1} (\xi) \otimes P^{k+2l} (\eta) \otimes P^{m+3n-1} (\tau)$\\ \hline
\end{tabular}
\begin{tablenotes}
    \item[*] This is used to assist understanding the numerical effect of the temporal correction field.
\end{tablenotes}
\end{threeparttable}
\end{table}

Moreover, analogous to \textbf{Theorem}~\ref{pro:sta}, we have the following conclusion:
\begin{theorem} \label{pro:moving}
The FR scheme~\eqref{eq:STFR_new_ref} with solution, flux and correction field polynomial degrees following those listed in Table~\ref{Tab:ST_Variables}, and with a spatial order higher than two (i.e. $k \geq 1$),
numerically satisfies the GCL~\eqref{eq:GCL},
no matter the metrics are calculated with the non-conservative or conservative approaches~\cite{AbeEtAl_JCP_2015}.  Moreover, under the conditions listed above, Eq.~\eqref{eq:STFR_new_ref} is numerically equivalent to Eq.~\eqref{eq:STFR_new_phy}.
\end{theorem}

\begin{proof}
The proof is analogous to that for \textbf{Theorem}~\ref{pro:sta} by adding the numerical representation of $\partial \widetilde{\vect{Q}}^{\text{loc}}/\partial \tau$, and that of the temporal flux (i.e. solution) $\vect{Q}$-related terms in  $\widetilde{\vect{F}}^{\text{loc}}$ and $\widetilde{\vect{G}}^{\text{loc}}$. Again, consider the freestream with constant $\vect{Q}$. Thus, Eq.~\eqref{eq:GCL_dis} is augmented with the following term
\begin{equation}
\vect{Q}^{\text{loc}} \left(
   \sum_j (|J| \tau_t)_j^{\text{SP}} \frac{\partial w_{j}^r(\vect{\xi}, \tau)}{\partial \tau}
  +\sum_j (|J| \xi_t)_j^{\text{SP}} \frac{\partial w_{j}^r(\vect{\xi}, \tau)}{\partial \xi}
  +\sum_j (|J| \eta_t)_j^{\text{SP}} \frac{\partial w_{j}^r(\vect{\xi}, \tau)}{\partial \eta}
  \right).
\end{equation}

Different from the GCL~\eqref{eq:GCL_Sta} used in \textbf{Theorem}~\ref{pro:sta}, the first formula in the GCL~\eqref{eq:GCL} needs the following two equations
\begin{equation}
\label{eq:GCL_Cond_strong}
    k+2l-1 \geq 2l, \ \text{and} \ m+3n-1 \geq 2n
\end{equation}
to hold to ensure that $|J| \tau_t$, $|J| \xi_t$, and $|J| \eta_t$ can be exactly evaluated at any point within the curvilinear space-time element when they are constructed as polynomials analytically. Since $n \geq 1$ for any space-time element (i.e. $\tau$ is at least a linear function of $t$), we have $n = 3n - 2n \geq 1 \geq 1 - m $, when $m \geq 0$. This means that the condition $m+3n-1 \geq 2n$ always holds. Therefore, the only condition that needs to meet is $k \geq 1$.  

\end{proof}

Compared to the stationary grid case, special attention is needed when evaluating $\vect{Q}^{\text{loc}}$ and $\widetilde{\vect{Q}}^{\text{cor}}$ from $\vect{Q}_{\vect{h}}$. 
Note that there are $m+3n$ solution points along the time dimension in the present space-time framework.
As shown in Table~\ref{Tab:ST_Variables}, to ensure that the polynomial approximations of $\widetilde{\vect{Q}}^{\text{loc}}$ and $\widetilde{\vect{Q}}^{\text{cor}}$ are consistent with those of $\widetilde{\vect{F}}^{\text{loc}/\text{cor}}$ and $\widetilde{\vect{G}}^{\text{loc}/\text{cor}}$ to ease analysis and implementation,
$\vect{Q}^{\text{loc}}$ can always be treated as degree $m+n$ polynomials along the time dimension. 
Thus, the degree $P^{m+3n}$ correction function $g_{L/R} (\tau)$ is used for $\widetilde{\vect{Q}}^{\text{cor}}$ to completely eliminate the aliasing errors caused by the Jacobian $|J|$. As a result, we have the following observation regarding $\vect{Q}^{\text{loc}}$ and $\vect{Q}_{\vect{h}}$.

\begin{remark}
Since  $g_{L/R} (\tau) \in P^{m+3n} (\tau)$, this indicates that $\vect{Q}^{\text{loc}} \in \mathbb{Q}^k (\xi, \eta) \otimes P^{m+n} (\tau)$ essentially. This high-order information has been propagated to $\vect{Q}_{\vect{h}}$ through the correction field to enhance its accuracy. 
\end{remark}


\subsubsection{Discussions and aliasing error control} \label{subsubsec:new_discuss}

According to Table~\ref{Tab:ST_Variables}, the STFR formulation for moving curvilinear elements can suffer from the so-called ``curse of dimensionality", especially in the time dimension. There are two ways to fix this issue. As the first one, the CSC approach in~\cite{AbeEtAl_JCP_2015} may provide a solution by matching the polynomial degrees of symmetric conservative metrics to that of physical variables. As the other, we can find solutions by re-examining \textbf{Theorem}~\ref{pro:sta} and \ref{pro:moving}. 

Note that when considering the GCL~\eqref{eq:GCL}, the solution $\vect{Q}$, and fluxes $\vect{F}$ and $\vect{G}$ are treated as constants. For moving grid simulations, based on Table~\ref{Tab:ST_metrics}, the most demanding GCL formulation is
\begin{equation*}
\frac{\partial |J| \tau_t }{\partial \tau}
+
\frac{\partial |J| \xi_t}{\partial \xi}
+
\frac{\partial |J| \eta_t}{\partial \eta}
= 0.
\end{equation*}
We can use the basis function $w_j^r (\xi,\eta, \tau) \in \mathbb{Q}^{2l} (\xi,\eta) \otimes P^{2n} (\tau)$, with $x,y \in \mathbb{Q}^l (\xi, \eta) \otimes P^n (\tau)$ and $t \in P^n (\tau)$, in the discrete GCL when constructing metrics. Since $|J| \tau_t$, $|J| \xi_t$, and $|J| \eta_t$ are at most of degree $2l$ in $\xi$ and $\eta$ directions, and at most of degree $2n$ in the $\tau$ direction, they can be exactly evaluated at any point within the curvilinear space-time element when they are constructed as polynomials analytically. As a result, derivations related to the discrete GCL in \textbf{Theorem}~\ref{pro:sta} and \ref{pro:moving} hold. 
This gives the following conclusion:
\begin{proposition} \label{pro:moving_reduced}
To numerically satisfy 
the GCL~\eqref{eq:GCL} by using exact metrics of curvilinear moving grids with a nominal solution approximation $Q_{h,i} (\xi, \eta, \tau) \in \mathbb{Q}^k (\xi, \eta) \otimes P^m (\tau)$, the basis function $w_j^r (\xi,\eta, \tau)$ should satisfy the following requirement:
\begin{equation} \label{eq:order_req}
    w_j^r (\xi,\eta, \tau) \in \mathbb{Q}^{\alpha} (\xi,\eta) \otimes P^{\beta} (\tau), \ \text{where} \ \alpha = \text{max} \ (2l, k) \ \text{and} \ \beta = \text{max} \ (2n, m).
\end{equation}
\end{proposition}
This condition should not be confused with Eq.~\eqref{eq:GCL_Cond_strong} used in \textbf{Theorem}~\ref{pro:moving}.  Note that $\widetilde{\vect{Q}} = |J| \vect{Q}$ is used as the hidden working variable when implementing Eq.~\eqref{eq:2DFR_new_phy} or~\eqref{eq:STFR_new_phy} strictly following the polynomial degree requirements listed in Table~\ref{Tab:Phy_Variables} or~\ref{Tab:ST_Variables}. Therein, $\alpha = k+2l-1 \geq \text{max} (2l,k)$ when $k \geq 1$, and $\beta = m+3n-1 \geq \text{max} (2n,m)$, considering the fact that $l \geq 1$ and $n \geq 1$. We also mention that \textbf{Proposition}~\ref{pro:moving_reduced} gives concrete numerical conditions that satisfy \textbf{Lemma}~\ref{lem:GCL_num}.

Note that the requirement~\eqref{eq:order_req} is similar to those required to satisfy freestream preservation conditions for curvilinear stationary grids in the work by Abe \textit{et al.}~\cite{AbeEtAl_JCP_2015}. However, we have the following conclusion regarding the requirement~\eqref{eq:order_req}.

\begin{remark}
Although satisfying the requirement~\eqref{eq:order_req} ensures that the GCL~\eqref{eq:GCL} is held numerically (or the discrete GCL holds), it cannot guarantee that Eq.~\eqref{eq:STFR_new_ref} is numerically equivalent to Eq.~\eqref{eq:STFR_new_phy}. We emphasize that only when the conditions in \textbf{Theorem}~\ref{pro:moving} are satisfied, Eq.~\eqref{eq:STFR_new_ref} is numerically equivalent to Eq.~\eqref{eq:STFR_new_phy}. 
\end{remark}

Recall the discussions in \textbf{Sect.}~\ref{subsubsec:RefD}, explicitly using $\widetilde{\vect{Q}}$ as the working variable may cause accuracy deterioration of the solution $\vect{Q}$. Therefore, it is desirable to implement Eq.~\eqref{eq:STFR_new_phy} with respect to the reference domain (see \textbf{Key Procedures}~\ref{alg:2DFR_key_new} in \textbf{Sect.}~\ref{subsubsec:Stat_CE}). Since it is very challenging, if not impossible, to exactly measure how the geometric nonlinearity affects the solution accuracy, we develop projection-based polynomial filtering procedures to control the aliasing errors caused by the nonlinear interaction between physical quantities of the flow and the curvilinear geometric representation of space-time elements. 

Assume that any solution variable $Q^H (\vect{\xi}) = Q^H (\xi, \eta, \tau)$ from the solution vector $\vect{Q}$ is approximated by the polynomial space $\mathbb{Q}^{k^H} (\xi, \eta) \otimes P^{m^H} (\tau)$, and the basis functions are the multidimensional Lagrange polynomials $L^H_{i_1,i_2,i_3} (\vect{\xi})= \phi^H_{i_1}(\xi) \phi^H_{i_2}(\eta) \phi^H_{i_3}(\tau)$, $i_1, i_2 = 1, \ldots, k^H+1$ and $i_3 = 1, \ldots, m^H+1$, consisting of 1D Lagrange polynomials $\phi^H_i (*)$, where `$*$' stands for $\xi \in [-1,1]$, $\eta \in [-1,1]$, or $\tau \in [-1,1]$. $\phi^H_i (*)$'s are constructed based on the $k^H + 1$ Gauss--Legendre quadrature points (i.e. solution points in FR), denoted as $\xi^H_j$ or $\eta^H_j$, in each spatial dimension, or on the $m^H + 1$ Gauss--Legendre quadrature points, denoted as $\tau^H_j$, in the time dimension. The corresponding quadrature weights associated with spatial quadrature points are denoted as $\omega^H_{s,j}$, and those with temporal quadrature points as $\omega^H_{t,j}$. 
The projection space sits in $\mathbb{Q}^{k^L} (\xi, \eta) \otimes P^{m^L} (\tau)$, and all quantities associated with this space are marked with a superscript `$L$', such as the projected solution $Q^L (\vect{\xi}) = Q^L (\xi, \eta, \tau)$ and its basis function $L^L_{i_1,i_2,i_3} (\vect{\xi})= \phi^L_{i_1}(\xi) \phi^L_{i_2}(\eta) \phi^L_{i_3}(\tau)$, $i_1, i_2 = 1, \ldots, k^L+1$ and $i_3 = 1, \ldots, m^L+1$. We can then express $Q^H (\vect{\xi})$ and $Q^L (\vect{\xi})$ as
\begin{equation}
\begin{aligned}
Q^H(\vect{\xi}) & = 
\sum_{i_1=1}^{k^H+1}  \sum_{i_2=1}^{k^H+1} \sum_{i_3=1}^{m^H+1} 
Q^H(\vect{\xi}^H_{i_1,i_2,i_3})
L^H_{i_1,i_2,i_3} (\vect{\xi}) \\
& =
\sum_{i_1=1}^{k^H+1}  \sum_{i_2=1}^{k^H+1} \sum_{i_3=1}^{m^H+1} Q^H_{i_1,i_2,i_3}
\phi^H_{i_1}(\xi) \phi^H_{i_2}(\eta) \phi^H_{i_3}(\tau)
\\
Q^L(\vect{\xi}) & =  
\sum_{i_1=1}^{k^L+1}  \sum_{i_2=1}^{k^L+1} \sum_{i_3=1}^{m^L+1} Q^L(\vect{\xi}^L_{i_1,i_2,i_3})
L^L_{i_1,i_2,i_3} (\vect{\xi}) \\
& =
\sum_{i_1=1}^{k^L+1}  \sum_{i_2=1}^{k^L+1} \sum_{i_3=1}^{m^L+1} Q^L_{i_1,i_2,i_3}
\phi^L_{i_1}(\xi) \phi^L_{i_2}(\eta) \phi^L_{i_3}(\tau)
\end{aligned},
\label{eq:Sol_Lagrange}
\end{equation}
and the projection operation can be written as
\begin{equation}
\begin{aligned}
&   \iiint \sum_{i_1=1}^{k^L+1}  
    \sum_{i_2=1}^{k^L+1} \sum_{i_3=1}^{m^L+1} Q^L_{i_1,i_2,i_3} L^L_{i_1,i_2,i_3} (\vect{\xi}) 
    L^L_{j_1,j_2,j_3} (\vect{\xi}) d\vect{\xi} \\
= & \iiint \sum_{i_1=1}^{k^H+1}  
    \sum_{i_2=1}^{k^H+1} \sum_{i_3=1}^{m^H+1} Q^H_{i_1,i_2,i_3} L^H_{i_1,i_2,i_3} (\vect{\xi}) 
    L^L_{j_1,j_2,j_3} (\vect{\xi}) d\vect{\xi}
\end{aligned},
\label{eq:Proj}
\end{equation}
for any test function $L^L_{j_1,j_2,j_3} (\vect{\xi})$, $j_1, j_2 = 1, \ldots, k^L+1$ and $j_3 = 1, \ldots, m^L+1$. Note that $Q^H_{i_1,i_2,i_3}$'s are known on solution points, and $Q^L_{j_1,j_2,j_3}$'s are calculated from Eq.~\eqref{eq:Proj} as
\begin{equation} \label{eq:Proj_Coef}
   Q^L_{j_1,j_2,j_3} = \frac{1}{\omega_{s,j_1}^L \omega_{s,j_2}^L \omega_{s,j_3}^L}
   \sum_{i_1=1}^{k^H+1}  
    \sum_{i_2=1}^{k^H+1} \sum_{i_3=1}^{m^H+1} Q^H_{i_1,i_2,i_3}
    \phi^L_{j_1}(\xi_{i_1}^H) \phi^L_{j_2}(\eta_{i_2}^H) \phi^L_{j_3}(\tau_{i_3}^H)
    \omega_{s,i_1}^H \omega_{s,i_2}^H \omega_{s,i_3}^H.
\end{equation}
A proof of Eq.~\eqref{eq:Proj_Coef} is given in~\ref{app:proj}. 

Then with $Q^L(\vect{\xi})$ defined in Eq.~\eqref{eq:Sol_Lagrange}, we can calculate $Q^L(\vect{\xi}_{j_1,j_2,j_3}^H)$, $j_1, j_2 = 1, \ldots, k^H+1$ and $j_3 = 1, \ldots, m^H+1$.
Since $k^L < k^H$ and $m^L < m^H$, $Q^L(\vect{\xi})$ can also be defined as
\begin{equation}
Q^L(\vect{\xi})  = 
\sum_{i_1=1}^{k^H+1}  \sum_{i_2=1}^{k^H+1} \sum_{i_3=1}^{m^H+1} 
Q^L(\vect{\xi}^H_{i_1,i_2,i_3})
\phi^H_{i_1}(\xi) \phi^H_{i_2}(\eta) \phi^H_{i_3}(\tau).
\label{eq:L_Lagrange}
\end{equation}
A proof of the equivalency between $Q^L(\vect{\xi})$ defined in Eq.~\eqref{eq:Sol_Lagrange} and that in Eq.~\eqref{eq:L_Lagrange} is given in~\ref{app:interp}.

On defining a parameter $\theta \in [0,1]$, and $\Delta f_{i_1,i_2,i_3}^H = Q^H(\vect{\xi}^H_{i_1,i_2,i_3}) - Q^L(\vect{\xi}_{i_1,i_2,i_3}^H)$, we can reconstruct $\bar{Q}^H (\vect{\xi})$, the filtered $Q^H (\vect{\xi})$, as
\begin{equation}
    \bar{Q}^H (\vect{\xi}) = 
    \sum_{i_1=1}^{k^H+1}  \sum_{i_2=1}^{k^H+1} \sum_{i_3=1}^{m^H+1} 
\left(Q^L(\vect{\xi}^H_{i_1,i_2,i_3}) +
      \theta \Delta f_{i_1,i_2,i_3}^H
\right)
\phi^H_{i_1}(\xi) \phi^H_{i_2}(\eta) \phi^H_{i_3}(\tau).
\label{eq:Fil_Lagrange}
\end{equation}
It is clear that when $\theta = 1$,  $\bar{Q}^H (\vect{\xi}) =  Q^H (\vect{\xi})$; and when $\theta = 0$, $\bar{Q}^H (\vect{\xi}) =  Q^L (\vect{\xi})$. On defining $\Delta f(\vect{\xi}) = Q^H (\vect{\xi}) - Q^L (\vect{\xi})$, and its energy $E_g = \iiint \Delta f^2(\vect{\xi}) d \vect{\xi}$, we have the following conclusion:
\begin{itemize}
    \item For any $\theta \in [0,1]$, a $(1-\theta^2)$ portion of $E_g$ is exactly filtered from $Q^H (\vect{\xi})$ in the numerical reconstruction of the solution by Eq.~\eqref{eq:Fil_Lagrange}. 
\end{itemize} 
This conclusion is also proved in ~\ref{app:interp}.

To summarize, key procedures of projection-based polynomial filtering for aliasing error control of curvilinear space-time moving elements after each physical time step when implementing Eq.~\eqref{eq:STFR_new_phy} are given below:

\begin{algorithm}[!htbp]
\caption{Projection-based Filtering for Aliasing Error Control}
\label{alg:filter_proj}
\begin{algorithmic} [1]
\State  
Perform the solution projection from the polynomial space $\mathbb{Q}^{k^H} (\xi, \eta) \otimes P^{m^H} (\tau)$ to $\mathbb{Q}^{k^L} (\xi, \eta) \otimes P^{m^L} (\tau)$ following Eq.~\eqref{eq:Proj} and~\eqref{eq:Proj_Coef}.

\State
Reconstruct the filtered solution polynomial using Eq.~\eqref{eq:Fil_Lagrange}.

\end{algorithmic}
\end{algorithm}

\section{Numerical Experiments}
\label{sec:results} 

In this section, we use numerical experiments to study the spatial and temporal accuracy of the nodal tensor-product STFR method on curvilinear moving grids. Both 1D and 2D hyperbolic conservation laws are used to conduct the numerical tests. For example, when the 1D linear wave propagation (advection) problem is studied, $Q$ becomes a scalar, and the flux $F = cQ$, where $c$ is a constant. When the 2D nonlinear Euler vortex propagation problem is studied, we have
$\vect{Q} = \left( \rho, \rho u, \rho v, E \right)$, $\vect{F} = \left( \rho u, \rho u^2 + p, \rho u v, (E+p)u \right)$, and $\vect{G} = \left( \rho v, \rho u v, \rho v^2 + p, (E+p)v \right)$,
where $\rho$ is the fluid density, $u$ and $v$ are the velocities in the $x$- and $y$-directions, $p$ is the pressure, and $E$ is the total energy per unit volume. The Euler system is closed by the perfect gas law. Periodic boundary conditions are enforced in space in all tests presented in this section.

To measure the accuracy, the $\ell_2$ error of any quantity $q$ within the spatial domain $\Omega$ at a specific time, e.g. the time $t_n$ of the simulation, is used. It is defined as
\begin{equation} \label{eq:L2Norm}
  \norm{\Delta q_n}_{\ell_2; \Omega} = 
  \left(
  \frac{\int_{\Omega} \left(
  q_{h,n}^I - q_{a,n}
  \right)^2 dV}{V}
  \right)^{\frac{1}{2}}.
\end{equation}
Herein, $q_h$ is the space-time approximation of the quantity $q$, $q_a$ is the analytical (exact) value of $q$,  and the superscript `$I$’ and subscript `$n$' indicate that the value is an interpolated one at the time $t_n$. Specifically, $q_{h,n}^I$ is calculated from the space-time slab $\Omega \times [t_n-\Delta t,t_n]$ as 
\begin{equation} \label{eq:TLagInterp}
  q_{h,n}^I = 
 \sum_{i=1}^N \phi_i(1) q_{h,i}, \quad
 \text{with} \ \phi_i(\tau) = \prod_{j=1, j\neq i}^N \frac{\tau - \tau_j}{\tau_i - \tau_j},
\end{equation}
where $\tau_i \in [-1,1]$, $i=1, \ldots ,N$, are the temporal solution points in the standard space-time element.

%
%
%
%

Note that when conducting spatial accuracy measurement (i.e. convergence tests via spatial grid refinement), very high-order accurate temporal discretization, such as the $14^{\text{th}}$-order temporal scheme in this study, is used to ensure that the spatial errors dominate in numerical experiments. Similarly, when conducting temporal accuracy measurement (i.e. convergence tests via time step refinement), very high-order accurate spatial discretization, such as the $14^{\text{th}}$-order spatial scheme in this study, is used to ensure that the temporal errors dominate.  

We first summarize several key features observed from numerical tests of the STFR method on stationary grids~\cite{Yu_Space_Time_2017}. This sets up the cornerstone when explaining numerical experiment results from STFR simulation on curvilinear moving/deformable grids.

\subsection{Stationary domain results summary}
\label{subsec:Stationary}
This section presents some key convergence test results for simulations of the 1D linear wave propagation and 2D Euler vortex propagation on stationary grids.
We mention that on stationary grids tessellated with regular uniform quadrilateral elements, where the Jacobian and metrics are constant, all STFR schemes with Gauss--Legendre solution points satisfy \textbf{Theorem}~\ref{pro:sta}.  In Figure~\ref{fig:1DLinear_Sta}, the spatial convergence rates using $P^1$ to $P^5$ spatial constructions, and the temporal convergence rates using $P^1$ to $P^4$ temporal constructions are presented. We observe that the spatial convergence rate reaches its optimal value, i.e. the  ($k+1$)th order is achieved for a degree $k$ polynomial construction; see Figure~\ref{subfig:Sta_space_conv}. 
As shown in Figure~\ref{subfig:Sta_time_conv}, superconvergence with the ($2k+1$)th order for a degree $k$ temporal polynomial construction shows up in temporal convergence tests. 
We mention again that a space-time tensor product operation is used to construct the FR formulation, and the Gauss--Legendre quadrature points are used as solution points both in space and time (see Figure~\ref{fig:ST_SF}). 
As has been proved by Huynh~\cite{Huynh_JSC_2023}, in the temporal direction, the FR scheme with the Gauss--Legendre solution points is equivalent to the so-called IRK DG-Gauss scheme when the quadrature rule based on the solution points (i.e. quadrature points used in DG) is sufficiently accurate to integrate the space-time curvilinear elements.
Since the space-time discretization in the time dimension with a degree $k$ polynomial construction is equivalent to the IRK DG-Gauss scheme with ($k+1$) stages, the order of accuracy of which is ($2k+1$), the temporal superconvergence rate of the STFR scheme agrees with its theoretical value; see Section 4.4 in~\cite{Huynh_JSC_2023}.

\begin{figure}[!htbp]
  \centering
  \subfloat[Spatial Convergence]{\includegraphics[width=0.45\textwidth]{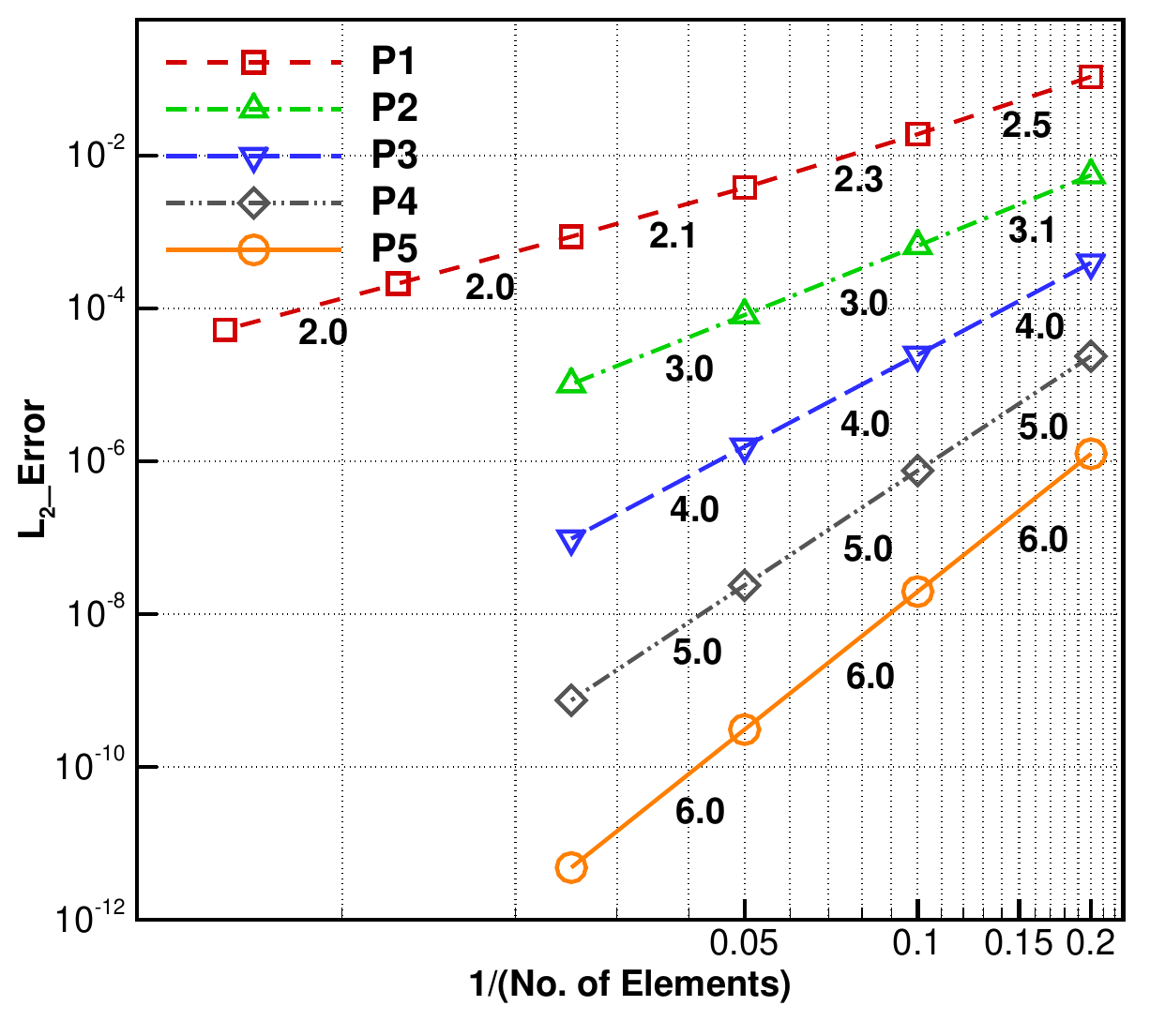}
  \label{subfig:Sta_space_conv}
  }
\hspace{0.2em}
  \subfloat[Temporal Convergence]{\includegraphics[width=0.45\textwidth]{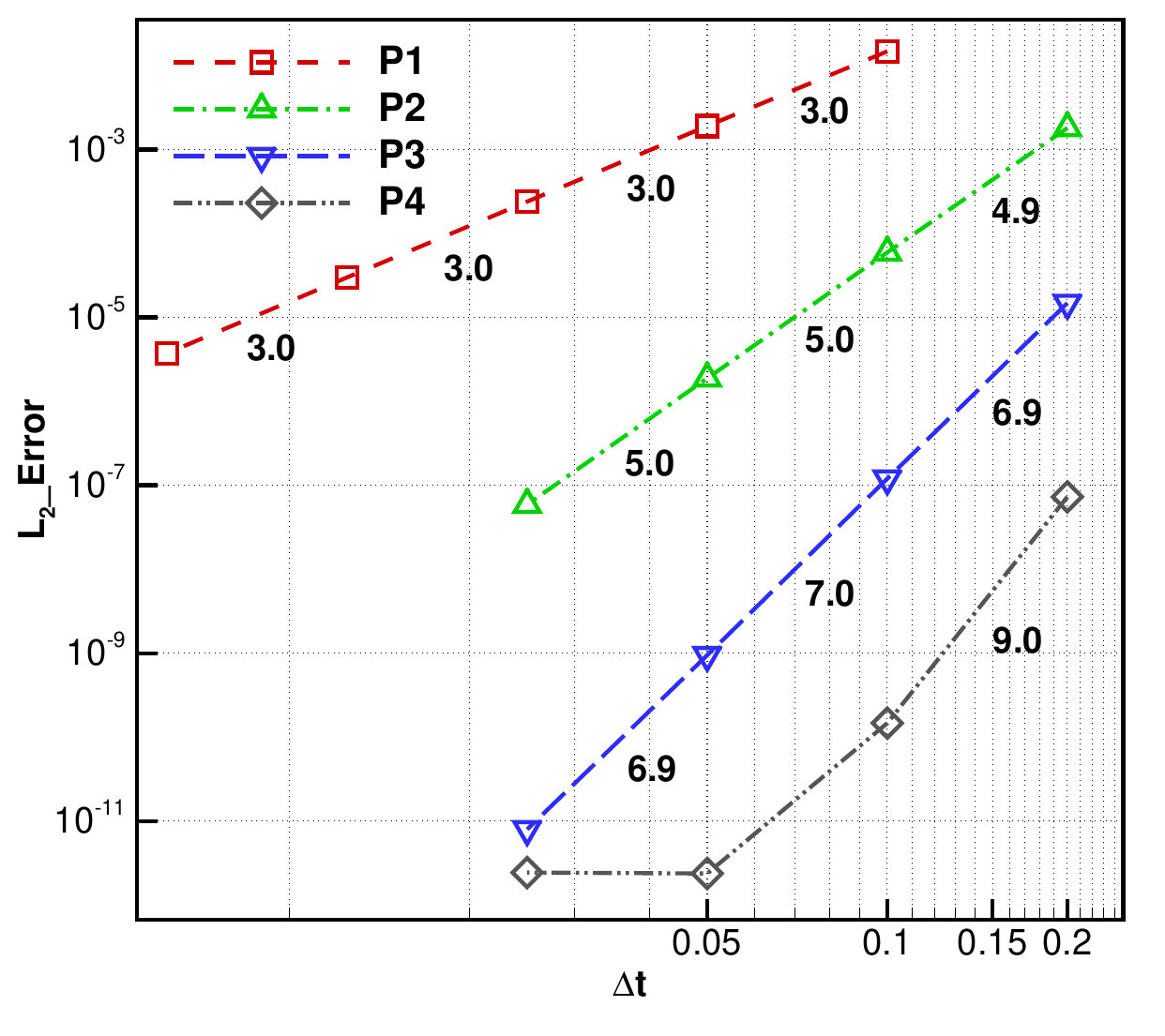}
  \label{subfig:Sta_time_conv}
  }
\caption{(a) Spatial convergence rates from $P^1$ to $P^5$ spatial constructions; and (b) temporal convergence rates from $P^1$ to $P^4$ temporal constructions for the 1D wave propagation problem on a stationary grid at $t=1$ (i.e. one wave propagation period).}
  \label{fig:1DLinear_Sta}
\end{figure}

The spatial and temporal convergence rates for simulations of the Euler vortex propagation (see the inviscid vortex setup in~\cite{Yu_Space_Time_2017}) were also studied with $\mathbb{Q}^2$ to $\mathbb{Q}^5$ spatial constructions and $P^1$ to $P^3$ temporal constructions. Results are presented in Figure~\ref{fig:2DEuler_Sta}. Similar to the 1D linear wave propagation tests, we observe that the spatial convergence rates match their optimal values, i.e. the  ($k+1$)th order for a degree $k$ polynomial construction, and the temporal convergence rates show the theoretical superconvergent features, i.e. the  ($2k+1$)th order for a degree $k$ polynomial construction. This demonstrates that spatial nonlinearity from the governing PDEs does not affect the temporal superconvergence of the STFR schemes. Therefore, we will use 2D linear wave propagation to test the curvilinear space-time effects of moving domains on the convergence of STFR schemes in following sections.

\begin{figure}[!htbp]
  \centering
  \subfloat[Spatial Convergence]{\includegraphics[width=0.45\textwidth]{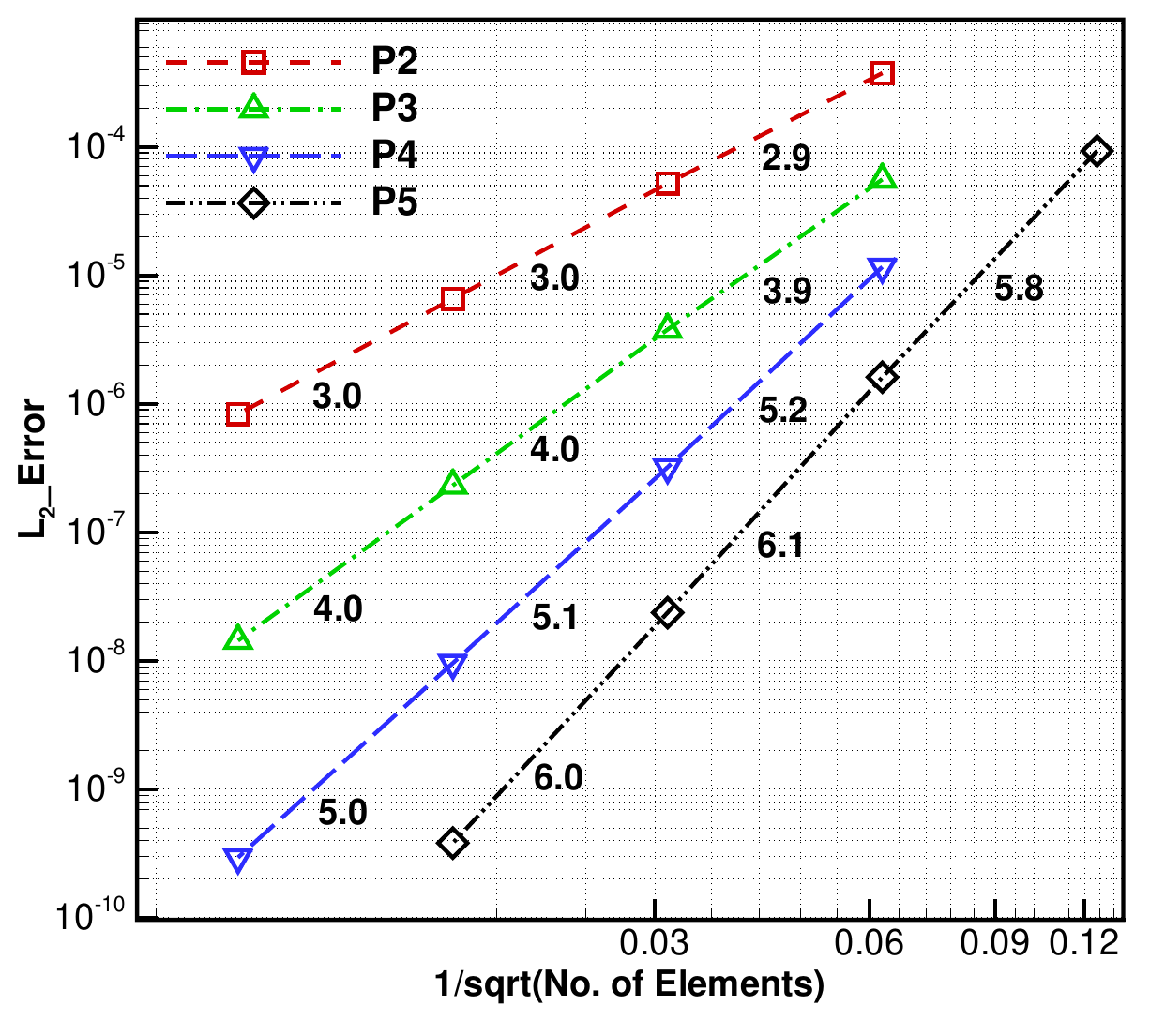}}\hspace{0.2em}
\hspace{0.2em}
  \subfloat[Temporal Convergence]{\includegraphics[width=0.45\textwidth]{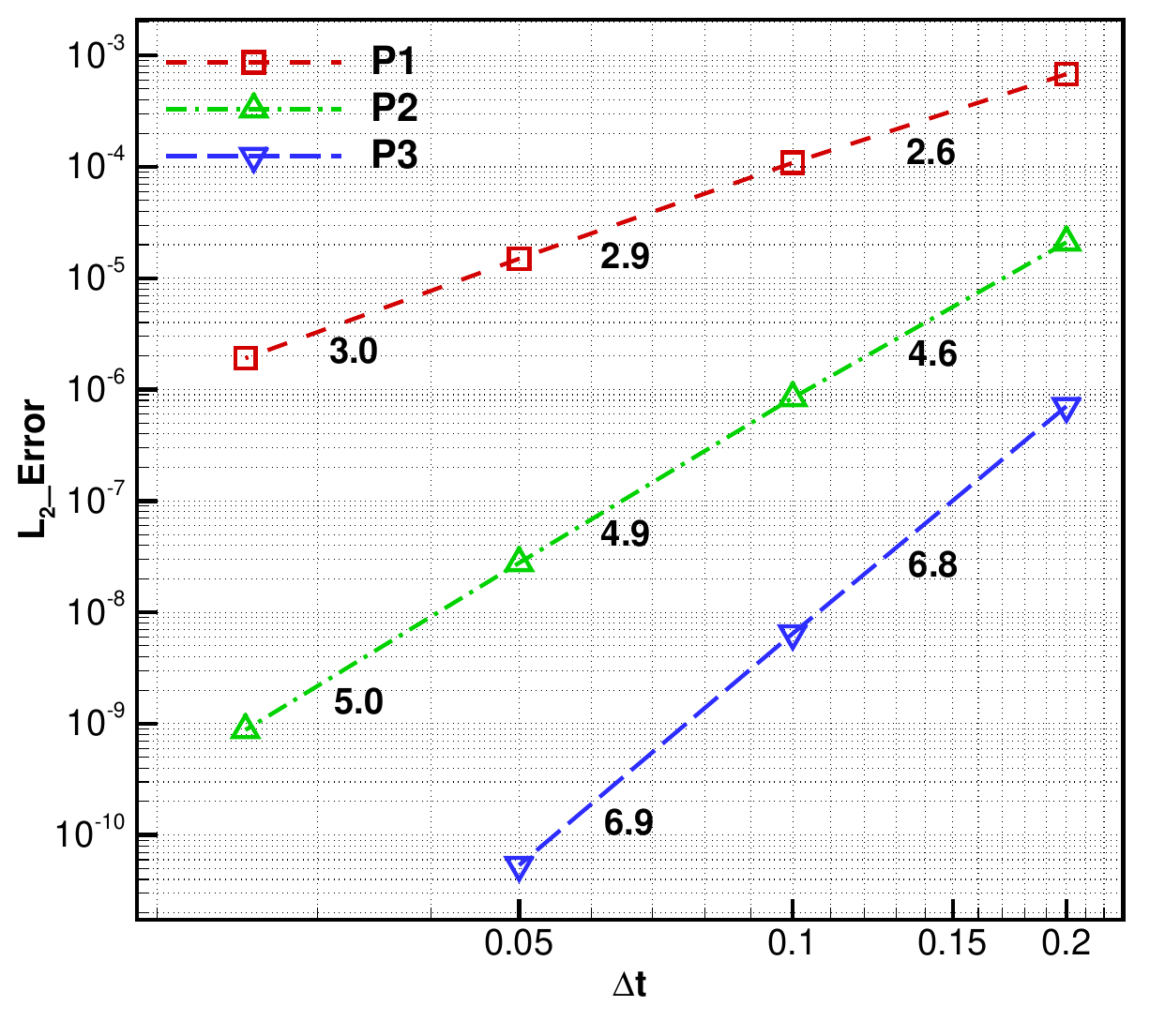}}
    \caption{(a) Spatial convergence rates from $\mathbb{Q}^2$ to $\mathbb{Q}^5$ spatial constructions for the 2D Euler vortex propagation problem on a stationary grid at the 1/8 wave propagation period; and (b) temporal convergence rates from the $P^1$ to $P^3$ temporal constructions evaluated at one wave propagation period.}
  \label{fig:2DEuler_Sta}
\end{figure}

\begin{Observation}
With Gauss--Legendre quadrature points serving as spatiotemporal solution points, the nodal STFR formulation shows a ($2k+1$)th order superconvergence rate for a degree $k$ polynomial construction (i.e. with $k+1$ solution points) in the time dimension on stationary grids, no matter for linear or nonlinear hyperbolic conservation laws with adequate smooth features.
\end{Observation}

\subsection{Tensor-product space-time elements with linear spatial and temporal geometric representations in a moving domain}
\label{subsec:MovingLinear}
Starting from this section, numerical schemes will be termed with three labels, i.e. `\textbf{S}' representing those satisfying \textbf{Theorem}~\ref{pro:moving}, `\textbf{P}' representing those satisfying the \textbf{Proposition}~\ref{pro:moving_reduced} but violating \textbf{Theorem}~\ref{pro:moving}, and `\textbf{V}' representing those violating the \textbf{Proposition}~\ref{pro:moving_reduced} (and \textbf{Theorem}~\ref{pro:moving}). 
We mention that when Eq.~\eqref{eq:STFR_new_phy} is implemented with \textbf{Key Procedures}~\ref{alg:2DFR_key_new}, a scheme labeled `\textbf{V}' that violates \textbf{Proposition}~\ref{pro:moving_reduced} can still preserve the freestream. 

Now we examine which accuracy label, i.e. `\textbf{S}', `\textbf{P}', or `\textbf{V}', the numerical schemes with different nominal spatial approximation degree $k$ and temporal approximation degree $m$ should have when linear spatial and temporal representations are used to describe the grid element. 
According to the polynomial space used in Table~\ref{Tab:ST_Variables}, we have $l=1$ and $n=1$ for linear space-time elements. Thus, for $k=2$, $3$, and $4$ (i.e. nominal $3^{rd}$-, $4^{th}$-, and $5^{th}$-order spatial constructions), and $m=1$ and $2$ (i.e. nominal $2^{nd}$- and $3^{rd}$-order temporal constructions with the corresponding $3^{rd}$- and $5^{th}$-order superconvergence, respectively), the accuracy labels of the STFR schemes are summarized in Table~\ref{Tab:Bilinear}. 
Since $\widetilde{\vect{Q}}$ is used as the hidden working variable in the STFR schemes for general curvilinear grids, different numbers of solution points can be used to construct the physical quantity $\vect{Q}$, depending on the desired accuracy of the hidden working variable $\widetilde{\vect{Q}}$. Note that as a rule of thumb, the order of accuracy of a numerical scheme is measured by the number of solution points used to construct the explicit working variable, i.e. the solution $\vect{Q}$ here. However, depending on the nominal order of accuracy, which depends on the expected spatial approximation degree $k$, we can label the scheme with `\textbf{S}', `\textbf{P}', or `\textbf{V}'. Take ``\# of SPs" equal to 4 in Table~\ref{SubTab:Linear_Space} as an example. If \textbf{Theorem}~\ref{pro:moving} is strictly satisfied (i.e. with label `\textbf{S}'), then the nominal order of accuracy for the solution $\vect{Q}$ should be 3; if only \textbf{Proposition}~\ref{pro:moving_reduced} needs to be satisfied, then the nominal order of accuracy for the solution $\vect{Q}$ can be up to 4. In practical numerical tests, the convergence rate can be between 3 and 4 when \textbf{Key Procedures}~\ref{alg:2DFR_key_new} is used to implement STFR. The filtering operation in \textbf{Key Procedures}~\ref{alg:filter_proj} can then be used to stabilize the convergence rate to 3; see numerical experimental results presented in \textbf{Sect.}~\ref{subsubsec:SD_deform}.           
\begin{table}[!htb]
\caption{Scheme labels based on the nominal order of accuracy when linear space-time elements, i.e. $l=1$ and $n=1$, are used to represent the moving/deformable grids. Herein, ``OoA" stands for ``Order of Accuracy", ``SPs" stands for ``Solution Points", and ``\#" indicates ``number".}
\label{Tab:Bilinear}
\centering
\subfloat[Spatial Discretizations] {
\label{SubTab:Linear_Space}
\begin{tabular}{|cccc|}
\hline
\multicolumn{4}{|c|}{Desired OoA = \# of SPs}                                                                                                                           \\ \hline
\multicolumn{1}{|c|}{$k$}                & \multicolumn{1}{c|}{Nominal OoA}                                                              & \multicolumn{1}{c|}{\# of SPs} & Label \\ \hline
\multicolumn{1}{|c|}{\multirow{2}{*}{2}} & \multicolumn{1}{c|}{\multirow{2}{*}{3}}                                                       & \multicolumn{1}{c|}{3}         & `P'   \\ \cline{3-4} 
\multicolumn{1}{|c|}{}                   & \multicolumn{1}{c|}{}                                                                         & \multicolumn{1}{c|}{4}         & `S'   \\ \hline
\multicolumn{1}{|c|}{\multirow{2}{*}{3}} & \multicolumn{1}{c|}{\multirow{2}{*}{4}}                                                       & \multicolumn{1}{c|}{4}         & `P'   \\ \cline{3-4} 
\multicolumn{1}{|c|}{}                   & \multicolumn{1}{c|}{}                                                                         & \multicolumn{1}{c|}{5}         & `S'   \\ \hline
\multicolumn{1}{|c|}{\multirow{2}{*}{4}} & \multicolumn{1}{c|}{\multirow{2}{*}{5}}                                                       & \multicolumn{1}{c|}{5}         & `P'   \\ \cline{3-4} 
\multicolumn{1}{|c|}{}                   & \multicolumn{1}{c|}{}                                                                         & \multicolumn{1}{c|}{6}         & `S'   \\ \hline
\end{tabular}
}
\quad
\subfloat[Temporal Discretizations] {
\label{SubTab:Linear_Time}
\begin{tabular}{|cccc|}
\hline
\multicolumn{4}{|c|}{\begin{tabular}[c]{@{}c@{}}Desired OoA with superconvergence \\ = 2 $\times$ (\# of SPs) $-$ 1\end{tabular}}                                                                                                                    \\ \hline
\multicolumn{1}{|c|}{$m$}                & \multicolumn{1}{c|}{\begin{tabular}[c]{@{}c@{}}Nominal \\ superconvergence\end{tabular}} & \multicolumn{1}{c|}{\# of SPs} & Label \\ \hline
\multicolumn{1}{|c|}{\multirow{3}{*}{1}} & \multicolumn{1}{c|}{\multirow{3}{*}{3}}                                                       & \multicolumn{1}{c|}{2}         & `V'   \\ \cline{3-4} 
\multicolumn{1}{|c|}{}                   & \multicolumn{1}{c|}{}                                                                         & \multicolumn{1}{c|}{3}         & `P'   \\ \cline{3-4} 
\multicolumn{1}{|c|}{}                   & \multicolumn{1}{c|}{}                                                                         & \multicolumn{1}{c|}{4}         & `S'   \\ \hline
\multicolumn{1}{|c|}{\multirow{3}{*}{2}} & \multicolumn{1}{c|}{\multirow{3}{*}{5}}                                                       & \multicolumn{1}{c|}{3}         & `P'   \\ \cline{3-4} 
\multicolumn{1}{|c|}{}                   & \multicolumn{1}{c|}{}                                                                         & \multicolumn{1}{c|}{4}         & `P'   \\ \cline{3-4} 
\multicolumn{1}{|c|}{}                   & \multicolumn{1}{c|}{}                                                                         & \multicolumn{1}{c|}{5}         & `S'   \\ \hline
\end{tabular}
}
\end{table}

This section shows the convergence test results for simulations of the 2D linear wave propagation on moving/deformable grids with the linear space-time element representation. 
Two types of problems are considered: 
(1) deformable grids with spatial and temporal symmetry; and (2) moving and deformable grids without spatial and temporal symmetry. Starting from this section, we use the number of solution points (SPs) employed in STFR schemes for solution construction as the measure of desired order of accuracy. Note that the desired order of accuracy is always equal to or greater than the nominal one.  

%
%

\subsubsection{Deformable grids with spatial and temporal symmetry}
\label{subsubsec:SD_deform}
In this subsection, numerical tests are performed with linear spatial and temporal representations of deformable grids with spatial and temporal symmetry. The grid deformation strategy is given as follows~\cite{Mavriplis_Nastase_AIAA_2008}:
\begin{equation}
\Biggl\{
\begin{array}{l}
dx(\vect{x},t)=A_x L_x  \text{sin}(\omega_t t) \text{sin}(\omega_x x) \text{sin}(\omega_y y) dt⁄t_{max},\\
dy(\vect{x},t)=A_y L_y  \text{sin}(\omega_t t) \text{sin}(\omega_x x) \text{sin}(\omega_y y) dt⁄t_{max}.
\end{array}
\label{eq:mesh_motion_sym_def}
\end{equation}
Herein, $dx$ and $dy$ are the grid displacement at the time $t$ in the $x$- and $y$-directions, $dt$ is the temporal step size, $L_x$, $L_y$ and $t_{max}$ are reference values in the $x$-, $y$- and $t$-directions, and $A_x$ and $A_y$ are scaling factors to control the amplitudes of grid displacement in the $x$- and $y$-directions. The temporal and spatial angular frequencies are defined as
\[
\omega_t=(n_t \pi)⁄t_{max}, \ \omega_x=(n_x \pi)⁄L_x, \ \omega_y=(n_y \pi)⁄L_y.
\]
In this study, the parameters are set as $A_x=A_y=0.1$, $L_x=L_y=1.0$, $n_t=0.5$, $n_x=n_y=4.0$, and $t_{max}=0.2$. The computational domain is set as $[0,1] \times [0,1]$, and all simulations are run until $t=t_{max}=T/4$. The wave field on top of the corresponding deformed grid tessellated with $64 \times 64$ elements for the 2D linear wave propagation problem at $t=t_{max}$ is displayed in Figure~\ref{subfig:Sin_deform_flow} for visualization. 
Since bilinear polynomials are used to represent grid elements, as shown in Figure~\ref{subfig:Sin_deform_mesh_comp}, where a grid with $16 \times 16$ elements is plotted on top of that with $8 \times 8$ elements, the two sets of grids used in the spatial convergence tests do not match each other. This issue also exists in the time dimension, and can affect the spatiotemporal convergence rates as will be shown shortly.

\begin{figure}[!htbp]
  \centering
  \subfloat[]{\includegraphics[width=0.45\textwidth]{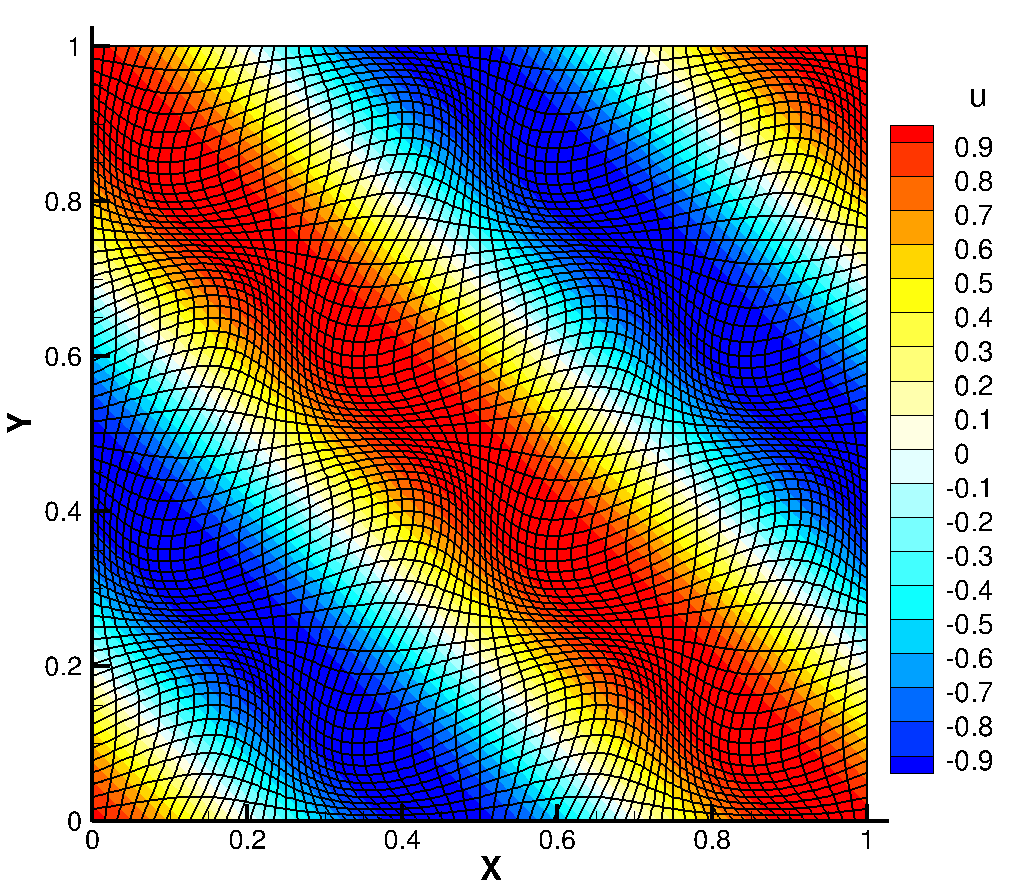}
  \label{subfig:Sin_deform_flow}
  }
\hspace{0.2em}
  \subfloat[]{\includegraphics[width=0.38\textwidth]{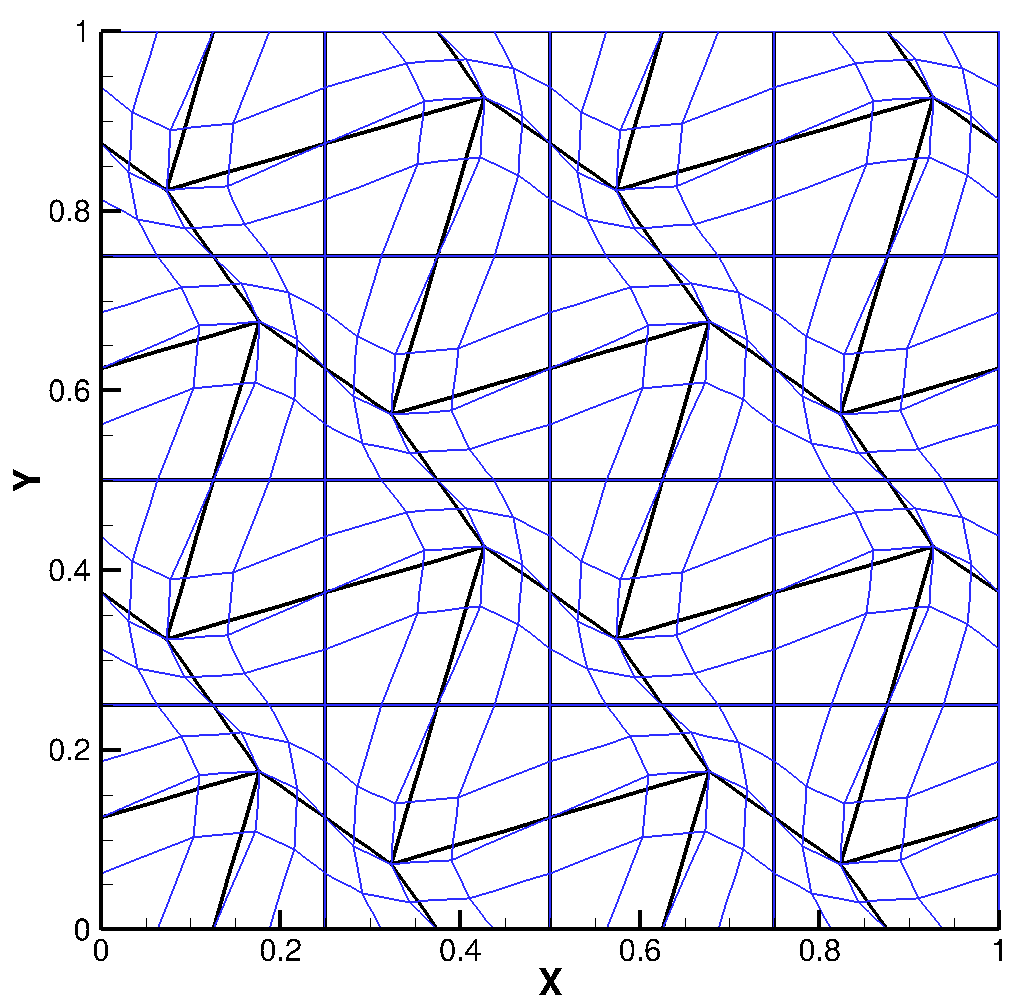}
  \label{subfig:Sin_deform_mesh_comp}
  }
    \caption{(a) The wave field and the corresponding deformed grid at $t=t_{max}=0.2$ for the 2D linear wave equation. (b) The grid with $16 \times 16$ bilinear elements plotted on top of that with $8 \times 8$ bilinear elements.}
  \label{fig:Sin_deform}
\end{figure}

We first present the spatial convergence rates for spatial constructions using 3 to 6 solution points in each spatial dimension in Figure~\ref{fig:DS_bilinear_space}. We observe that the order of accuracy shows severe deterioration at the coarser grid side, e.g. grids with $8 \times 8$ and $16 \times 16$ elements, and gradually recovers towards the optimal value (i.e. the number of solution points in each spatial dimension) when the grid becomes finer. The convergence rate deterioration over coarser grids is expected to be fixed when grid representation errors are reduced when curvilinear elements are adopted. This will be further discussed in \textbf{Sect.}~\ref{subsec:MovingCurved}.   

\begin{figure}[!htbp]
  \centering
  \subfloat[Spatial Convergence]{\includegraphics[width=0.45\textwidth]{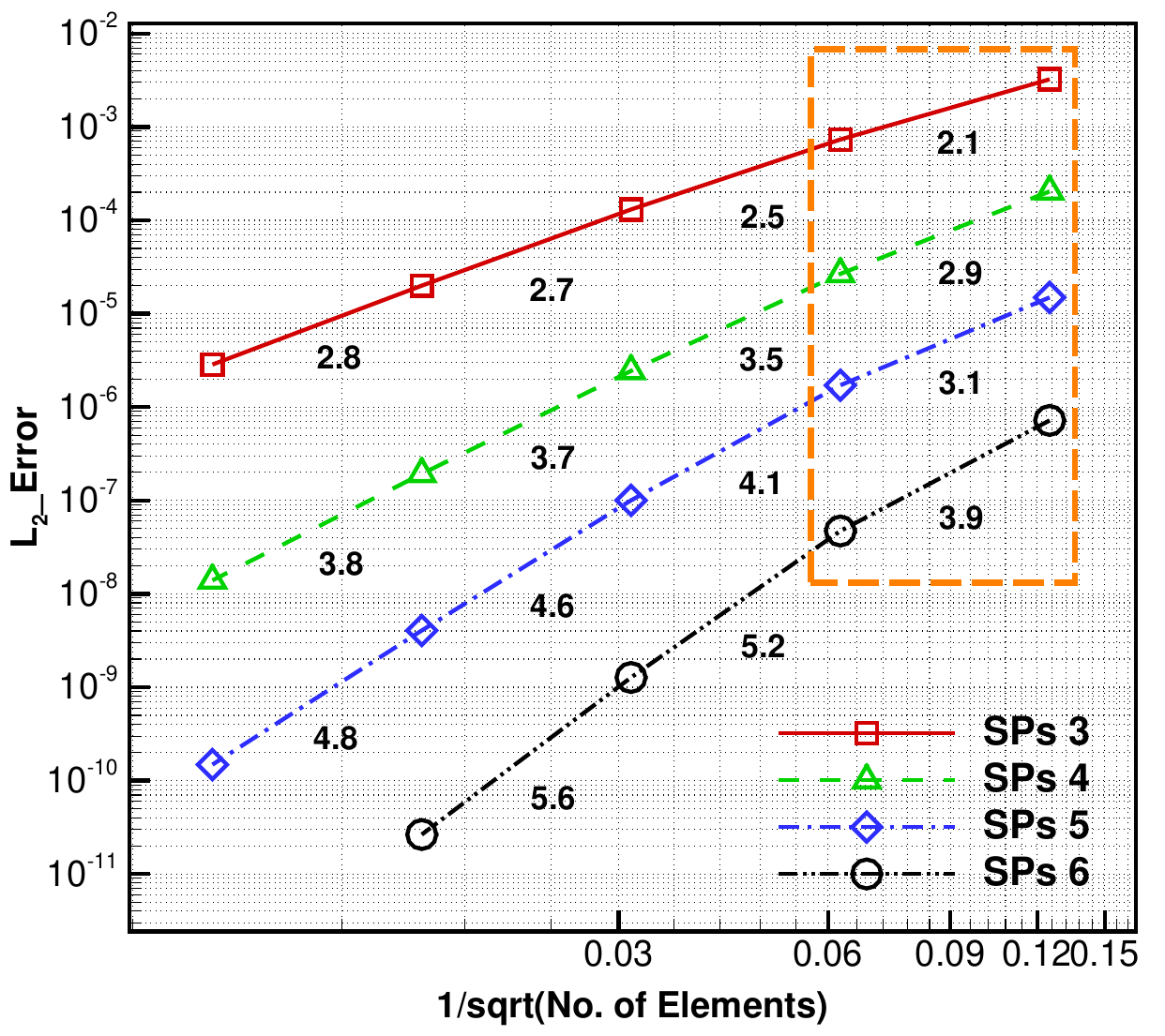}
  \label{fig:DS_bilinear_space}}
\hspace{0.2em}
  \subfloat[Temporal Convergence]{\includegraphics[width=0.45\textwidth]{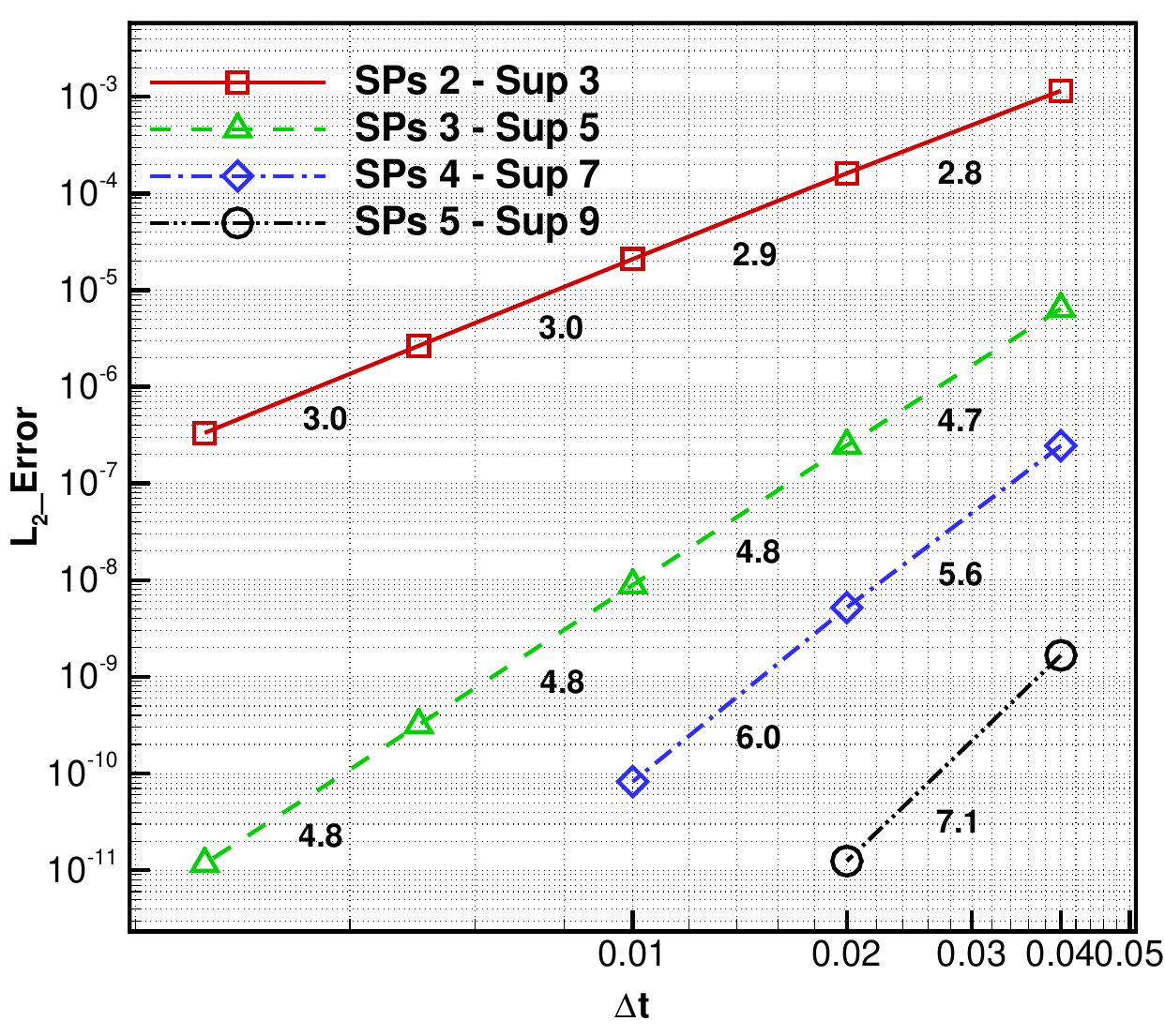}
  \label{fig:DS_bilinear_time}}
    \caption{(a) Spatial convergence rates for spatial constructions using 3 to 6 solution points in each spatial dimension; and (b) temporal convergence rates for temporal constructions using 2 to 5 solution points in the time dimension for the 2D linear wave equation on a deforming square domain tessellated with linear space-time elements.}
  \label{fig:DS_bilinear}
\end{figure}

The temporal convergence rates are shown in Figure~\ref{fig:DS_bilinear_time}. We observe that the schemes with 2 and 3 solution points in the time dimension show the same superconvergence rate as that on stationary grids; specifically, with $j$ temporal solution points, a $(2j-1)$th order superconvergence shows up. Note that based on Table~\ref{Tab:Bilinear}, the scheme with 2 temporal solution points violates the \textbf{Proposition}~\ref{pro:moving_reduced} as well as \textbf{Theorem}~\ref{pro:moving}. However, this does not cause convergence rate deterioration. As will be reported in following sections, similar phenomena have been observed in cases with stronger nonlinearity. This contributes to the good numerical property of the STFR schemes when Eq.~\eqref{eq:STFR_new_phy} is implemented with \textbf{Key Procedures}~\ref{alg:2DFR_key_new}; see similar observations reported in~\cite{YU201470}.  
For schemes with 4 and 5 temporal solution points, the superconvergence rates do not reach the desired optimal values. Due to the very small absolute errors (i.e. smaller than $10^{-9}$) possessed by this deforming grid problem, it is hard to probe the exact numerical issues. But as will be seen from the deforming grid cases without temporal symmetry presented in \textbf{Sect.}~\ref{subsubsec:Deform_asym_bilinear} and~\ref{subsec:MovingCurved}, the superconvergence rate deterioration can be due to the grid representation error. Evidence
will be provided later when numerical errors are sufficiently large so that a group of superconvergence rates can be compared in a meaningful manner.

Note that all schemes presented in Figure~\ref{fig:DS_bilinear} do not explicitly apply the aliasing error control. As discussed in \textbf{Sect.}~\ref{subsubsec:new_discuss}, we can use projection-based polynomial filtering to control aliasing errors caused by the nonlinear interaction between flow quantities and the curvilinear geometric representation of space-time elements, thus stabilizing the convergence rates. We first present results of spatial polynomial filtering with different filtering strengths in Figure~\ref{fig:DS_Space_Filter_bilinear}.  Recall that with the parameter $\theta$, a $(1-\theta^2)$ portion of the energy in the solution between the original and projection spaces is filtered out. From Figure~\ref{fig:DS_Space_Filter_bilinear}, when $1 \%$ (i.e. $\theta^2 = 0.99$ ) energy is filtered out, the solution is not significantly affected; when a nontrivial portion, such as more than $10 \%$, of energy is filtered out, the convergence rate of the scheme can be stabilized at the value possessed by the projection space.  
This further explains the practical meaning of the scheme labels presented in Table~\ref{SubTab:Linear_Space}. As explained early in \textbf{Sect.}~\ref{subsec:MovingLinear}, and still taking ``\# of SPs" of 4 as an example, when the expected nominal order of accuracy is 4, the scheme is labeled with `\textbf{P}'. As shown in Figure~\ref{fig:DS_Space_Filter_bilinear}, its actual convergence rate without filtering is 3.8, close to 4. However, when the expected nominal order of accuracy is 3, the scheme is labeled with `\textbf{S}'. Apparently, the actual convergence rate of the solution is greater than 3 (or in another way, at least 3, which meets the order of accuracy expectation according to the conditions to meet \textbf{Theorem}~\ref{pro:moving}), and with polynomial filtering of considerable strength, the convergence rate can be stabilized at 3.      

\begin{figure}[]
  \centering
  \subfloat[SPs 4]{\includegraphics[width=0.45\textwidth]{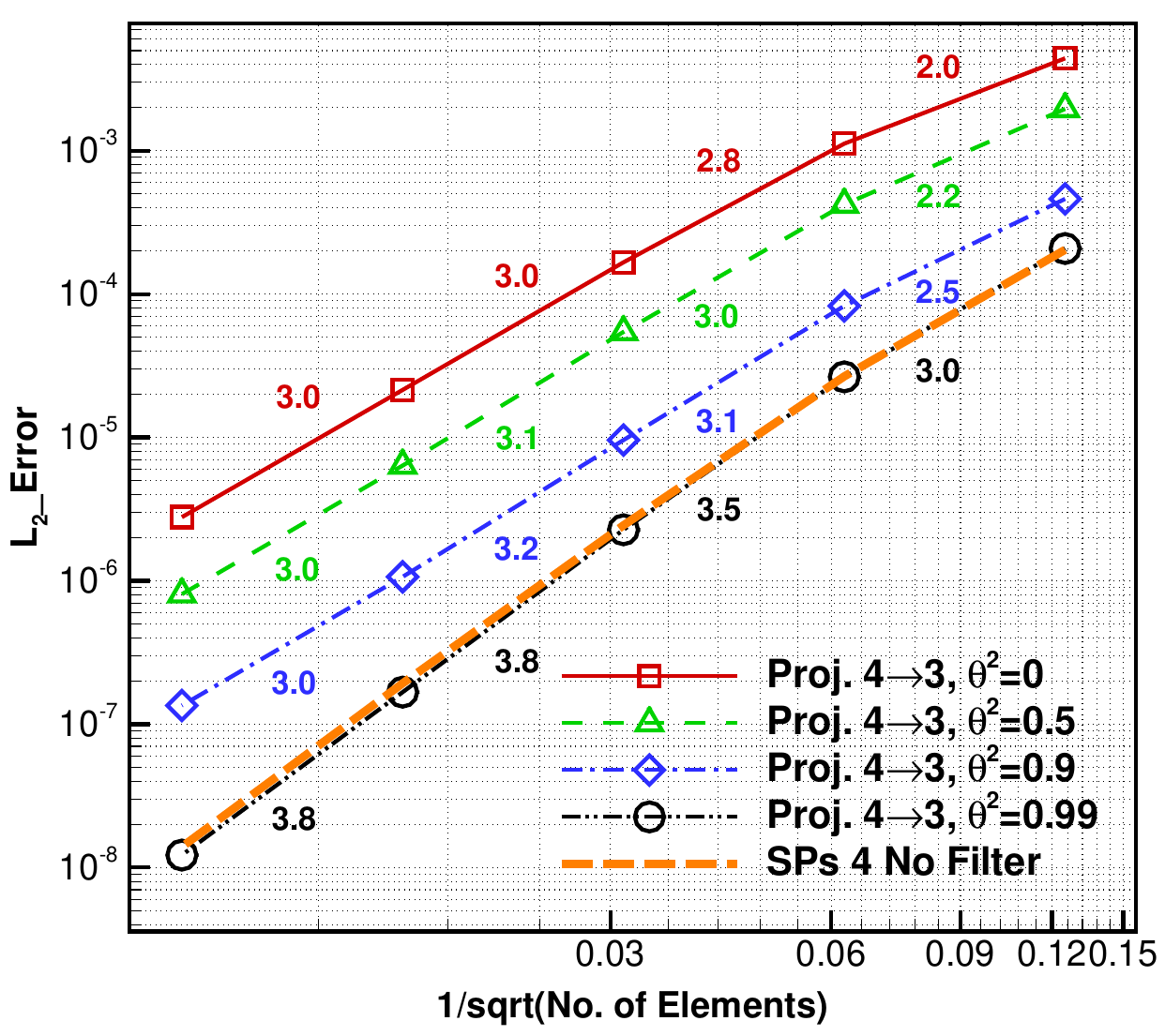}}\hspace{0.2em}
\hspace{0.2em}
  \subfloat[SPs 5]{\includegraphics[width=0.45\textwidth]{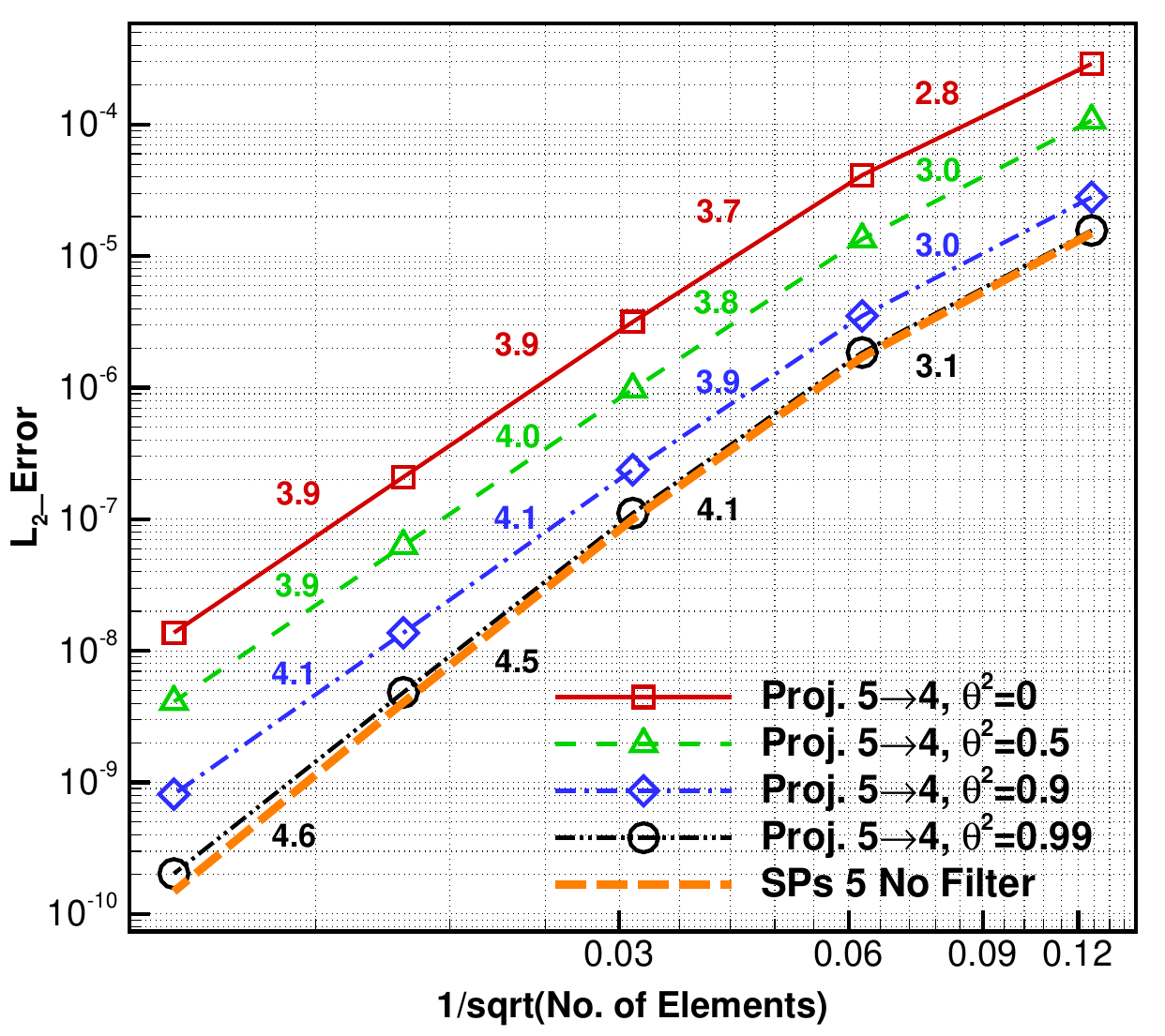}}
    \caption{Spatial convergence rates for spatial constructions using (a) 4 and (b) 5 solution points in each spatial dimension with different filtering strengths.}
  \label{fig:DS_Space_Filter_bilinear}
\end{figure}

When filtering is applied in the time dimension, we find that the superconvergence property of the temporal schemes is lost. As shown in Figure~\ref{fig:DS_Time_Filter_bilinear}, when a degree $j-1$ temporal polynomial construction with $j$ solution points is projected onto the space sit by degree $j-2$ polynomials, the local convergence rate is $j-1$; due to the error accumulation in time marching (i.e. when the time step size is halved, the simulation step doubles), the global convergence rate of the filtered scheme becomes $j-2$. It is clear that all schemes tested with different filtering strengths show the theoretical global convergence rate. Since the convergence rate of the filtered scheme drops from the superconvergent one $2j-1$ to the regular global value $j-2$ when $j$ temporal solution points are used, and the absolute error of the filtered scheme can be significantly larger than that of the original scheme, temporal polynomial filtering is not recommended even if a scheme cannot reach its theoretical superceonvergence rate, such as the schemes with 4 or 5 temporal solution points. 

\begin{figure}[]
  \centering
  \subfloat[SPs 3]{\includegraphics[width=0.32\textwidth]{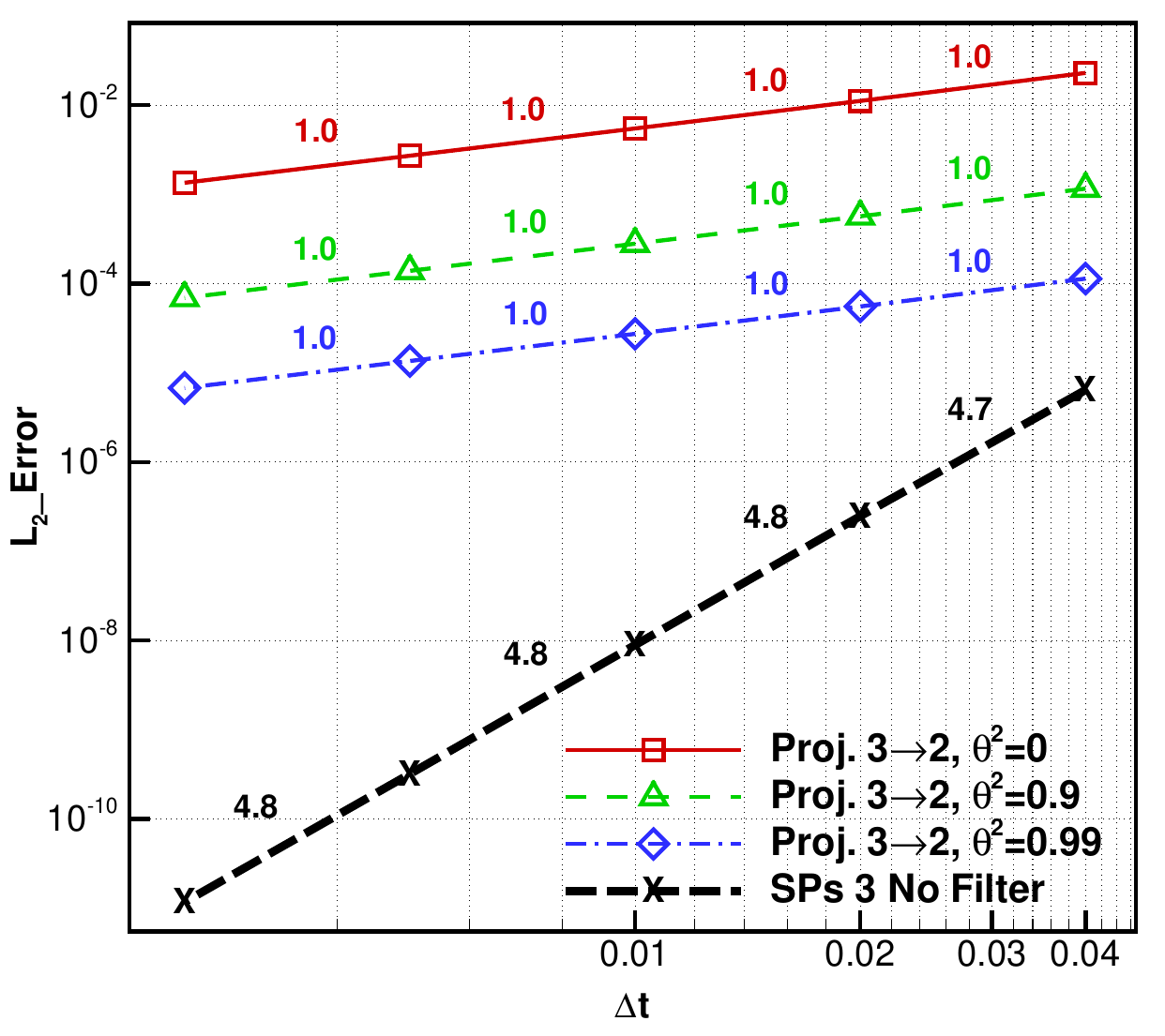}}\hspace{0.2em}
\hspace{0.2em}
  \subfloat[SPs 4]{\includegraphics[width=0.32\textwidth]{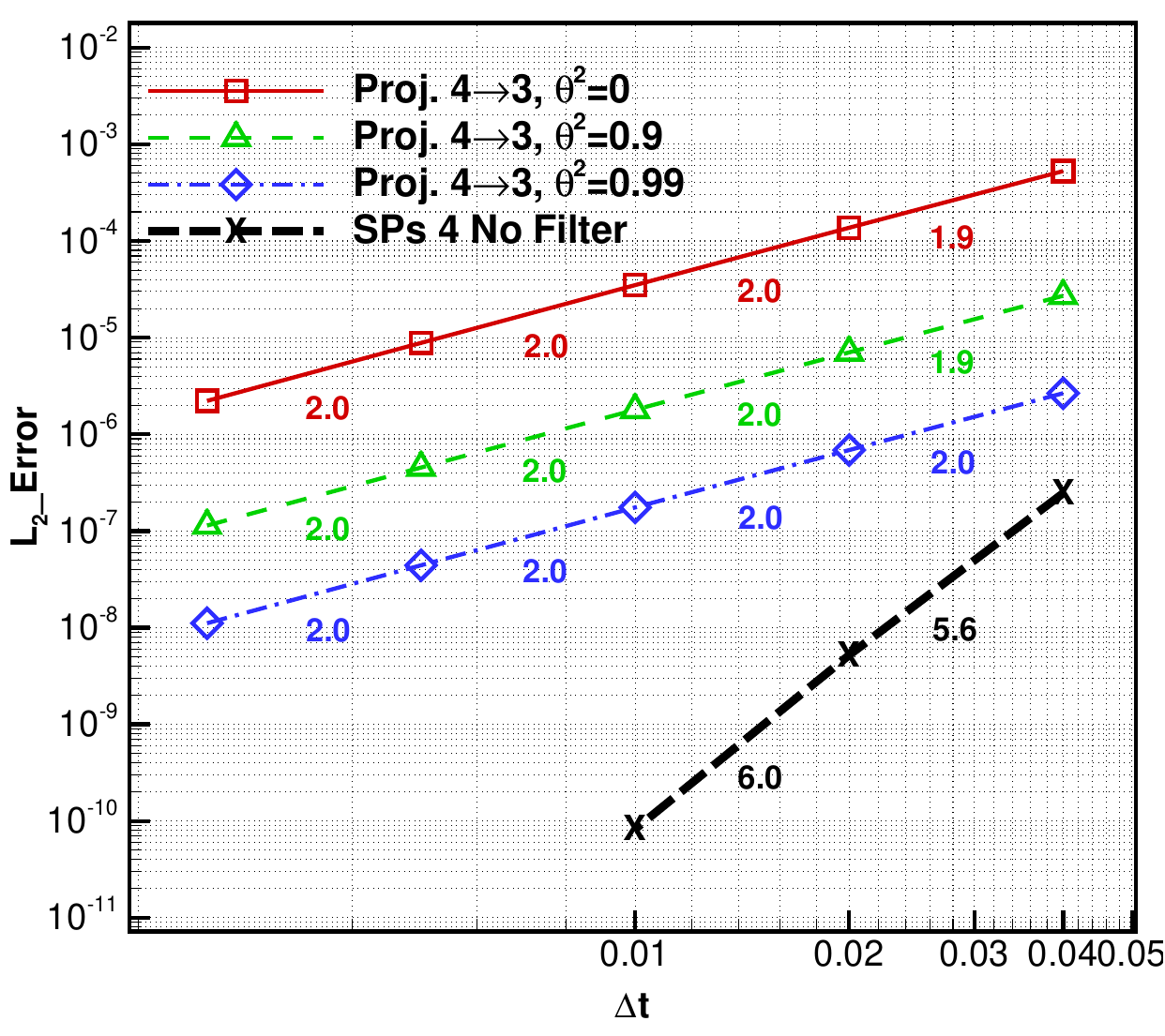}}
\hspace{0.2em}
  \subfloat[SPs 5]{\includegraphics[width=0.32\textwidth]{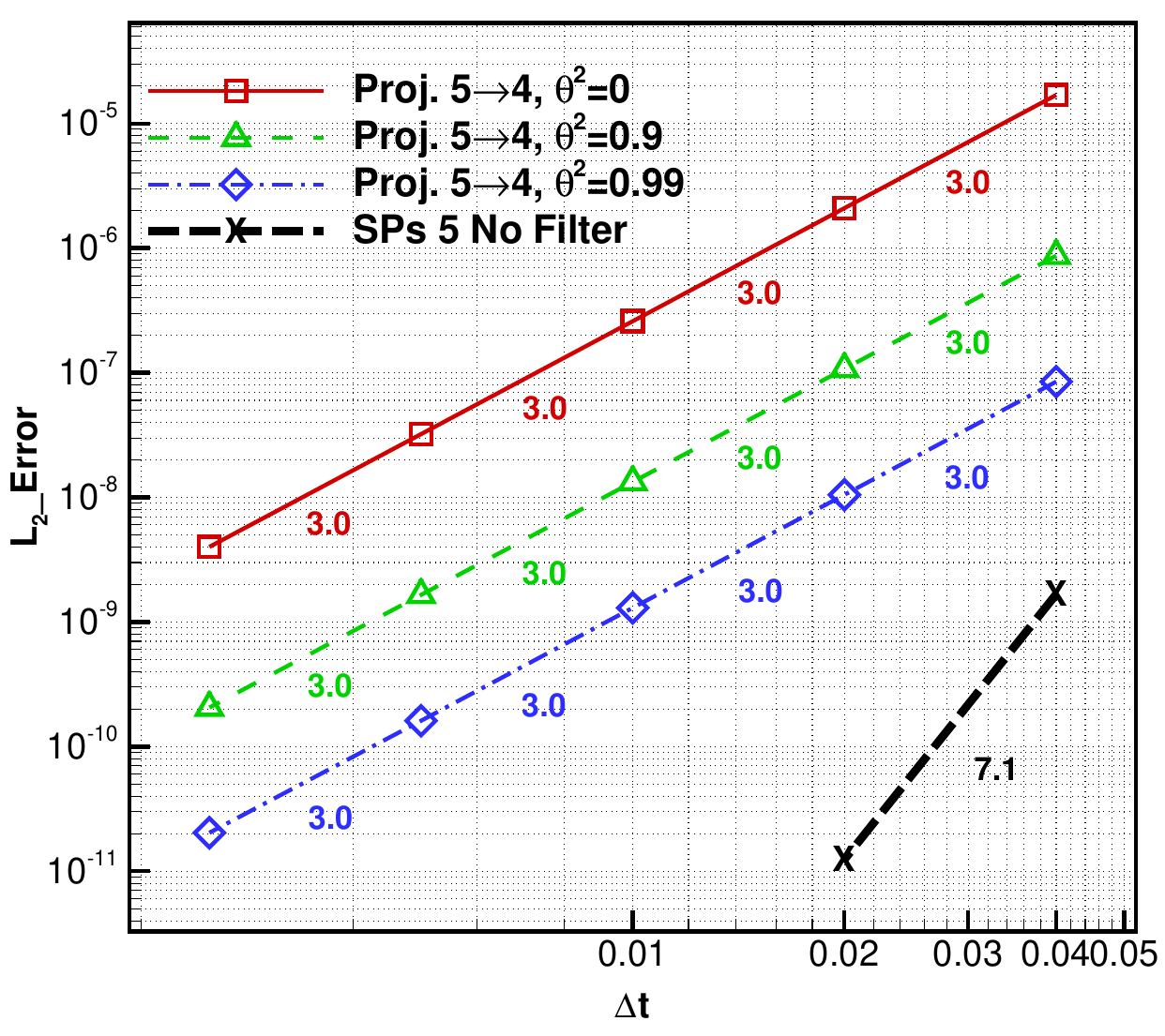}}
    \caption{Global temporal convergence rates for temporal constructions using (a) 3, (b) 4, and (c) 5 solution points in the time dimension with different filtering strengths. The convergence rate of the filtered scheme drops from the superconvergent one $2j-1$ to the regular global value $j-2$, where $j$ is the number of temporal solution points.}
  \label{fig:DS_Time_Filter_bilinear}
\end{figure}

\subsubsection{Deformable grids without spatial and temporal symmetry}
\label{subsubsec:Deform_asym_bilinear}
To further study the convergence property of the moving grid STFR method, a moving and deforming circular domain case designed in~\cite{Wukie_EtAl_AIAA_2023} is studied in this subsection. Since the purpose of this study is to test whether the time-varying Jacobian and metrics of a moving grid with spatiotemporal symmetry-breaking large deformation can disturb the flow fields, only the 2D linear wave propagation problem with analytical boundary conditions is used for testing purposes. The grid motion within a circular domain of a radius $0.5$ centered at $(0,0)$ is given as follows: 
\begin{equation}
\begin{array}{cl}
\left (
\begin{array}{c}
      x(r_0, \theta_0, t) \\
      y(r_0, \theta_0, t)
\end{array} \right )  = &
\left (
\begin{array}{ccc}
      1 & 0 & 0 \\
      0 & 1 & \alpha (t)
\end{array} \right )  
\left (
\begin{array}{ccc}
      \text{cos} \left(A_{\theta} \alpha (t) \right) & -\text{sin} \left(A_{\theta} \alpha (t) \right)  & 0\\
      \text{sin} \left(A_{\theta} \alpha (t) \right) & \text{cos} \left(A_{\theta} \alpha (t) \right) & 0 \\
      0 & 0 & 1
\end{array} \right ) \\
&
\left (
\begin{array}{ccc}
      \psi(t) & 0 & 0\\
      0 & 1/\psi(t) & 0 \\
      0 & 0 & 1
\end{array} \right )
\left (
\begin{array}{c}
      r_0 \text{cos} \left( \theta_g (r_0, \theta_0, t) \right) \\
      r_0 \text{sin} \left( \theta_g (r_0, \theta_0, t) \right) \\
      1
\end{array} \right ).
\end{array}
\label{eq:mesh_motion_asym_def}
\end{equation}
Herein, $(r_0, \theta_0)$ is the relative position of the grid point to the origin $(0,0)$ at $t=0$ measured in the polar coordinate system. The functions $\alpha (t)$, $\psi (t)$, and $\theta_g (r_0,\theta_0,t)$ are defined as follows:
\begin{equation}
\Biggl\{
\begin{array}{l}
\alpha(t) = t^3 (8-3t) / 16\\
\psi(t) = 1 + (A_a - 1) \alpha(t)\\
\theta_g (r_0, \theta_0, t) = \theta_0 + A_g 
f_g(r_0, \theta_0, t)
\end{array},
\label{eq:func_asym_def}
\end{equation}
where the function $f_g$ is defined as
\begin{equation}
f_g (r_0,\theta_0,t) = \frac{t^6}{t^6 + 0.01} \left(16 r_0^4 + \eta (t,10,0.7) (\text{cos} (32 \pi r_0^4 )-1) \right) \eta (\theta_0, 1, 0.7)
\label{eq:fg}
\end{equation}
with the symmetry-breaking perturbation $\eta$ defined as
\[
\eta(\lambda,\omega,\tau) = \text{sin} \left (\omega \lambda + \tau \left(1-\text{cos} (\omega \lambda) \right) \right).
\]
Note that the definition of the function $f_g$ in Eq.~\eqref{eq:fg} is slightly different to that defined in~\cite{Wukie_EtAl_AIAA_2023} to ensure that there is no motion jump at $t=0$.  
In the above equations, $A_{\theta}$ is a constant rotation amplitude, $A_a$ is a constant amplification factor for the deformation of a circle into an ellipse, $A_g$ is a constant volume deformation amplitude. Following~\cite{Wukie_EtAl_AIAA_2023}, they are set as $A_{\theta} = \pi$, $A_a=1.5$, and $A_g=0.15$. In the function $\eta$, the variable $\lambda$ represents the independent variables, such as $t$ and $\theta_0$, and the variables $\omega$ and $\tau$ are the angular frequency and phase-lag factor, respectively.
The physical domain at $t=0$ and the deformed one at $t=1$, together with the corresponding linear wave fields, are shown in Figure~\ref{fig:Cylinder_Flow_Comp}. We observe that the wave field is not disturbed by the grid motion and deformation.

\begin{figure}[!htbp]
  \centering
  \subfloat[$t=0$]{\includegraphics[width=0.45\textwidth]{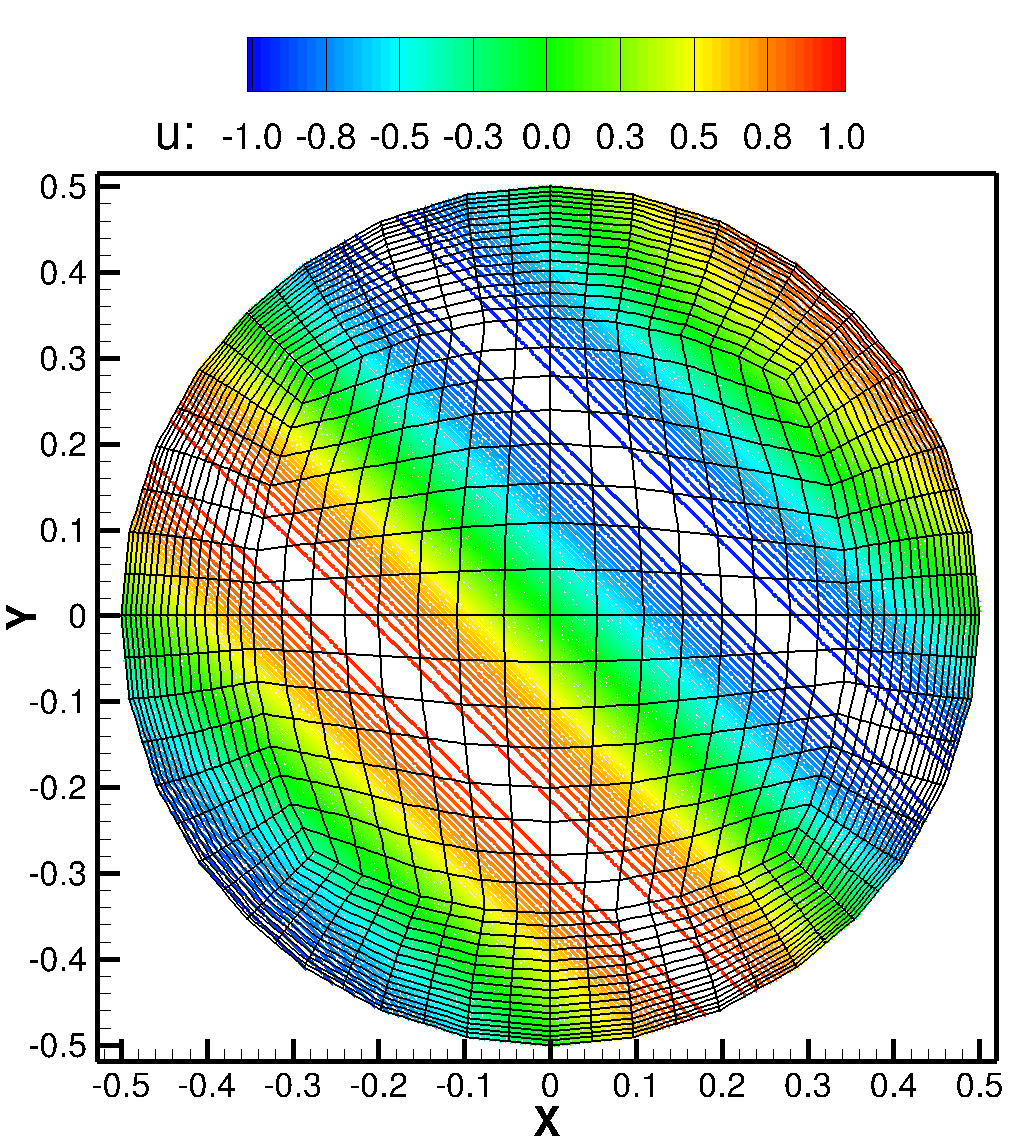}}\hspace{0.2em}
\hspace{0.2em}
  \subfloat[$t=1$]{\includegraphics[width=0.45\textwidth]{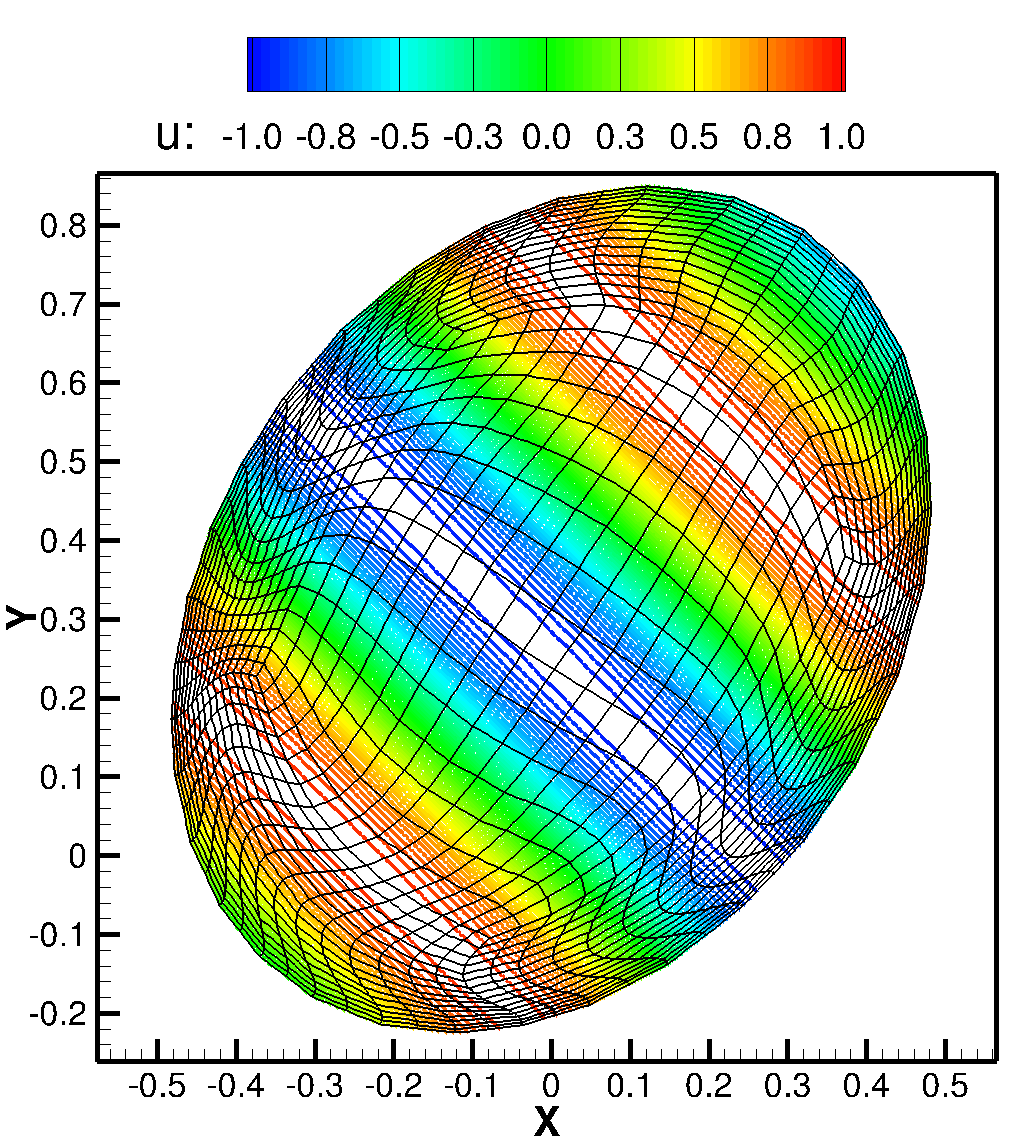}}
    \caption{The wave field and the corresponding grid at (a) $t=0$ and (b) $t=1$ for the 2D linear wave equation on a moving and deforming circular domain. 
    }
  \label{fig:Cylinder_Flow_Comp}
\end{figure}

\begin{figure}[!htbp]
  \centering
  \subfloat[Spatial Convergence]{\includegraphics[width=0.45\textwidth]{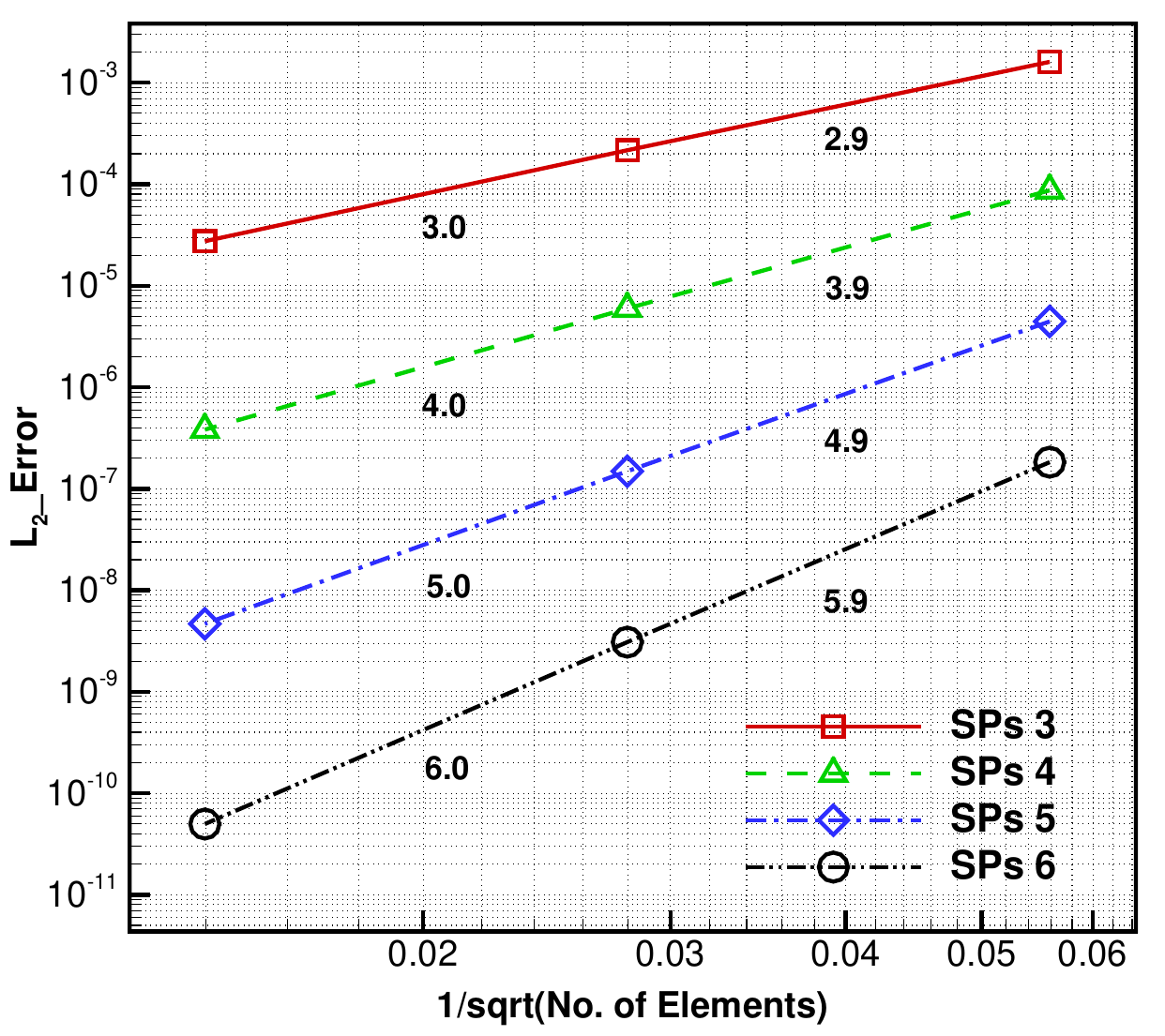}
  \label{fig:Cylinder_bilinear_space}
  }
\hspace{0.2em}
  \subfloat[Temporal Convergence]{\includegraphics[width=0.45\textwidth]{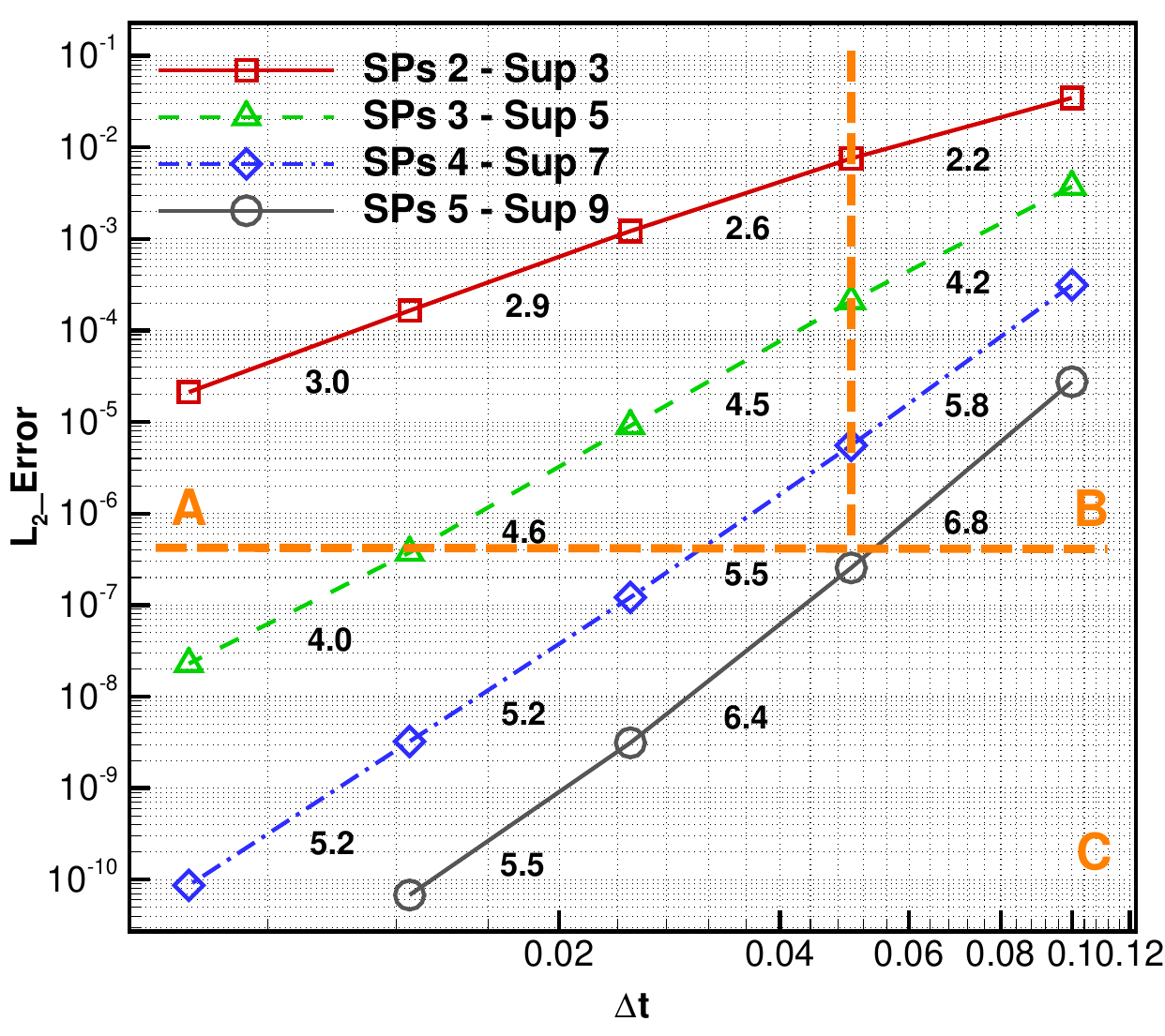}
  \label{fig:Cylinder_bilinear_time}}
    \caption{(a) Spatial convergence rates for spatial constructions using 3 to 6 solution points in each spatial dimension; and (b) temporal convergence rates for temporal constructions using 2 to 5 solution points in the time dimension for the 2D linear wave equation on a moving and deforming circular domain tessellated with linear space-time elements.}
  \label{fig:Cylinder_bilinear}
\end{figure}

We now present results from the space-time order of accuracy studies in Figure~\ref{fig:Cylinder_bilinear}. The spatial convergence rates for all STFR schemes reach their optimal rates.
The temporal convergence rates show some interesting trends, which becomes more apparent in Figure~\ref{fig:Cylinder_curved_time} when curvilinear space-time elements are used to represent the moving/deforming grid. When 2 temporal solution points are used, the scheme's order of accuracy converges toward the superconvergent one, i.e. 3; when 3 temporal solution points are used, the scheme's order of accuracy first converges toward the superconvergent one, i.e. 5, and then starts to decrease when the time step size is further refined; and when 4 or 5 temporal solution points are used, the schemes show superconvergence but not at the optimal rate, and this rate decreases in the time step refinement study.


To explain this phenomenon, 
we divide the $\Delta t$--Error diagram into three regions, as shown in Figure~\ref{fig:Cylinder_bilinear_time}. In Region ``A", the physical errors dominate. This is why the order of accuracy of the schemes converge towards the optimal superconvergence rate. In Region ``B", significant space-time element representation errors exist. In Figure~\ref{fig:Cylinder_Motion_Track}, the trajectories $x$ and $y$ of the grid point originally at $(r_0, \theta_0 )=(0.4,1.07274)$ within the time interval $[0,1]$ under different $\Delta t$ conditions are displayed. It shows that although at the end time the $x$ and $y$ are the same for all cases, their trajectories are different. This results in the uncertain convergence rate for all schemes tested in that region.
In Region ``C", the space-time element representation errors interact with the physical errors, causing convergence rate deterioration.
This will be further confirmed in Figure~\ref{fig:Cylinder_curved_time} when curvilinear space-time elements are used to reduce the space-time element representation errors.

\begin{figure}[]
  \centering
  \subfloat[]{\includegraphics[width=0.48\textwidth]{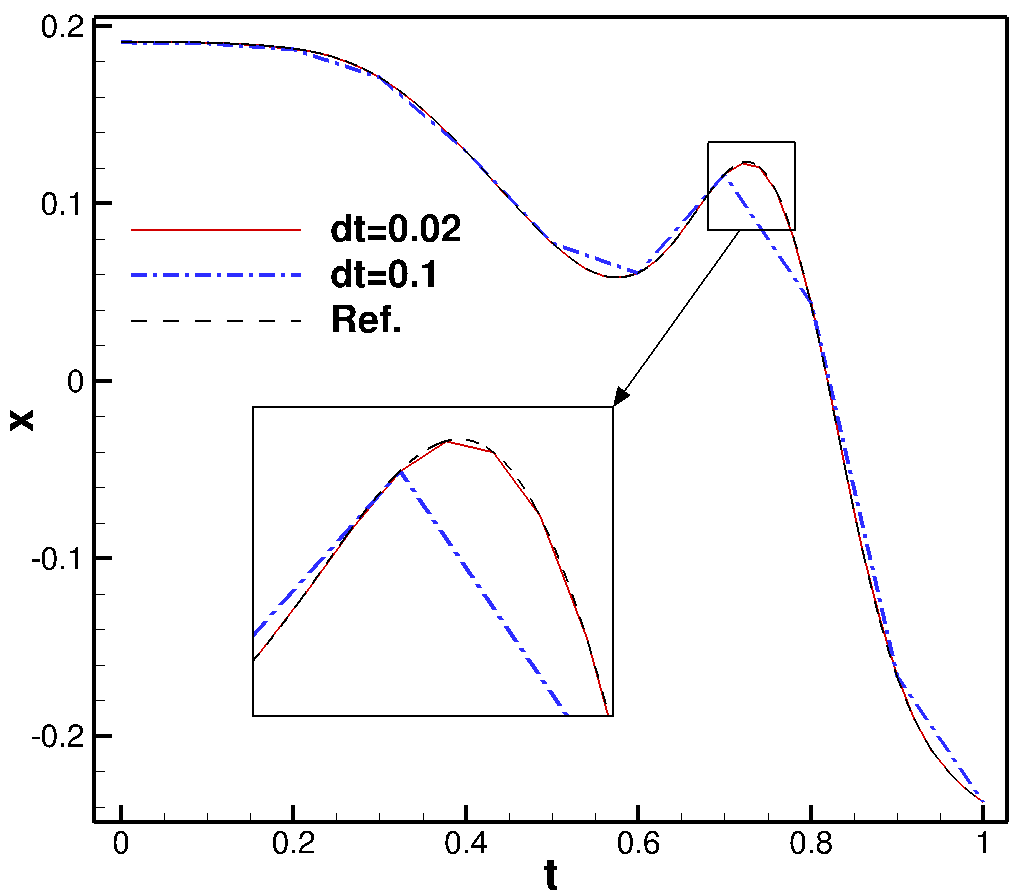}}
\hspace{0.2em}
  \subfloat[]{\includegraphics[width=0.48\textwidth]{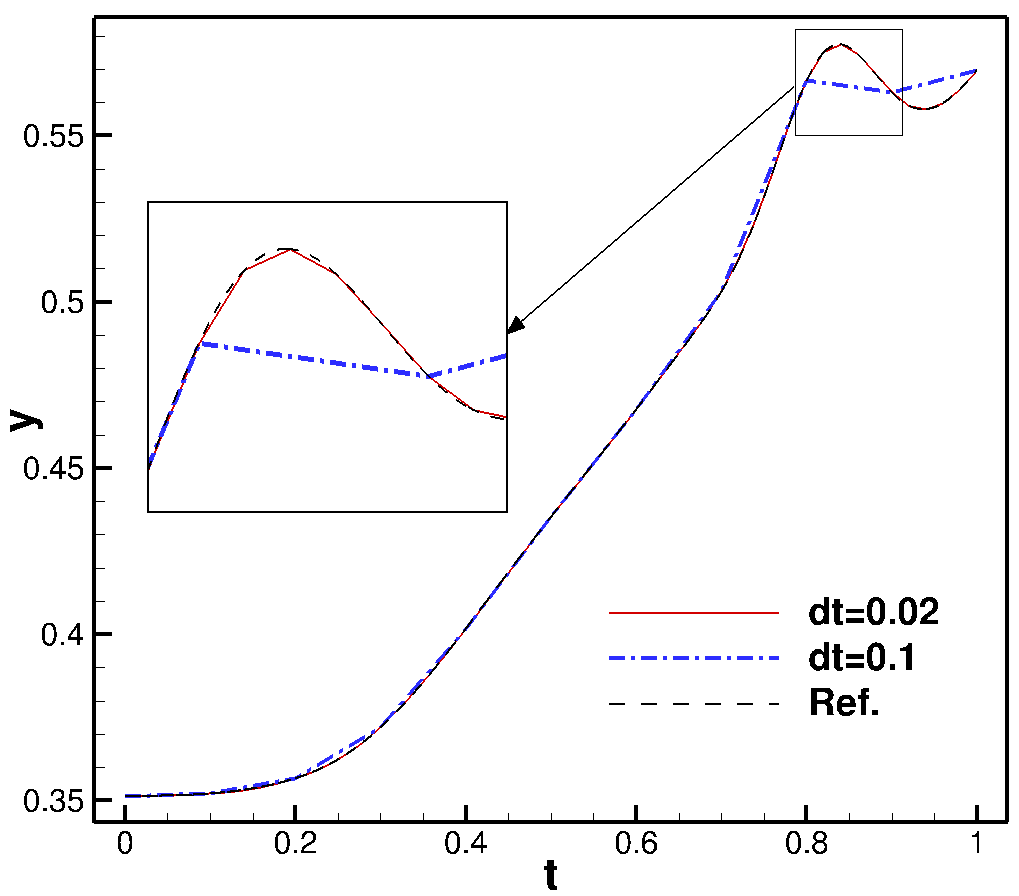}}
    \caption{The trajectories (a) $x$ and (b) $y$ of the grid point originally at $(r_0, \theta_0 )=(0.4,1.07274)$ under different $\Delta t$ conditions when the grid motion is approximated as a linear polynomial.}
  \label{fig:Cylinder_Motion_Track}
\end{figure}

\begin{Observation}
STFR schemes show superconvergence in the temporal direction on deformable grids with large deformation, although they may not reach the optimal rate for very high-order setups due to insufficient grid motion trajectory resolution compared to the resolution of flow variables. The same observation also applies to the curvilinear space-time element cases tested in \textbf{Sect.}~\ref{subsec:MovingCurved}.
\end{Observation}

\begin{Observation}
STFR schemes with the label `\textbf{S}' (i.e. satisfying \textbf{Theorem}~\ref{pro:moving}) can at least achieve the nominal order of accuracy when no filtering operation is applied. The same observation also applies to the curvilinear space-time element cases tested in \textbf{Sect.}~\ref{subsec:MovingCurved}.
\end{Observation}

\begin{Observation}
\label{Obs:filtering}
Spatiotemporal projection-based polynomial filtering can stabilize the convergence rate of the STFR scheme. Specifically, spatial filtering can be an effective way to control the aliasing errors when implementing Eq.~\eqref{eq:STFR_new_phy} with \textbf{Key Procedures}~\ref{alg:2DFR_key_new}. Temporal filtering can result in the loss of superconvergence.   
\end{Observation}




\subsection{Curvilinear tensor-product space-time elements in a moving domain}
\label{subsec:MovingCurved}
In this section, curvilinear space-time grids are used to represent moving/deforming domains. The same two grid deformation strategies, i.e. Eq.~\eqref{eq:mesh_motion_sym_def} and~\eqref{eq:mesh_motion_asym_def}, used in \textbf{Sect.}~\ref{subsec:MovingLinear}, are employed here. Quadratic space-time elements with $l=2$ and $n=2$ are used to better capture the shape and motion of moving and deforming grids. In Figure~\ref{SubFig:Curv_Space_Deform}, the grid tessellated with $16 \times 16$ curvilinear elements is plotted on top of that with $8 \times 8$ curvilinear elements for comparison. Although they do not strictly match each other, the errors caused by imperfect geometric representation of deforming grids during the convergence rate study are better controlled compared to the linear grid cases used in \textbf{Sect.}~\ref{subsec:MovingLinear} (see Figure~\ref{subfig:Sin_deform_mesh_comp}). Similarly, as shown in Figure~\ref{SubFig:Curv_Time_Deform_x} and~\ref{SubFig:Curv_Time_Deform_y}, quadratic representations of the motion can better capture the grid point trajectories compared to linear representations in Figure~\ref{fig:Cylinder_Motion_Track}.

\begin{figure}[!htbp]
  \centering
  \subfloat[]{\includegraphics[width=0.28\textwidth]{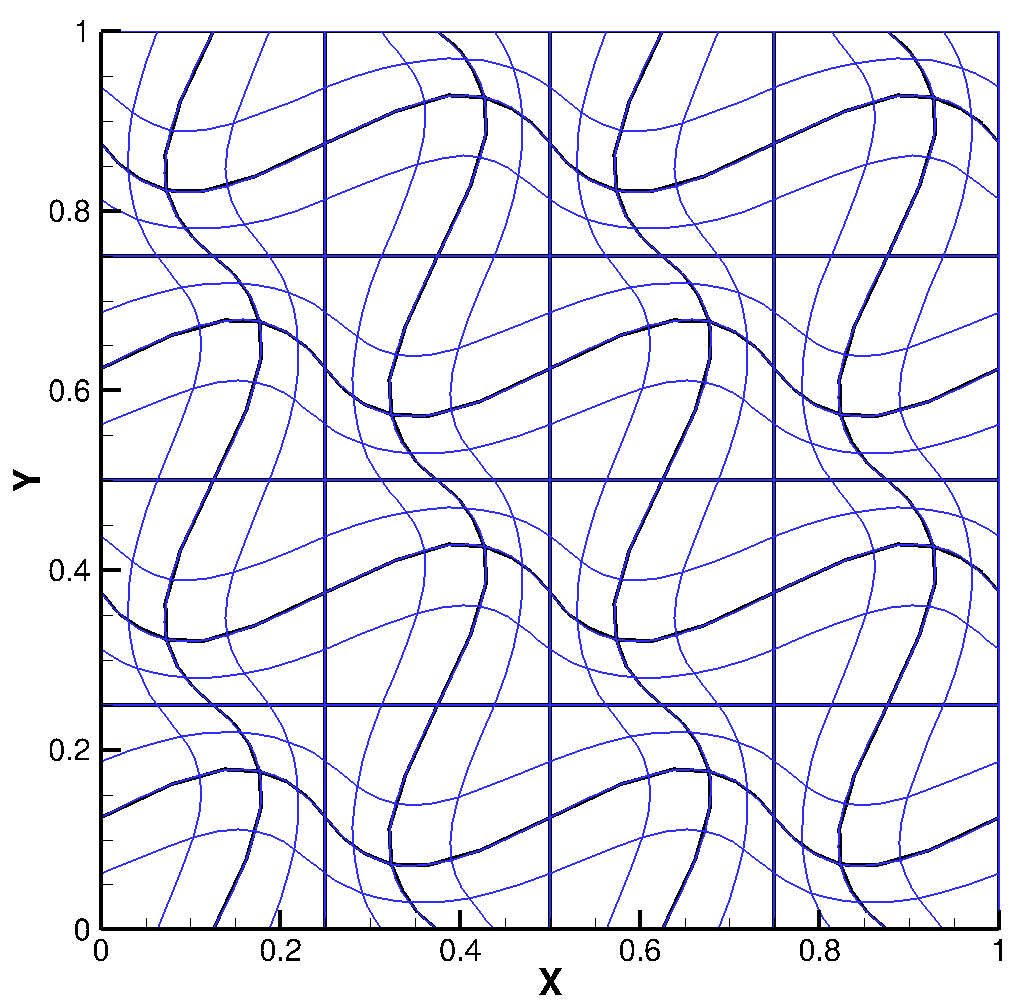}
  \label{SubFig:Curv_Space_Deform}
  }
\hspace{0.2em}
  \subfloat[]{\includegraphics[width=0.31\textwidth]{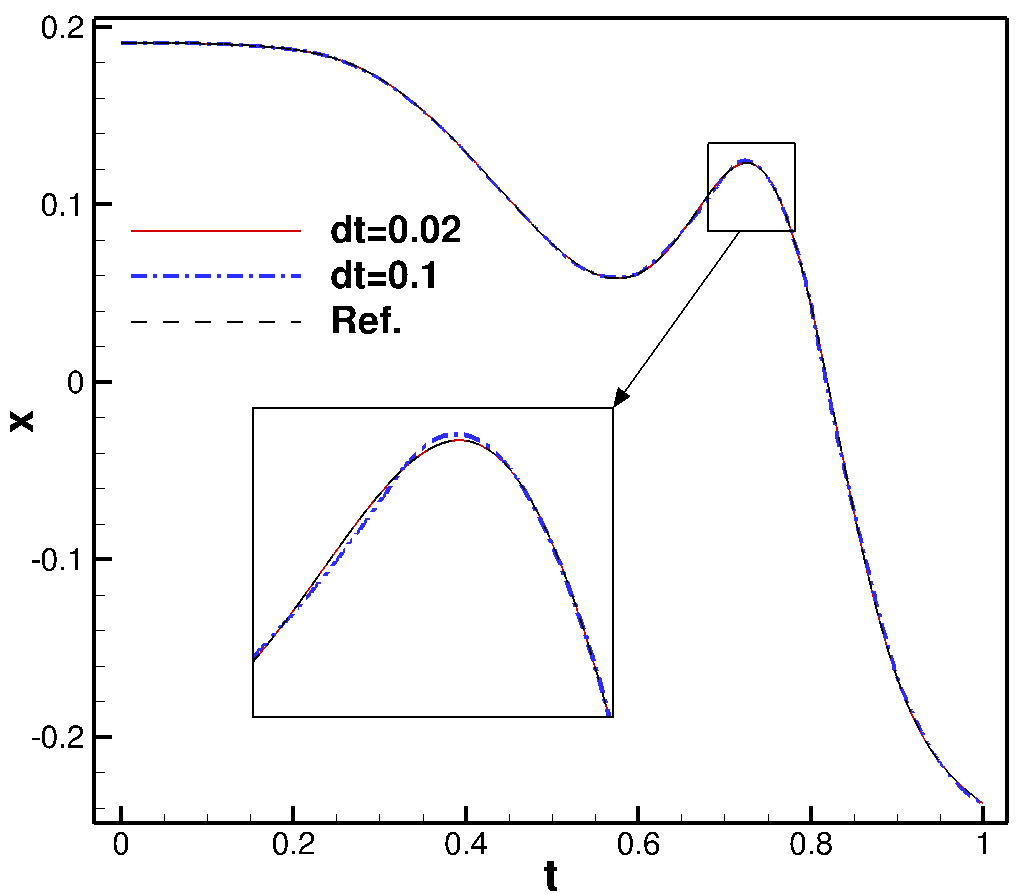}
  \label{SubFig:Curv_Time_Deform_x}
  }
\hspace{0.2em}
  \subfloat[]{\includegraphics[width=0.31\textwidth]{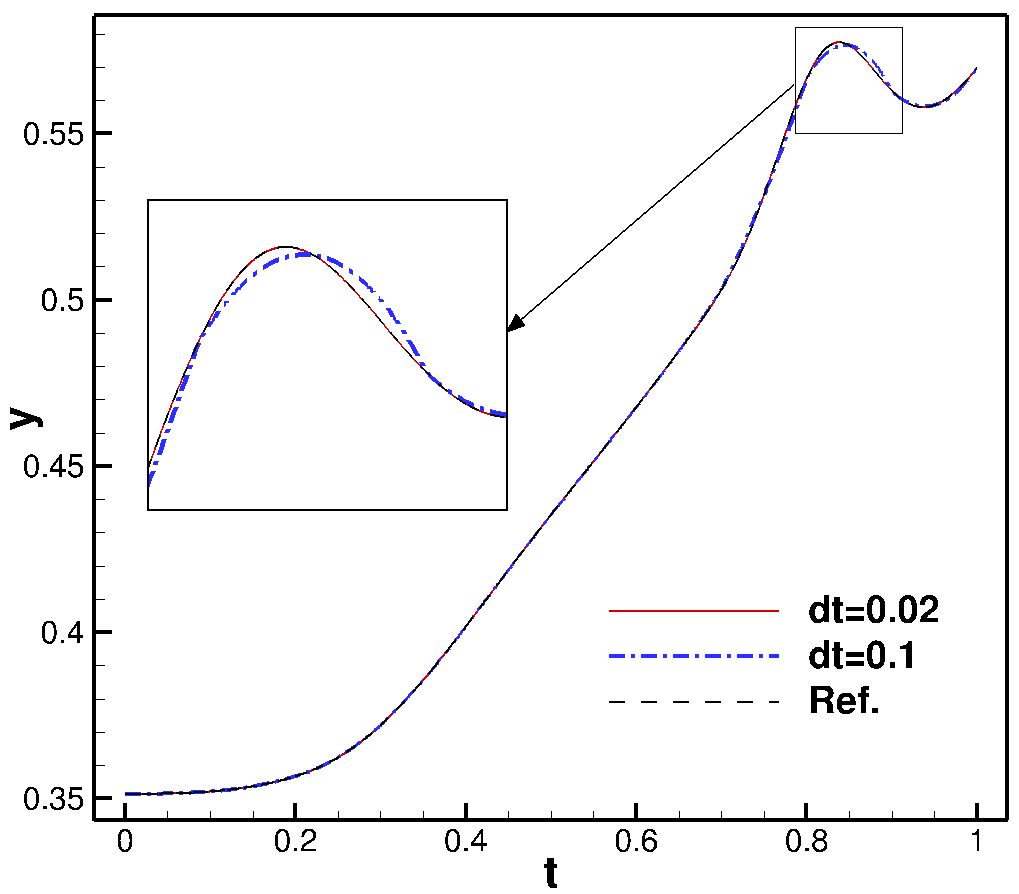}
  \label{SubFig:Curv_Time_Deform_y}
  }
    \caption{(a) The grid with $16 \times 16$ curvilinear elements plotted on top of that with $8 \times 8$ curvilinear elements with the motion described by Eq.~\eqref{eq:mesh_motion_sym_def}. The trajectories (b) $x$ and (c) $y$ of the grid point originally at $(r_0, \theta_0 )=(0.4,1.07274)$ under different $\Delta t$ conditions when the grid motion described by Eq.~\eqref{eq:mesh_motion_asym_def} is approximated as a quadratic polynomial.}
  \label{fig:Curved_Mesh}
\end{figure}

Similar to \textbf{Sect.}~\ref{subsec:MovingLinear}, we list the accuracy label, i.e. `\textbf{S}', `\textbf{P}', or `\textbf{V}', of the STFR schemes with different nominal spatial approximation degree $k$ and temporal approximation degree $m$ for curvilinear spatial and temporal representations of grid elements with $l=2$ and $n=2$ in Table~\ref{Tab:Curvilinear}. Compared to the linear space-time grid cases with $l=1$ and $n=1$, it becomes harder for the STFR schemes to meet the conditions required by \textbf{Theorem}~\ref{pro:moving} when quadratic space-time elements are used to represent moving/deforming grids. As a result, numerical simulations are more prone to being affected by aliasing errors due to the nonlinear interaction between flow quantities and the geometric representation of space-time elements.

\begin{table}[]
\caption{Scheme labels based on the nominal order of accuracy when curvilinear space-time elements with $l=2$ and $n=2$ are used to represent the moving/deforming grids.}
\label{Tab:Curvilinear}
\centering
\subfloat[Spatial Discretizations] {
\label{SubTab:Curv_Space}
\begin{tabular}{|cccc|}
\hline
\multicolumn{4}{|c|}{Desired OoA = \# of SPs}                                                                     \\ \hline
\multicolumn{1}{|c|}{$k$}                & \multicolumn{1}{c|}{Nominal OoA}        & \multicolumn{1}{c|}{\# of SPs} & Label \\ \hline
\multicolumn{1}{|c|}{\multirow{4}{*}{2}} & \multicolumn{1}{c|}{\multirow{4}{*}{3}} & \multicolumn{1}{c|}{3}         & `V'   \\ \cline{3-4} 
\multicolumn{1}{|c|}{}                   & \multicolumn{1}{c|}{}                   & \multicolumn{1}{c|}{4}         & `V'   \\ \cline{3-4} 
\multicolumn{1}{|c|}{}                   & \multicolumn{1}{c|}{}                   & \multicolumn{1}{c|}{5}         & `P'   \\ \cline{3-4} 
\multicolumn{1}{|c|}{}                   & \multicolumn{1}{c|}{}                   & \multicolumn{1}{c|}{6}         & `S'   \\ \hline
\multicolumn{1}{|c|}{\multirow{4}{*}{3}} & \multicolumn{1}{c|}{\multirow{4}{*}{4}} & \multicolumn{1}{c|}{4}         & `V'   \\ \cline{3-4} 
\multicolumn{1}{|c|}{}                   & \multicolumn{1}{c|}{}                   & \multicolumn{1}{c|}{5}         & `P'   \\ \cline{3-4} 
\multicolumn{1}{|c|}{}                   & \multicolumn{1}{c|}{}                   & \multicolumn{1}{c|}{6}         & `P'   \\ \cline{3-4} 
\multicolumn{1}{|c|}{}                   & \multicolumn{1}{c|}{}                   & \multicolumn{1}{c|}{7}         & `S'   \\ \hline
\multicolumn{1}{|c|}{\multirow{4}{*}{4}} & \multicolumn{1}{c|}{\multirow{4}{*}{5}} & \multicolumn{1}{c|}{5}         & `P'   \\ \cline{3-4} 
\multicolumn{1}{|c|}{}                   & \multicolumn{1}{c|}{}                   & \multicolumn{1}{c|}{6}         & `P'   \\ \cline{3-4} 
\multicolumn{1}{|c|}{}                   & \multicolumn{1}{c|}{}                   & \multicolumn{1}{c|}{7}         & `P'   \\ \cline{3-4} 
\multicolumn{1}{|c|}{}                   & \multicolumn{1}{c|}{}                   & \multicolumn{1}{c|}{8}         & `S'   \\ \hline
\end{tabular}
}
\quad
\subfloat[Temporal Discretizations] {
\label{SubTab:Curv_Time}
\begin{tabular}{|cccc|}
\hline
\multicolumn{4}{|c|}{\begin{tabular}[c]{@{}c@{}}Desired OoA with superconvergence \\ = 2 $\times$ (\# of SPs) $-$ 1\end{tabular}}                                                                                                                  \\ \hline
\multicolumn{1}{|c|}{$m$}                & \multicolumn{1}{c|}{\begin{tabular}[c]{@{}c@{}}Nominal\\ superconvergence\end{tabular}} & \multicolumn{1}{c|}{\# of SPs} & Label \\ \hline
\multicolumn{1}{|c|}{\multirow{6}{*}{1}} & \multicolumn{1}{c|}{\multirow{6}{*}{3}}                                                       & \multicolumn{1}{c|}{2}         & `V'   \\ \cline{3-4} 
\multicolumn{1}{|c|}{}                   & \multicolumn{1}{c|}{}                                                                         & \multicolumn{1}{c|}{3}         & `V'   \\ \cline{3-4} 
\multicolumn{1}{|c|}{}                   & \multicolumn{1}{c|}{}                                                                         & \multicolumn{1}{c|}{4}         & `V'   \\ \cline{3-4} 
\multicolumn{1}{|c|}{}                   & \multicolumn{1}{c|}{}                                                                         & \multicolumn{1}{c|}{5}         & `P'   \\ \cline{3-4} 
\multicolumn{1}{|c|}{}                   & \multicolumn{1}{c|}{}                                                                         & \multicolumn{1}{c|}{6}         & `P'   \\ \cline{3-4} 
\multicolumn{1}{|c|}{}                   & \multicolumn{1}{c|}{}                                                                         & \multicolumn{1}{c|}{7}         & `S'   \\ \hline
\multicolumn{1}{|c|}{\multirow{6}{*}{2}} & \multicolumn{1}{c|}{\multirow{6}{*}{5}}                                                       & \multicolumn{1}{c|}{3}         & `V'   \\ \cline{3-4} 
\multicolumn{1}{|c|}{}                   & \multicolumn{1}{c|}{}                                                                         & \multicolumn{1}{c|}{4}         & `V'   \\ \cline{3-4} 
\multicolumn{1}{|c|}{}                   & \multicolumn{1}{c|}{}                                                                         & \multicolumn{1}{c|}{5}         & `P'   \\ \cline{3-4} 
\multicolumn{1}{|c|}{}                   & \multicolumn{1}{c|}{}                                                                         & \multicolumn{1}{c|}{6}         & `P'   \\ \cline{3-4} 
\multicolumn{1}{|c|}{}                   & \multicolumn{1}{c|}{}                                                                         & \multicolumn{1}{c|}{7}         & `P'   \\ \cline{3-4} 
\multicolumn{1}{|c|}{}                   & \multicolumn{1}{c|}{}                                                                         & \multicolumn{1}{c|}{8}         & `S'   \\ \hline
\end{tabular}
}
\end{table}


We first present results in convergence studies of the deformable grid case with the motion described by Eq.~\eqref{eq:mesh_motion_sym_def}. 
The results of the spatial and temporal convergence rates are shown in Figure~\ref{fig:DS_curved}. Compared to the linear spatial grid case presented in Figure~\ref{fig:DS_bilinear_space}, the spatial convergence rate at the coarser grid side has recovered to the optimal value due to the better geometric representation. Similar to the linear motion representation case presented in Figure~\ref{fig:DS_bilinear_time}, the schemes with 2 and 3 temporal solution points show the optimal superconvergence rate. The superconvergent feature of the scheme with 4 temporal solution points improves due to the reduced motion trajectory representation errors.  The deterioration of the superconvergence rate of the scheme with 5 temporal solution points is due to accuracy limitation (i.e. with an absolute error around $\mathcal{O} (10^{-12})$) of the numerical schemes used in this study.  

\begin{figure}[!htbp]
  \centering
  \subfloat[Spatial Convergence]{\includegraphics[width=0.45\textwidth]{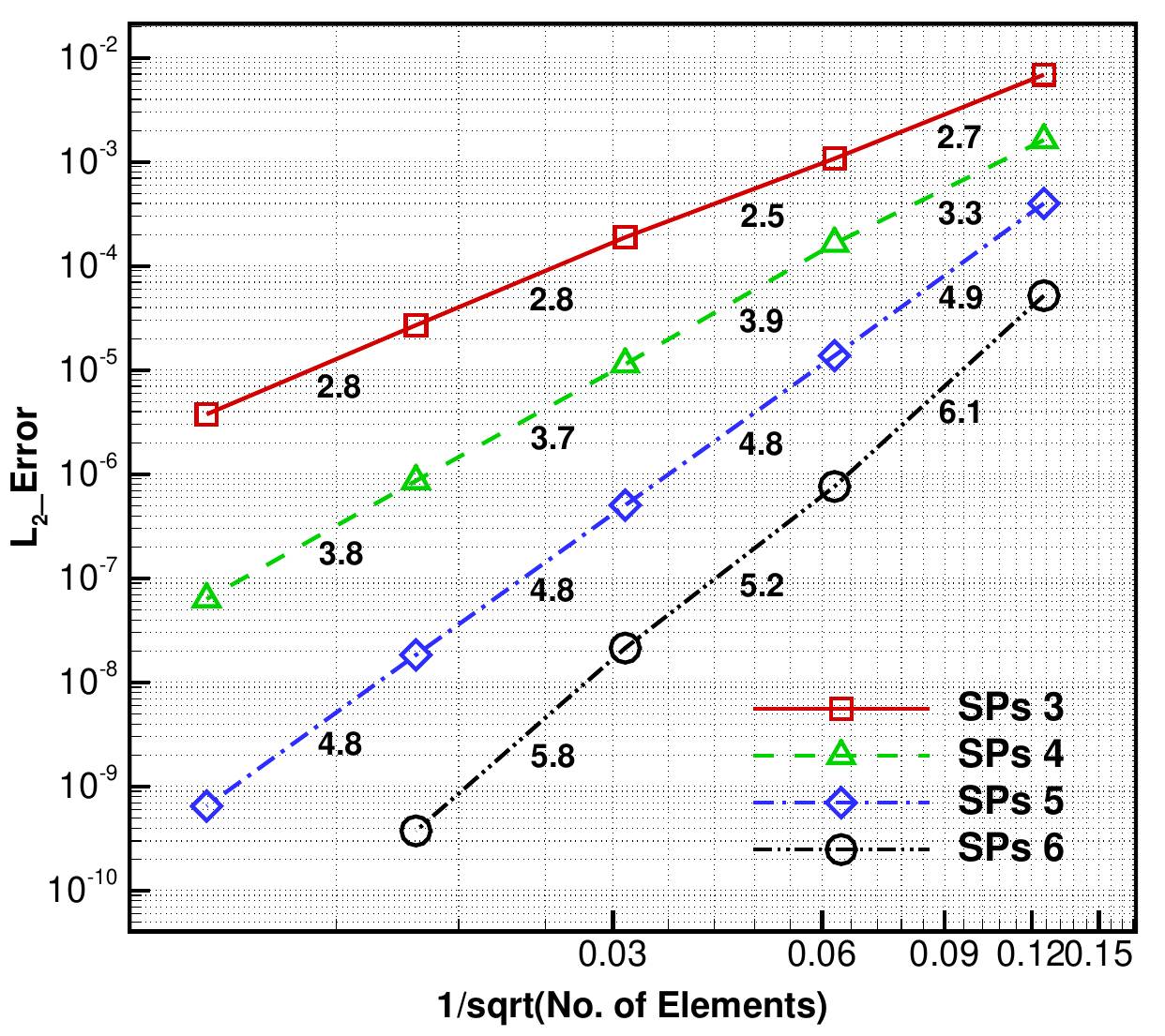}}
\hspace{0.2em}
  \subfloat[Temporal Convergence]{\includegraphics[width=0.45\textwidth]{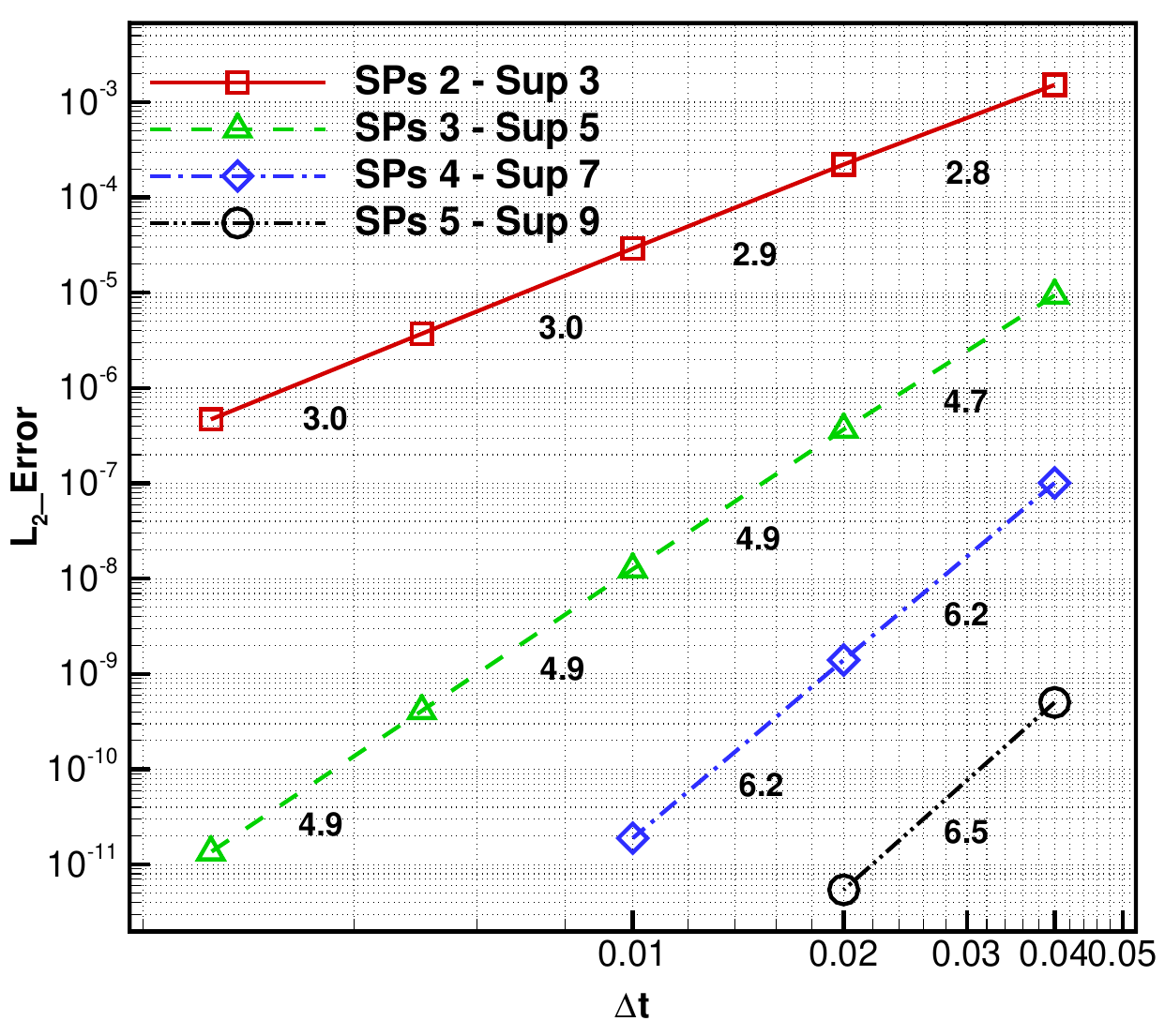}}
    \caption{(a) Spatial convergence rates for spatial constructions using 3 to 6 solution points in each spatial dimension; and (b) temporal convergence rates for temporal constructions using 2 to 5 solution points in the time dimension for the 2D linear wave equation on a deforming square domain tessellated with quadratic space-time elements.}
  \label{fig:DS_curved}
\end{figure}

Spatial and temporal projection-based polynomial filtering is then applied to the schemes tested in Figure~\ref{fig:DS_curved}, and results are presented in Figure~\ref{fig:DS_curved_fil} and~\ref{fig:DS_Curved_Time_Fil}. As expected, we have similar observations as those obtained from the linear space-time grids; see \textbf{Key Observation}~\ref{Obs:filtering}. According to Table~\ref{SubTab:Curv_Space}, the scheme with 6 solution points in each spatial dimension has the label `\textbf{S}', `\textbf{P}', and `\textbf{P}' when the nominal order of accuracy is expected to 3, 4, and 5. Therefore, the projection-based polynomial filtering to $\mathbb{Q}^2$ (i.e. the tensor-product space, on which the scheme with 3 solution points in each spatial dimension is constructed), $\mathbb{Q}^3$, and $\mathbb{Q}^4$ spaces with the filtering parameter $\theta^2=0.9$ is applied to the scheme with 6 solution points in each spatial dimension. We observe from Figure~\ref{subfig:DS_curved_fil_order6} that all filtered schemes achieve the expected convergence rates. Therein, the corresponding unfiltered schemes with the same convergence rates are also plotted for comparison. We find that the absolute errors of the filtered schemes are about two orders of magnitude smaller than those of their unfiltered counterparts. 

\begin{figure}[!htbp]
  \centering
  \subfloat[]{\includegraphics[width=0.45\textwidth]{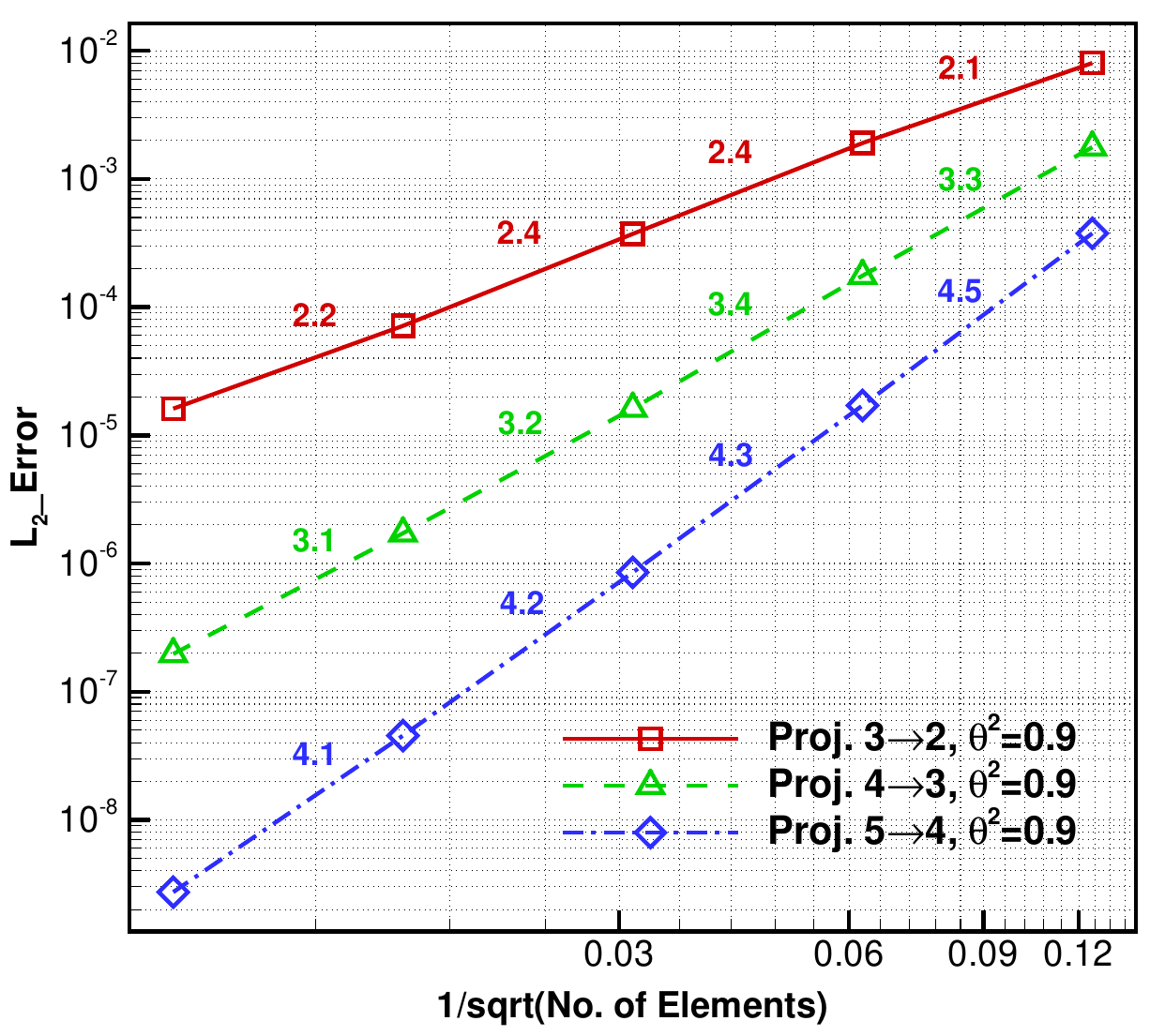}}
\hspace{0.2em}
  \subfloat[]{\includegraphics[width=0.45\textwidth]{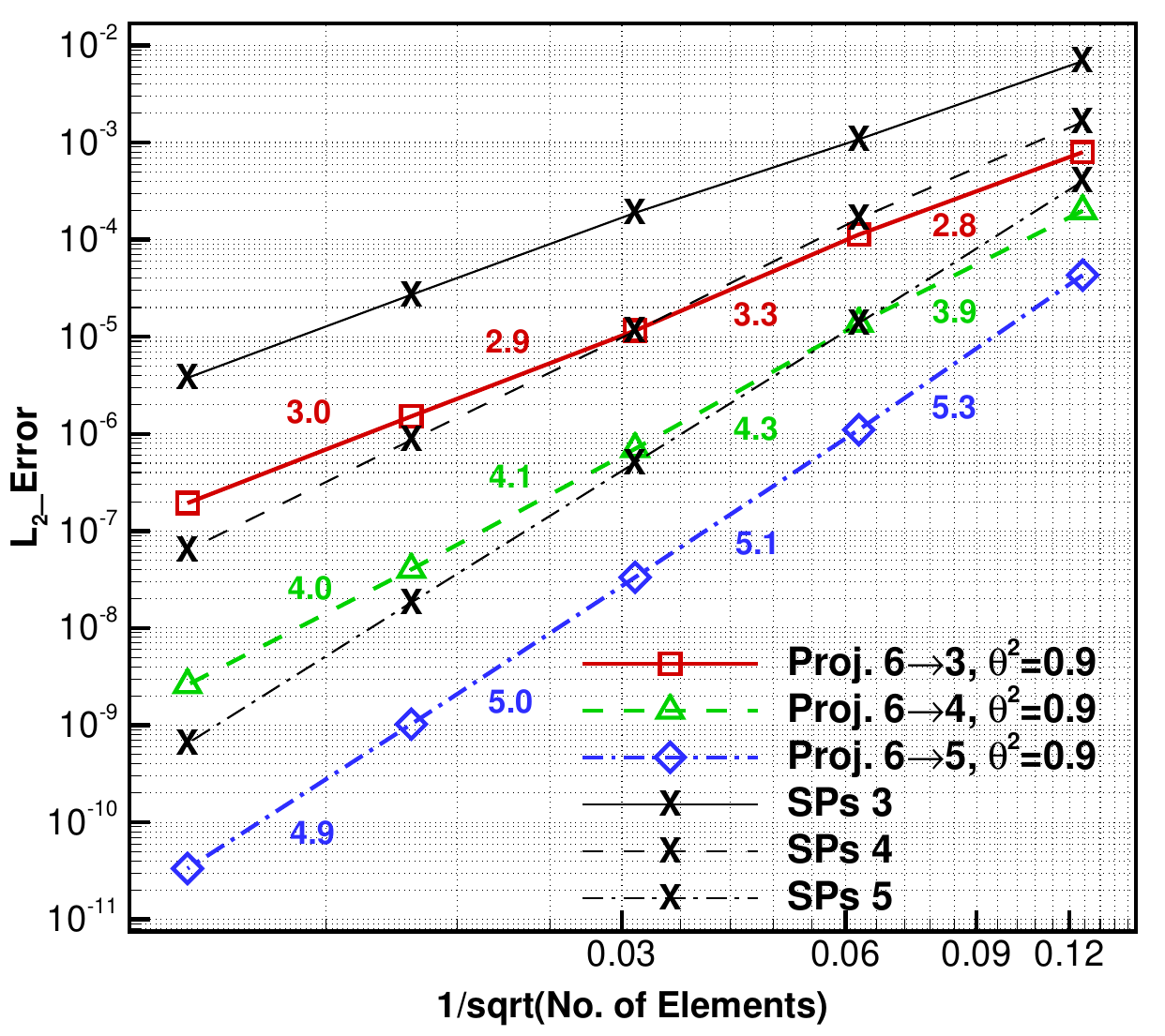}
  \label{subfig:DS_curved_fil_order6}
  }
    \caption{(a) Spatial convergence rates for spatial constructions using 3 to 5 solution points in each spatial dimension with the filtering parameter $\theta^2=0.9$. (b) Spatial convergence rates for the spatial construction using 6 solution points in each spatial dimension with the project-based polynomial filtering to $\mathbb{Q}^2$, $\mathbb{Q}^3$, and $\mathbb{Q}^4$ spaces and with the filtering parameter $\theta^2=0.9$.}
  \label{fig:DS_curved_fil}
\end{figure}

\begin{figure}[!htbp]
  \centering
  \includegraphics[width=0.45\textwidth]{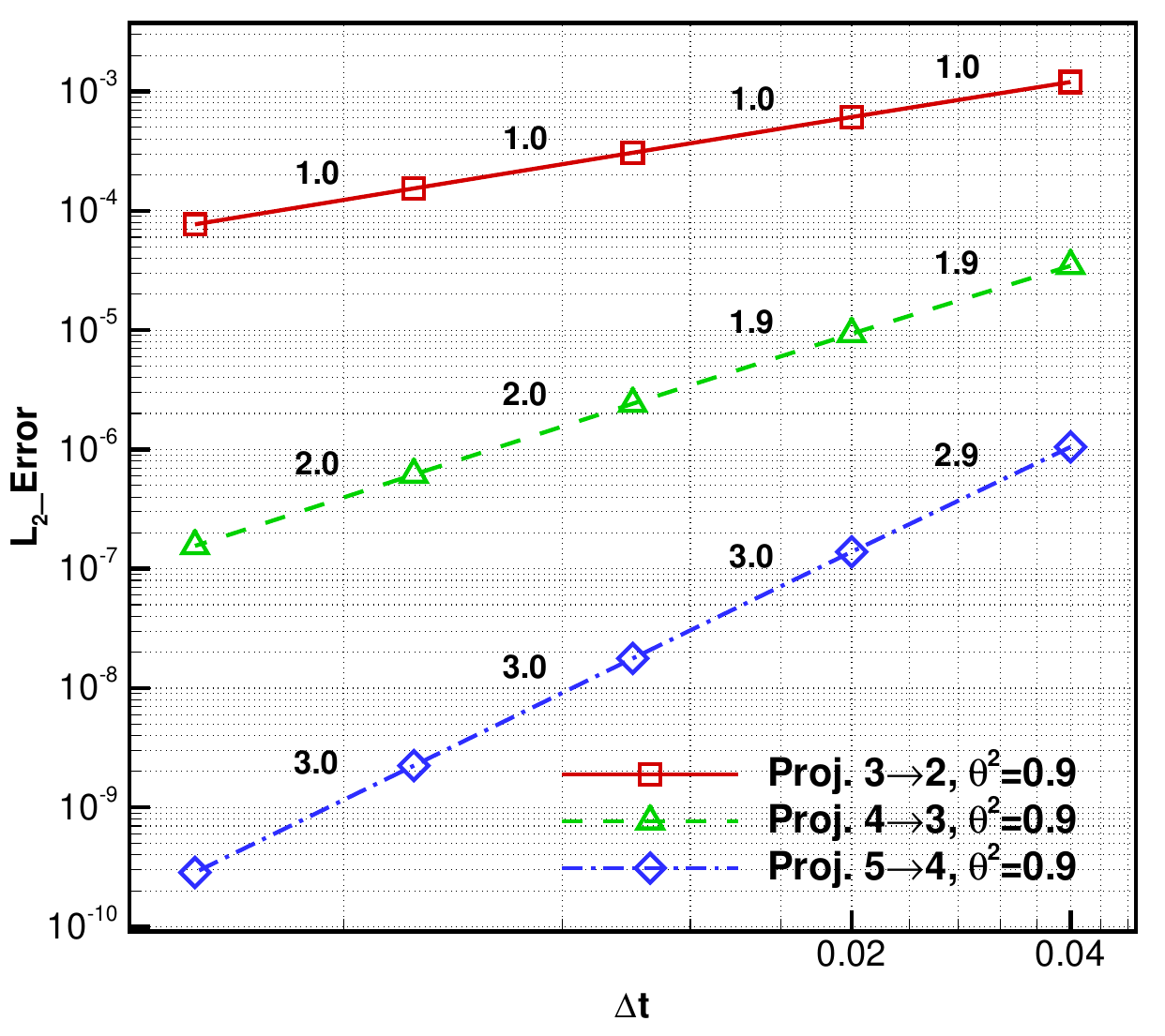}
    \caption{Global temporal convergence rates for temporal constructions using 3, 4, and 5 solution points in the time dimension with the filtering parameter $\theta^2=0.9$.}
  \label{fig:DS_Curved_Time_Fil}
\end{figure}

Finally, we present results in the order of accuracy studies of the moving/deforming circular case with the motion described by Eq.~\eqref{eq:mesh_motion_asym_def} in Figure~\ref{fig:Cylinder_curved}. Similar to the linear grid case (see Figure~\ref{fig:Cylinder_bilinear_space}), the spatial convergence rates for all STFR schemes reach their optimal rates. As shown in Figure~\ref{fig:Cylinder_curved_time}, we divide the $\Delta t$--Error diagram into three regions. Compared to Figure~\ref{fig:Cylinder_bilinear_time}, the area of Region ``A" in Figure~\ref{fig:Cylinder_curved_time} is enlarged, where the superconvergence rates of the schemes with 2, 3, and 4 temporal solution points are converging towards their optimal values. This is due to the reduced motion trajectory representation error when curvilinear space-time elements are used. Nevertheless, compared to the physical flow field error, the motion trajectory representation error becomes relatively large again in Region ``C", and the superconvergence rates start to degrade in this region. 

%
%

\begin{figure}[!htbp]
  \centering
  \subfloat[Spatial Convergence]{\includegraphics[width=0.45\textwidth]{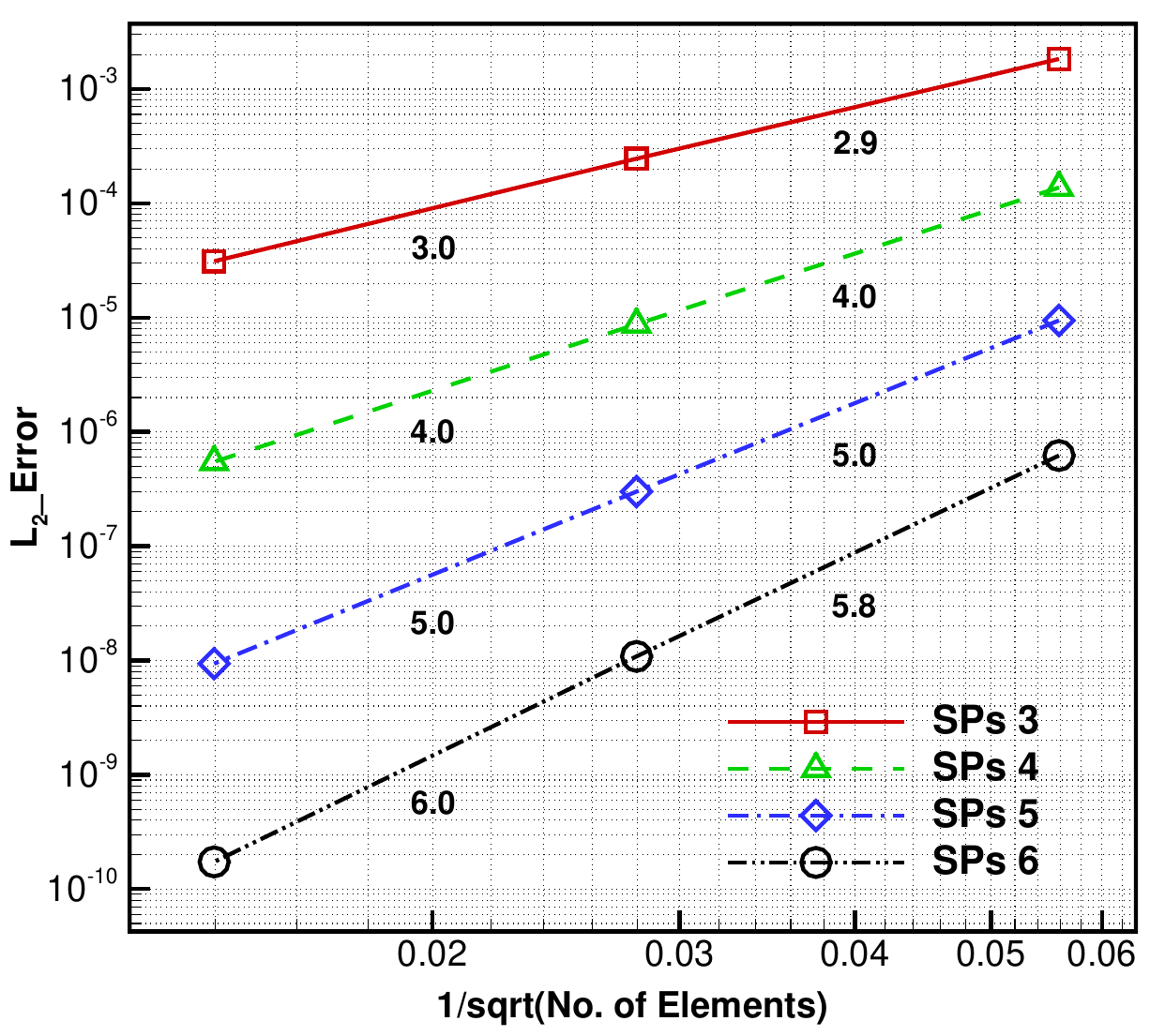}}
\hspace{0.2em}
  \subfloat[Temporal Convergence]{\includegraphics[width=0.45\textwidth]{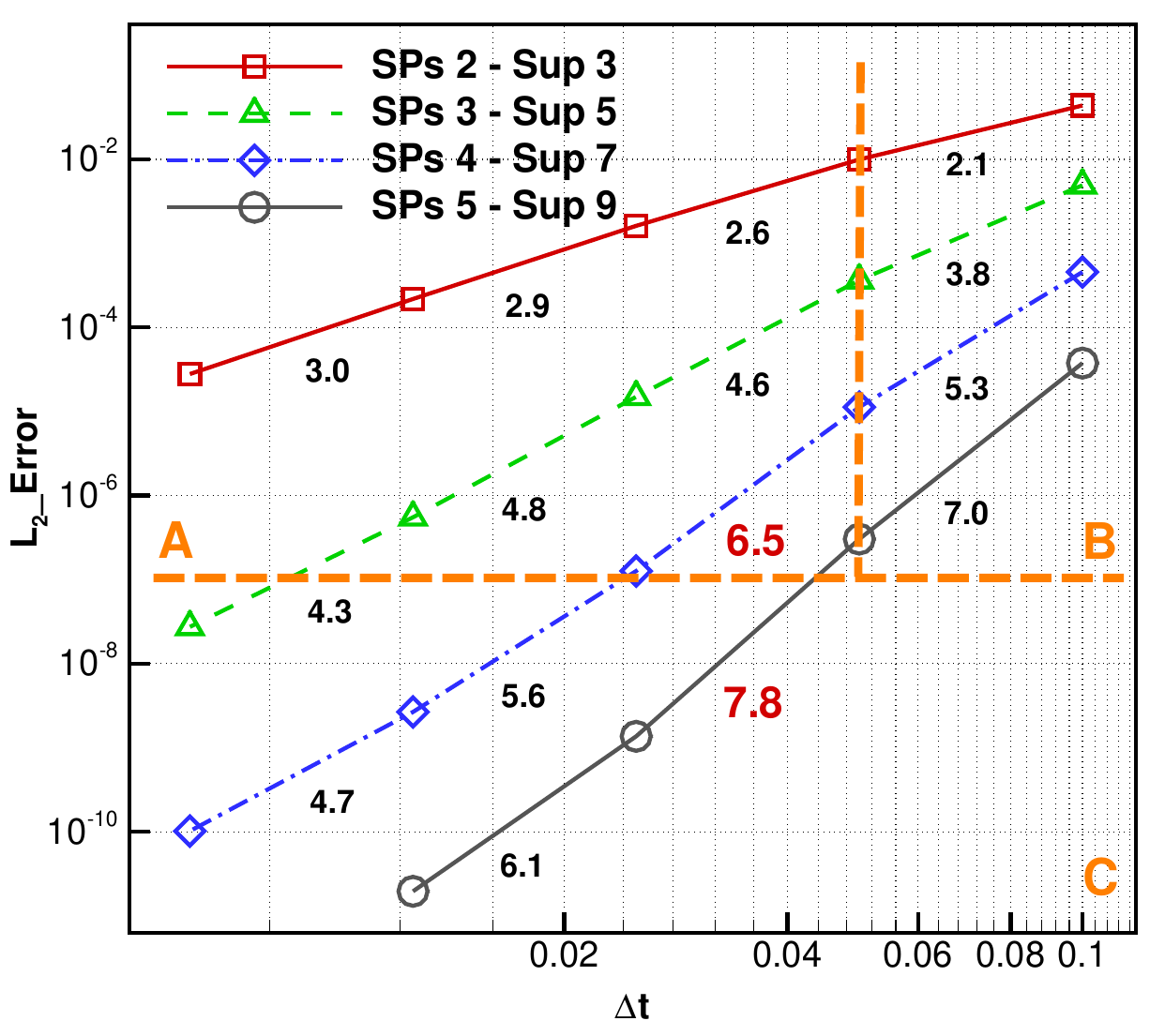}
  \label{fig:Cylinder_curved_time}}
    \caption{(a) Spatial convergence rates for spatial constructions using 3 to 6 solution points in each spatial dimension; and (b) temporal convergence rates for temporal constructions using 2 to 5 solution points in the time dimension for the 2D linear wave equation on a moving and deforming circular domain tessellated with quadratic space-time elements.}
  \label{fig:Cylinder_curved}
\end{figure}

\begin{Observation}
Curvilinear space-time elements can better capture the geometric features and motion of moving/deforming grids compared to their linear counterpart, with the trade-off of accumulated aliasing errors.  When aliasing errors due to insufficient spatial resolution dominate, they can be effectively controlled by spatial projection-based polynomial filtering. Aliasing errors due to the insufficient temporal resolution seem not to cause seriouss accuracy deterioration. Similarly to the linear space-time element cases, temporal filtering can make the filtered schemes fully lose temporal superconvergence, and present a regular global convergence rate. 
\end{Observation}

\subsection{Discussions} 
\label{subsec:Results_Dis}
From the numerical tests presented in \textbf{Sect.}~\ref{subsec:MovingLinear} and~\ref{subsec:MovingCurved}, we observed that schemes violating \textbf{Theorem}~\ref{pro:moving} and \textbf{Proposition}~\ref{pro:moving_reduced} do not substantially destroy the flow simulation accuracy, and when the scheme's order of accuracy improves, their accuracy is similar to that of the schemes which satisfy \textbf{Theorem}~\ref{pro:moving} or \textbf{Proposition}~\ref{pro:moving_reduced}. This is \textit{de facto} the numerical accuracy enhancement feature of the STFR schemes when Eq.~\eqref{eq:STFR_new_phy} is implemented with \textbf{Key Procedures}~\ref{alg:2DFR_key_new}. Similar phenomena have been reported in~\cite{YU201470} (see Section 5.3 in that work). We restate that when a scheme violates \textbf{Theorem}~\ref{pro:moving}, Eq.~\eqref{eq:STFR_new_phy} and~\eqref{eq:STFR_new_ref} are not numerically identical to each other. Violating \textbf{Proposition}~\ref{pro:moving_reduced} (i.e. the discrete GCL) can cause serious accuracy deterioration when Eq.~~\eqref{eq:STFR_new_ref} is implemented with $|J| \vect{Q}_{\vect{h}}$ explicitly approximated as polynomials; see examples from~\cite{YU201470,AbeEtAl_JCP_2015,ABE_EtAl_CF_2016}. As discussed in \textbf{Sect.}~\ref{subsubsec:new_discuss}, aliasing errors still need to be controlled when implementing Eq.~\eqref{eq:STFR_new_phy} with $\vect{Q}_{\vect{h}}$ explicitly approximated as polynomials in the reference domain. Based on the tests presented in this section, polynomial filtering can be an effective way to steer the spatial convergence rate of numerical schemes towards the desired value. Temporal filtering is not recommended considering that it can cause the numerical schemes to fully lose the temporal superconvergence property.

\section{Conclusion and Future Work}
\label{sec:con_FW}

In this study, high-order nodal tensor-product STFR methods based on Gauss--Legendre points have been developed to solve hyperbolic conservation laws on curvilinear moving grids. The complex grid motion and deformation are represented with general curvilinear space-time elements, and numerical resolution requirements to resolve curvilinear geometric features are studied.  A key finding of this study is that the STFR schemes consistently exhibit superconvergence in the time dimension, irrespective of whether the governing equations are linear or nonlinear, or whether the computations are performed on stationary or highly deforming moving grids. The superconvergence rate, however, can be influenced by errors associated with the grid motion representation. The use of curvilinear space-time elements, even when the grid and its motion are represented quadratically, substantially enhances the accuracy of the space-time geometric representation. Nevertheless, a finer numerical resolution than that used for linear space-time grids is required to adequately resolve both the flow field and the geometric features of the curvilinear elements.

A STFR scheme implementation procedure based on a hybrid reference-physics domain approach, i.e. \textbf{Key Procedures}~\ref{alg:2DFR_key_new}, is developed to enhance the accuracy of the numerical schemes. Note that this procedure can preserve the freestream even when the discrete GCL is violated. It has been proved that this procedure is numerically equivalent to the conservative form in the reference domain when \textbf{Theorem}~\ref{pro:moving} is satisfied. Space-time projection-based polynomial filtering has been developed to control aliasing errors of the STFR schemes implemented with the hybrid reference-physics domain approach when their resolution can either not satisfy that required by the discrete GCL (i.e. violating \textbf{Proposition}~\ref{pro:moving_reduced}) or not guarantee the equivalence between the hybrid form and its corresponding conservative form (i.e. violating \textbf{Theorem}~\ref{pro:moving}). A key observation here is that the polynomial filtering procedure can be used to stabilize the spatial order of accuracy of the STFR scheme to any value less than the desired one determined by the spatial solution point setup. However, the temporal superconvergence is totally lost when temporal polynomial filtering is used; instead, a global temporal convergence rate is observed.

Several directions for future work naturally extend from the present study.
First, in practical fluid flow simulations, four-dimensional (4D), i.e. 3D space plus 1D time, space-time grids need to be used. It is not clear whether the temporal superconvergence, an essential feature for numerical efficiency, still holds for general 4D polytopes~\cite{Behr_IJNMF_2008,FRONTIN_EtAl_2021_ANM}. 
Second, much research is still needed to bridge the method of lines and space-time discretization for moving domain simulation with curvilinear grids. Based on some preliminary tests presented in~\cite{Yu_Space_Time_Moving_2024}, very small time steps need to be used when the method of lines is adopted for time integration to match the accuracy obtained by the space-time method. Therefore, another promising future research direction can be identifying equivalent method of lines schemes to space-time ones, such as the DG-type methods for ordinary differential equations reported in~\cite{Huynh_JSC_2023}. As a result, the knowledge created from the space-time methods can be directly transferred to the development of the method of lines. Last but not least, curvilinear space-time (or its equivalent method-of-lines form) moving grid simulation of Navier--Stokes equations can be pursued. The numerical performance between the tensor-product-based schemes and those for irregular 4D polytopes needs to be compared.

\section*{Acknowledgments}
The early stage of this work was sponsored by the Office of Naval Research (ONR) through the award N00014-16-1-2735, and this work was also partially sponsored by the Army Research Office (ARO) through the Cooperative Agreement W911NF2420183. The views and conclusions contained in this document are those of the authors and should not be interpreted as representing the official policies, either expressed or implied, of the U.S. Government. Part of the hardware used in the computational studies is from the UMBC High Performance Computing Facility (HPCF). The facility is supported by the U.S. National Science Foundation through the MRI program (grant nos. CNS-0821258, CNS-1228778, OAC-1726023, and CNS-1920079) and the SCREMS program (grant no. DMS-0821311), with additional substantial support from the University of Maryland, Baltimore County (UMBC). 

\section*{Data Availability}
Data sets generated during the current study are available from the corresponding author on reasonable request. 

\section*{Declarations}

\noindent \textbf{Competing interest} 
The authors declare that they have no known competing financial interests or personal
relationships that could have appeared to influence the work reported in this paper.

\appendix
\section{Proof of Eq.~\eqref{eq:Proj_Coef}} 
\label{app:proj}
The same nomenclature as that adopted in \textbf{Sect.}~\ref{subsubsec:new_discuss} is used here as well as in~\ref{app:interp}. Since Gauss--Legendre quadrature points are used as solution points in STFR, and 1D Lagrange polynomials are generated based on those points, we have the following two fundamental properties:

\begin{itemize}
    \item With $N+1$ Gauss--Legendre quadrature points, a degree $2N+1$ polynomial can be integrated exactly.
    \item The set of Lagrange polynomials based on $N+1$ Gauss--Legendre quadrature points form a complete basis of the polynomial space $P^N$, and the bases are orthogonal to each other. 
\end{itemize}

Take $\phi_{i_1}^H (\xi)$, $i_1 = 1, \ldots, k^H + 1$, a set of degree $k^H$ Lagrange polynomials, as an example. They are based on $k^H + 1$ Gauss--Legendre quadrature points $\xi_{j_1}^H$, and those points can integrate a degree $2k^H+1$ polynomial exactly. Therefore, the product of any two bases $\phi_{i_1}^H (\xi) \phi_{j_1}^H (\xi)$, a degree $2k^H$ polynomial, can be integrated exactly with the $k^H + 1$ Gauss--Legendre quadrature points.   
Moreover, $\phi_{i_1}^H (\xi_{j_1}^H)$ equals to 1 when $i_1 = j_1$, and equals to 0 when $i_1 \neq j_1$. For any two bases $\phi_{i_1}^H (\xi)$ and $\phi_{j_1}^H (\xi)$, we then have their inner product calculated as
\begin{equation}
\begin{aligned}
    \int_{-1}^1 \phi_{i_1}^H (\xi) \phi_{j_1}^H (\xi) d \xi  & =
    \sum_{k_1 = 1}^{k^H+1} \phi_{i_1}^H (\xi_{k_1}^H) \phi_{j_1}^H (\xi_{k_1}^H) \omega_{s,k_1}^H \\
    & =
    \phi_{i_1}^H (\xi_{j_1}^H) \omega_{s,j_1}^H \\
    & = \Biggl\{ 
    \begin{aligned}
    \omega_{s,i_1}^H \quad \text{if} \ \ i_1 = j_1 \\
    0 \quad \text{if} \ \ i_1 \neq j_1
    \end{aligned}.
\end{aligned}
\end{equation}

The same procedures apply to $\phi_{i_1}^L (\xi)$, $i_1 = 1, \ldots, k^L + 1$, a set of degree $k^L$ Lagrange polynomials. Now we check the product of $\phi_{i_1}^H (\xi) \phi_{j_1}^L (\xi)$. It is a degree ($k^H + k^L$) polynomial, and can be integrated exactly with the $k^H + 1$ Gauss--Legendre quadrature points $\xi_{j_1}^H$. Therefore, we have the following formula
\begin{equation}
    \int_{-1}^1 \phi_{i_1}^H (\xi) \phi_{j_1}^L (\xi) d \xi   =
    \sum_{k_1 = 1}^{k^H+1} \phi_{i_1}^H (\xi_{k_1}^H) \phi_{j_1}^L (\xi_{k_1}^H) \omega_{s,k_1}^H 
     =
    \phi_{j_1}^L (\xi_{i_1}^H) \omega_{s,i_1}^H.
\end{equation}

The left side of Eq.~\eqref{eq:Proj_Coef} can then be written as
\begin{equation}
\begin{aligned}
&   \iiint \sum_{i_1=1}^{k^L+1}  
    \sum_{i_2=1}^{k^L+1} \sum_{i_3=1}^{m^L+1} Q^L_{i_1,i_2,i_3} L^L_{i_1,i_2,i_3} (\vect{\xi}) 
    L^L_{j_1,j_2,j_3} (\vect{\xi}) d\vect{\xi} \\
= & \sum_{i_1=1}^{k^L+1}  
    \sum_{i_2=1}^{k^L+1} \sum_{i_3=1}^{m^L+1}
    Q^L_{i_1,i_2,i_3} 
    \int_{-1}^1 \phi_{i_1}^L (\xi) \phi_{j_1}^L (\xi) d \xi
    \int_{-1}^1 \phi_{i_2}^L (\eta) \phi_{j_2}^L (\eta) d \eta
    \int_{-1}^1 \phi_{i_3}^L (\tau) \phi_{j_3}^L (\tau) d \tau \\
= & \sum_{i_1=1}^{k^L+1}  
    \sum_{i_2=1}^{k^L+1} \sum_{i_3=1}^{m^L+1}
    Q^L_{i_1,i_2,i_3}
    \phi_{i_1}^L (\xi_{j_1}^L) 
    \phi_{i_2}^L (\eta_{j_2}^L)
    \phi_{i_3}^L (\tau_{j_3}^L)
    \omega_{s,j_1}^L \omega_{s,j_2}^L \omega_{t,j_3}^L \\
= & Q^L_{j_1,j_2,j_3} \omega_{s,j_1}^L \omega_{s,j_2}^L \omega_{t,j_3}^L.
\end{aligned}
\label{eq:proj_l}
\end{equation}

The right side of Eq.~\eqref{eq:Proj_Coef} is calculated from
\begin{equation}
\begin{aligned}
&   \iiint \sum_{i_1=1}^{k^H+1}  
    \sum_{i_2=1}^{k^H+1} \sum_{i_3=1}^{m^H+1} Q^H_{i_1,i_2,i_3} L^H_{i_1,i_2,i_3} (\vect{\xi}) 
    L^L_{j_1,j_2,j_3} (\vect{\xi}) d\vect{\xi} \\
= & \sum_{i_1=1}^{k^H+1}  
    \sum_{i_2=1}^{k^H+1} \sum_{i_3=1}^{m^H+1} Q^H_{i_1,i_2,i_3} 
    \int_{-1}^1 \phi_{i_1}^H (\xi) \phi_{j_1}^L (\xi) d \xi
    \int_{-1}^1 \phi_{i_2}^H (\eta) \phi_{j_2}^L (\eta) d \eta
    \int_{-1}^1 \phi_{i_3}^H (\tau) \phi_{j_3}^L (\tau) d \tau \\
= & \sum_{i_1=1}^{k^H+1}  
    \sum_{i_2=1}^{k^H+1} \sum_{i_3=1}^{m^H+1}
    Q^H_{i_1,i_2,i_3}
    \phi_{j_1}^L (\xi_{i_1}^H) 
    \phi_{j_2}^L (\eta_{i_2}^H)
    \phi_{j_3}^L (\tau_{i_3}^H)
    \omega_{s,i_1}^H \omega_{s,i_2}^H \omega_{t,i_3}^H.
\end{aligned}
\label{eq:proj_r}
\end{equation}

On combining Eq.~\eqref{eq:proj_l} and Eq.~\eqref{eq:proj_r}, we have Eq.~\eqref{eq:Proj_Coef}.

\section{Proof of Eq.~\eqref{eq:L_Lagrange} and projection-based filtering strength discussion} 
\label{app:interp}

According to the basic principle of interpolation with Lagrange polynomials, $Q^L(\vect{\xi}) \in \mathbb{Q}^{k^L} (\xi, \eta) \otimes P^{m^L} (\tau)$ can be constructed using the basis functions $\phi_{i_1}^H (\xi)$, $\phi_{j_1}^H (\eta)$ and $\phi_{k_1}^H (\tau)$ of the polynomial space $\mathbb{Q}^{k^H} (\xi, \eta) \otimes P^{m^H} (\tau)$ as
\begin{equation}
Q^{',L}(\vect{\xi})  = 
\sum_{i_1=1}^{k^H+1}  \sum_{i_2=1}^{k^H+1} \sum_{i_3=1}^{m^H+1} 
Q^L(\vect{\xi}^H_{i_1,i_2,i_3})
\phi^H_{i_1}(\xi) \phi^H_{i_2}(\eta) \phi^H_{i_3}(\tau).
\label{eq:QL_alt}
\end{equation}

According to Eq.~\eqref{eq:Sol_Lagrange}, we have
\begin{equation*}
    Q^L(\vect{\xi}^H_{j_1,j_2,j_3}) =
    \sum_{i_1=1}^{k^L+1}  \sum_{i_2=1}^{k^L+1} \sum_{i_3=1}^{m^L+1} Q^L_{i_1,i_2,i_3}
    \phi^L_{i_1}(\xi_{j_1}^H) \phi^L_{i_2}(\eta_{j_2}^H) \phi^L_{i_3}(\tau_{j_3}^H).
\end{equation*}

Therefore, Eq.~\eqref{eq:QL_alt} can be written as
\begin{equation}
\begin{aligned}
    Q^{',L}(\vect{\xi}) & = 
\sum_{j_1=1}^{k^H+1}  \sum_{j_2=1}^{k^H+1} \sum_{j_3=1}^{m^H+1} 
\sum_{i_1=1}^{k^L+1}  \sum_{i_2=1}^{k^L+1} \sum_{i_3=1}^{m^L+1} Q^L_{i_1,i_2,i_3}
\phi^L_{i_1}(\xi_{j_1}^H) \phi^L_{i_2}(\eta_{j_2}^H) \phi^L_{i_3}(\tau_{j_3}^H)
\phi^H_{j_1}(\xi) \phi^H_{j_2}(\eta) \phi^H_{j_3}(\tau) \\
& = \sum_{j_1=1}^{k^H+1}  \sum_{i_1=1}^{k^L+1} 
\left( \phi^L_{i_1}(\xi_{j_1}^H) \phi^H_{j_1}(\xi)
    \sum_{j_2=1}^{k^H+1}  \sum_{i_2=1}^{k^L+1} 
\left( \phi^L_{i_2}(\eta_{j_2}^H) \phi^H_{j_2}(\eta)
    \sum_{j_3=1}^{m^H+1}  \sum_{i_3=1}^{m^L+1} Q^L_{i_1,i_2,i_3}
    \phi^L_{i_3}(\tau_{j_3}^H)
    \phi^H_{j_3}(\tau) \right) \right).
\end{aligned}
\label{eq:QL_alt_2}
\end{equation}

Since $k^L < k^H$ and $m^L < m^H$, any Lagrange polynomial $\phi_i^L(*)$ can be expressed by the Lagrange polynomial set $\phi_j^H(*)$. Take $\phi_{i_3}^L(\tau)$ and $\phi_{j_3}^H(\tau)$ as an example. We have
\begin{equation*}
    \phi_{i_3}^L (\tau) =
    \sum_{j_3=1}^{m^H+1}   
    \phi^L_{i_3}(\tau_{j_3}^H)
    \phi^H_{j_3}(\tau).
\end{equation*}
Thus, we further have that 
\begin{equation*}
\begin{aligned}
    \sum_{j_3=1}^{m^H+1}  \sum_{i_3=1}^{m^L+1} Q^L_{i_1,i_2,i_3}
    \phi^L_{i_3}(\tau_{j_3}^H)
    \phi^H_{j_3}(\tau) & = 
    \sum_{i_3=1}^{m^L+1} Q^L_{i_1,i_2,i_3} 
\left( \sum_{j_3=1}^{m^H+1}   
    \phi^L_{i_3}(\tau_{j_3}^H)
    \phi^H_{j_3}(\tau) \right) \\
   & =
   \sum_{i_3=1}^{m^L+1} Q^L_{i_1,i_2,i_3} \phi^L_{i_3}(\tau).
\end{aligned}
\end{equation*}

Therefore, Eq.~\eqref{eq:QL_alt_2} can be simplified as follows
\begin{equation}
\begin{aligned}
    Q^{',L}(\vect{\xi}) & = 
\sum_{j_1=1}^{k^H+1}  \sum_{i_1=1}^{k^L+1} 
\left( \phi^L_{i_1}(\xi_{j_1}^H) \phi^H_{j_1}(\xi)
    \sum_{j_2=1}^{k^H+1}  \sum_{i_2=1}^{k^L+1} 
\left( \phi^L_{i_2}(\eta_{j_2}^H) \phi^H_{j_2}(\eta)
     \sum_{i_3=1}^{m^L+1} Q^L_{i_1,i_2,i_3} \phi^L_{i_3}(\tau)
     \right) \right) \\
    & =
\sum_{j_1=1}^{k^H+1}  \sum_{i_1=1}^{k^L+1} 
\left( \phi^L_{i_1}(\xi_{j_1}^H) \phi^H_{j_1}(\xi)
\sum_{i_3=1}^{m^L+1} \phi^L_{i_3}(\tau)
\left(
\sum_{i_2=1}^{k^L+1} Q^L_{i_1,i_2,i_3}
\left(\sum_{j_2=1}^{k^H+1}  
\phi^L_{i_2}(\eta_{j_2}^H) \phi^H_{j_2}(\eta)  
     \right) \right) \right) \\
    & =
\sum_{j_1=1}^{k^H+1}  \sum_{i_1=1}^{k^L+1} 
\left( \phi^L_{i_1}(\xi_{j_1}^H) \phi^H_{j_1}(\xi)
\sum_{i_3=1}^{m^L+1} \sum_{i_2=1}^{k^L+1}
 Q^L_{i_1,i_2,i_3}
\phi^L_{i_3}(\tau) \phi^L_{i_2}(\eta)  
     \right) \\
     & =
\sum_{i_3=1}^{m^L+1} \sum_{i_2=1}^{k^L+1}
\phi^L_{i_3}(\tau) \phi^L_{i_2}(\eta)
\left( 
\sum_{i_1=1}^{k^L+1} Q^L_{i_1,i_2,i_3}
\left(
\sum_{j_1=1}^{k^H+1} \phi^L_{i_1}(\xi_{j_1}^H) \phi^H_{j_1}(\xi)
\right)
\right) \\
      & =
\sum_{i_3=1}^{m^L+1} \sum_{i_2=1}^{k^L+1} \sum_{i_1=1}^{k^L+1}
 Q^L_{i_1,i_2,i_3}
\phi^L_{i_3}(\tau) \phi^L_{i_2}(\eta) \phi^L_{i_1}(\xi) \\
      & = Q^L(\vect{\xi}).
\end{aligned}
\label{eq:QL_alt_3}
\end{equation}
This proves Eq.~\eqref{eq:L_Lagrange}.

Now the difference $\Delta f(\vect{\xi}) = Q^H (\vect{\xi}) - Q^L (\vect{\xi}) \in \mathbb{Q}^{k^H} (\xi, \eta) \otimes P^{m^H} (\tau)$ can be expressed as
\begin{equation}
  \begin{aligned}
    \Delta f(\vect{\xi}) & = 
    \sum_{i_1=1}^{k^H+1}  \sum_{i_2=1}^{k^H+1} \sum_{i_3=1}^{m^H+1} 
    \left(Q^H(\vect{\xi}^H_{i_1,i_2,i_3}) - Q^L(\vect{\xi}_{i_1,i_2,i_3}^H)
    \right)
    \phi^H_{i_1}(\xi) \phi^H_{i_2}(\eta) \phi^H_{i_3}(\tau)   \\
    & =
    \sum_{i_1=1}^{k^H+1}  \sum_{i_2=1}^{k^H+1} \sum_{i_3=1}^{m^H+1} 
    \Delta f_{i_1,i_2,i_3}^H
    \phi^H_{i_1}(\xi) \phi^H_{i_2}(\eta) \phi^H_{i_3}(\tau).
  \end{aligned}
\end{equation}
Thus, its energy $E_g$ can be evaluated from
\begin{equation}
    E_g =  \iiint \Delta f^2(\vect{\xi}) d \vect{\xi} = 
     \sum_{i_1=1}^{k^H+1}  \sum_{i_2=1}^{k^H+1} \sum_{i_3=1}^{m^H+1}
     \Delta f_{i_1,i_2,i_3}^{H,2}  \omega_{s,i_1}^H \omega_{s,i_2}^H \omega_{t,i_3}^H.
\label{eq:energy_f}
\end{equation}

Similarly, based on Eq.~\eqref{eq:Fil_Lagrange}, the difference between $\bar{Q}^H (\vect{\xi})$ and $Q^L (\vect{\xi})$ is 
\begin{equation}
    \Delta \bar{f}(\vect{\xi}) = \theta
    \sum_{i_1=1}^{k^H+1}  \sum_{i_2=1}^{k^H+1} \sum_{i_3=1}^{m^H+1} 
    \Delta f_{i_1,i_2,i_3}^H
    \phi^H_{i_1}(\xi) \phi^H_{i_2}(\eta) \phi^H_{i_3}(\tau).
\end{equation}
Therefore, following Eq.~\eqref{eq:energy_f}, its energy $\bar{E}_g$ is
\begin{equation}
    \bar{E}_g =  \iiint \Delta \bar{f}^2(\vect{\xi}) d \vect{\xi} = \theta^2
     \sum_{i_1=1}^{k^H+1}  \sum_{i_2=1}^{k^H+1} \sum_{i_3=1}^{m^H+1}
     \Delta f_{i_1,i_2,i_3}^{H,2}  \omega_{s,i_1}^H \omega_{s,i_2}^H \omega_{t,i_3}^H
     =\theta^2 E_g.
\end{equation}

It is clear that a $(1-\theta^2)$ portion of $E_g$ is exactly filtered from $Q^H (\vect{\xi})$ with Eq.~\eqref{eq:Fil_Lagrange}.

\bibliographystyle{elsarticle-num}
\bibliography{summary}

@inproceedings{Yu_Space_Time_Moving_2024,
  title={High-order Space-time Flux Reconstruction Methods for Moving Domain Simulation},
  author={Yu, M.},
doi = {
https://doi.org/10.48550/arXiv.2409.14005},
url = {https://arxiv.org/abs/2409.14005},
  booktitle={the 12th International Conference on
Computational Fluid Dynamics (ICCFD12)},
  address = {Kobe, Japan},
  year={2024}
}

@article{Yu2011,
title = {A high-order spectral difference method for unstructured dynamic grids},
journal = {Computers \& Fluids},
volume = {48},
number = {1},
pages = {84-97},
year = {2011},
issn = {0045-7930},
doi = {https://doi.org/10.1016/j.compfluid.2011.03.015},
url = {https://www.sciencedirect.com/science/article/pii/S0045793011001162},
author = {M.L. Yu and Z.J. Wang and H. Hu}
}

@article{YuWang_JSC_2013,
  title={On the Connection Between the Correction and Weighting Functions in the Correction Procedure via Reconstruction Method},
  author={Yu, M.L. and Wang, Z.J.},
  journal={Journal of Scientific Computing},
  volume={54},
  pages={227–244},
  year={2013}
}

@article{Yu2016,
title = {A high-order flux reconstruction/correction procedure via reconstruction formulation for unsteady incompressible flow on unstructured moving grids},
journal = {Computers \& Fluids},
volume = {139},
pages = {161-173},
year = {2016},
note = {13th USNCCM International Symposium of High-Order Methods for Computational Fluid Dynamics - A special issue dedicated to the 60th birthday of Professor David Kopriva},
issn = {0045-7930},
doi = {https://doi.org/10.1016/j.compfluid.2016.05.028},
url = {https://www.sciencedirect.com/science/article/pii/S004579301630175X},
author = {Meilin Yu and Lai Wang}
}

@inproceedings{Yu_Space_Time_2017,
  title={Nodal Space-Time Flux Reconstruction Methods for Conservation Laws},
  author={Yu, M.},
  booktitle={the 23rd AIAA Computational Fluid Dynamics Conference},
  address = {Denver, Colorado},
  pages={3095},
  year={2017}
}

@article{Wang_Yu_JSC_2020,
  title={{Comparison of ROW, ESDIRK, and BDF2 for Unsteady Flows with the High-Order Flux Reconstruction Formulation}},
  author={Wang, L. and Yu, M.},
  journal={Journal of Scientific Computing},
  volume={83},
  pages={1-27},
  year={2020},
  publisher={Springer}
}

@article{YU201470,
title = {On the accuracy and efficiency of discontinuous Galerkin, spectral difference and correction procedure via reconstruction methods},
journal = {Journal of Computational Physics},
volume = {259},
pages = {70-95},
year = {2014},
issn = {0021-9991},
doi = {https://doi.org/10.1016/j.jcp.2013.11.023},
url = {https://www.sciencedirect.com/science/article/pii/S0021999113007857},
author = {Meilin Yu and Z.J. Wang and Yen Liu}
}

@article{Huynh_JSC_2023,
  title={Discontinuous galerkin and related methods for ODE},
  author={Huynh, H. T.},
  journal={Journal of Scientific Computing},
  volume={96},
  pages={51},
  year={2023}
}

@inproceedings{huynh2007,
  title={A flux reconstruction approach to high-order schemes including discontinuous Galerkin methods},
  author={Huynh, H. T.},
  booktitle={18th AIAA Computational Fluid Dynamics Conference},
  pages={4079},
  year={2007}
}

@inproceedings{huynh2009,
  title={A Reconstruction Approach to High-Order Schemnes Including Discontinuous Galerkin for Diffusion},
  author={Huynh, H. T.},
  booktitle={47th AIAA Aerospace Sciences Meeting including The New Horizons Forum and Aerospace Exposition},
  pages={403},
  year={2009}
}

@article{wang2009,
  title={A unifying lifting collocation penalty formulation including the discontinuous Galerkin, spectral volume/difference methods for conservation laws on mixed grids},
  author={Wang, Zhijian and Gao, H.},
  journal={Journal of Computational Physics},
  volume={228},
  number={21},
  pages={8161--8186},
  year={2009},
  publisher={Elsevier}
}

@article{vincent2011,
  title={A new class of high-order energy stable flux reconstruction schemes},
  author={Vincent, P. E. and Castonguay, P. and Jameson, A.},
  journal={Journal of Scientific Computing},
  volume={47},
  number={1},
  pages={50--72},
  year={2011},
  publisher={Springer}
}

@article{bassi1997,
  title={A high-order accurate discontinuous finite element method for the numerical solution of the compressible Navier--Stokes equations},
  author={Bassi, F. and Rebay, S.},
  journal={Journal of computational physics},
  volume={131},
  number={2},
  pages={267--279},
  year={1997},
  publisher={Elsevier}
}

@article{Wang_EtAl2013,
	title        = {High-order CFD methods: current status and perspective},
	author       = {Wang, Z.J. and Fidkowski, K. and Abgrall, R. and Bassi, F. and Caraeni, D. and Cary, A. and  Deconinck, H. and Hartmann, R. and Hillewaert, K. and  Huynh, H. and Kroll,  N. and  May, G. and  Persson, P.-O. and  von Leer, B. and Visbal, M.},
	year         = {2013},
	journal      = {Int. J. Numer. Meth. Fluids},
	volume       = {72},
	pages        = {811--845}
}

@article{Cockburn_DG_1989,
  author = "B. Cockburn and C.-W. Shu",
  title = "{TVB Runge–Kutta local projection discontinuous Galerkin finite element method for conservation laws II: general framework}",
  year = "1989",
  journal = "Mathematics of Computation",
  volume = "52",
  pages = "411-435"
}

@article{Wang_SV_2002,
  author = "Z. J. Wang",
  title = "{Spectral (finite) volume method for conservation laws on unstructured grids: basic formulation}",
  year = "2002",
  journal = "Journal of Computational Physics",
  volume = "178",
  pages = "210-251"
}

@article{Wang_SD_2006,
  author = "{Y. Liu, M. Vinokur and Z. J. Wang}",
  title = "{Spectral difference method for unstructured grids I: Basic formulation}",
  year = "2006",
  journal = "Journal of Computational Physics",
  volume = "216",
  pages = "780-801"
}

@article{Kopriva:1996,
  author = "D. A. Kopriva and J. H. Kolias",
  title = "{A conservative staggered-grid Chebyshev multidomain method for compressible flows}",
  year = "1996",
  journal = "Journal of Computational Physics",
  volume = "125",
  pages = "244-261"
}

@article{PerssonEtAl_CMAME_2009,
title = {Discontinuous Galerkin solution of the Navier–Stokes equations on deformable domains},
journal = {Computer Methods in Applied Mechanics and Engineering},
volume = {198},
number = {17},
pages = {1585-1595},
year = {2009},
issn = {0045-7825},
doi = {https://doi.org/10.1016/j.cma.2009.01.012},
url = {https://www.sciencedirect.com/science/article/pii/S0045782509000450},
author = {P.-O. Persson and J. Bonet and J. Peraire}
}

@article{AbeEtAl_JCP_2015,
title = {On the freestream preservation of high-order conservative flux-reconstruction schemes},
journal = {Journal of Computational Physics},
volume = {281},
pages = {28-54},
year = {2015},
issn = {0021-9991},
doi = {https://doi.org/10.1016/j.jcp.2014.10.011},
url = {https://www.sciencedirect.com/science/article/pii/S0021999114006937},
author = {Yoshiaki Abe and Takanori Haga and Taku Nonomura and Kozo Fujii}
}

@article{Kopriva_JSC_2006,
  title={Metric Identities and the Discontinuous Spectral Element Method on Curvilinear Meshes},
  author={Kopriva, D. A.},
  journal={Journal of Scientific Computing},
  volume={26},
  pages={301-327},
  year={2006},
  publisher={Springer}
}

@inproceedings{Nishikawa_Padway_AIAAAva_2020,
  title={An Adaptive Space-Time Edge-Based Solver for Two-Dimensional Unsteady Inviscid Flows},
  author={Nishikawa, Hiroaki and Padway, Emmett},
  booktitle={AIAA AVIATION 2020 FORUM},
  pages={3024},
  year={2020}
}

@article{FRONTIN_EtAl_2021_ANM,
title = {Foundations of space-time finite element methods: Polytopes, interpolation, and integration},
journal = {Applied Numerical Mathematics},
volume = {166},
pages = {92-113},
year = {2021},
issn = {0168-9274},
doi = {https://doi.org/10.1016/j.apnum.2021.03.019},
url = {https://www.sciencedirect.com/science/article/pii/S0168927421000994},
author = {Cory V. Frontin and Gage S. Walters and Freddie D. Witherden and Carl W. Lee and David M. Williams and David L. Darmofal}
}

@article{LiangEtAl_JCP2014,
title = {An efficient correction procedure via reconstruction for simulation of viscous flow on moving and deforming domains},
journal = {Journal of Computational Physics},
volume = {256},
pages = {55-68},
year = {2014},
issn = {0021-9991},
doi = {https://doi.org/10.1016/j.jcp.2013.08.046},
url = {https://www.sciencedirect.com/science/article/pii/S0021999113005901},
author = {Chunlei Liang and Koji Miyaji and Bin Zhang}
}

@article{DONEA_EtAl_CMAME_1982,
title = {An arbitrary lagrangian-eulerian finite element method for transient dynamic fluid-structure interactions},
journal = {Computer Methods in Applied Mechanics and Engineering},
volume = {33},
number = {1},
pages = {689-723},
year = {1982},
issn = {0045-7825},
doi = {https://doi.org/10.1016/0045-7825(82)90128-1},
url = {https://www.sciencedirect.com/science/article/pii/0045782582901281},
author = {J. Donea and S. Giuliani and J.P. Halleux}
}

@article{MAVRIPLIS_JCP_2011,
title = {On the geometric conservation law for high-order discontinuous Galerkin discretizations on dynamically deforming meshes},
journal = {Journal of Computational Physics},
volume = {230},
number = {11},
pages = {4285-4300},
year = {2011},
note = {Special issue High Order Methods for CFD Problems},
issn = {0021-9991},
doi = {https://doi.org/10.1016/j.jcp.2011.01.022},
url = {https://www.sciencedirect.com/science/article/pii/S0021999111000477},
author = {Dimitri J. Mavriplis and Cristian R. Nastase}
}

@article{KOPRIVA_EtAl_CF_2016,
title = {A provably stable discontinuous Galerkin spectral element approximation for moving hexahedral meshes},
journal = {Computers \& Fluids},
volume = {139},
pages = {148-160},
year = {2016},
note = {13th USNCCM International Symposium of High-Order Methods for Computational Fluid Dynamics - A special issue dedicated to the 60th birthday of Professor David Kopriva},
issn = {0045-7930},
doi = {https://doi.org/10.1016/j.compfluid.2016.05.023},
url = {https://www.sciencedirect.com/science/article/pii/S0045793016301700},
author = {David A. Kopriva and Andrew R. Winters and Marvin Bohm and Gregor J. Gassner}
}

@article{ABE_EtAl_CF_2016,
title = {Conservative high-order flux-reconstruction schemes on moving and deforming grids},
journal = {Computers \& Fluids},
volume = {139},
pages = {2-16},
year = {2016},
note = {13th USNCCM International Symposium of High-Order Methods for Computational Fluid Dynamics - A special issue dedicated to the 60th birthday of Professor David Kopriva},
issn = {0045-7930},
doi = {https://doi.org/10.1016/j.compfluid.2016.03.024},
url = {https://www.sciencedirect.com/science/article/pii/S0045793016300883},
author = {Yoshiaki Abe and Takanori Haga and Taku Nonomura and Kozo Fujii}
}

@article{Wang_Persson_CF_2015,
title = {A high-order discontinuous Galerkin method with unstructured space–time meshes for two-dimensional compressible flows on domains with large deformations},
journal = {Computers \& Fluids},
volume = {118},
pages = {53-68},
year = {2015},
issn = {0045-7930},
doi = {https://doi.org/10.1016/j.compfluid.2015.05.026},
url = {https://www.sciencedirect.com/science/article/pii/S0045793015001814},
author = {Luming Wang and Per-Olof Persson}
}

@article{VANDERVEGT_JCP_2002,
title = {Space–Time Discontinuous Galerkin Finite Element Method with Dynamic Grid Motion for Inviscid Compressible Flows: I. General Formulation},
journal = {Journal of Computational Physics},
volume = {182},
number = {2},
pages = {546-585},
year = {2002},
issn = {0021-9991},
doi = {https://doi.org/10.1006/jcph.2002.7185},
url = {https://www.sciencedirect.com/science/article/pii/S0021999102971858},
author = {J.J.W. {van der Vegt} and H. {van der Ven}}
}

@article{KLAIJ_EtAl_JCP_2006,
title = {Space–time discontinuous Galerkin method for the compressible Navier–Stokes equations},
journal = {Journal of Computational Physics},
volume = {217},
number = {2},
pages = {589-611},
year = {2006},
issn = {0021-9991},
doi = {https://doi.org/10.1016/j.jcp.2006.01.018},
url = {https://www.sciencedirect.com/science/article/pii/S0021999106000246},
author = {C.M. Klaij and J.J.W. {van der Vegt} and H. {van der Ven}}
}

@article{Petersen_EtAl_IJNME_2009,
title = {A space–time discontinuous Galerkin method for the solution of the wave equation in the time domain},
journal = {Int. J. Numer. Meth. Engng},
volume = {78},
pages = {275-295},
year = {2009},
author = {Steffen Petersen and Charbel Farhat and Radek Tezaur}
}

@InProceedings{Feistauer_Cesenek_2011,
author="Feistauer, Miloslav
and {\v{C}}esenek, Jan",
editor="Dimov, Ivan
and Dimova, Stefka
and Kolkovska, Natalia",
title="Space-Time Discontinuous Galerkin Finite Element Method for Convection-Diffusion Problems and Compressible Flow",
booktitle="Numerical Methods and Applications",
year="2011",
publisher="Springer Berlin Heidelberg",
address="Berlin, Heidelberg",
pages="1--13",
isbn="978-3-642-18466-6"
}

@article{Corrigan_EtAl_IJNMF_2019,
title = {A moving discontinuous Galerkin finite element method for flows with interfaces},
journal = {Int. J. Numer. Methods Fluids},
volume = {89},
pages = {362-406},
year = {2019},
author = {Andrew Corrigan and Andrew D. Kercher and David A. Kessler}
}

@article{LUO_EtAl_JCP_2021,
title = {A moving discontinuous Galerkin finite element method with interface condition enforcement for compressible flows},
journal = {Journal of Computational Physics},
volume = {445},
pages = {110618},
year = {2021},
issn = {0021-9991},
doi = {https://doi.org/10.1016/j.jcp.2021.110618},
url = {https://www.sciencedirect.com/science/article/pii/S0021999121005131},
author = {Hong Luo and Gianni Absillis and Robert Nourgaliev}
}

@inproceedings{Mavriplis_Nastase_AIAA_2008,
  title={{On the Geometric Conservation Law for High-order Discontinuous Galerkin Discretizations on Dynamically Deforming Meshes}},
  author={Mavriplis, Dimitri J. and Nastase, Cristian R.},
  booktitle={46th AIAA Aerospace Sciences Meeting and Exhibit},
  pages={778},
  year={2008}
}

@article{FARHAT_CMAME_2004,
title = {Design and analysis of robust ALE time-integrators for the solution of unsteady flow problems on moving grids},
journal = {Computer Methods in Applied Mechanics and Engineering},
volume = {193},
number = {39},
pages = {4073-4095},
year = {2004},
note = {The Arbitrary Lagrangian-Eulerian Formulation},
issn = {0045-7825},
doi = {https://doi.org/10.1016/j.cma.2003.09.027},
url = {https://www.sciencedirect.com/science/article/pii/S0045782504002130},
author = {Charbel Farhat and Philippe Geuzaine}
}

@book{Butcher2002,
  title = {Numerical Methods for Ordinary Differential Equations},
  author = {J. C. Butcher},
  year = {2002},
  publisher = {Wiley},
  address   = {Chichester}
}

@article{BIJL_EtAl_JCP_2002,
title = {Implicit Time Integration Schemes for the Unsteady Compressible Navier–Stokes Equations: Laminar Flow},
journal = {Journal of Computational Physics},
volume = {179},
number = {1},
pages = {313-329},
year = {2002},
issn = {0021-9991},
doi = {https://doi.org/10.1006/jcph.2002.7059},
url = {https://www.sciencedirect.com/science/article/pii/S0021999102970592},
author = {Hester Bijl and Mark H. Carpenter and Veer N. Vatsa and Christopher A. Kennedy}
}

@article{CICCHINO_EtAl_JCP_2022,
title = {Provably stable flux reconstruction high-order methods on curvilinear elements},
journal = {Journal of Computational Physics},
volume = {463},
pages = {111259},
year = {2022},
issn = {0021-9991},
doi = {https://doi.org/10.1016/j.jcp.2022.111259},
url = {https://www.sciencedirect.com/science/article/pii/S0021999122003217},
author = {Alexander Cicchino and David C. {Del Rey Fernández} and Siva Nadarajah and Jesse Chan and Mark H. Carpenter}
}

@article{HUGHES_CMAME_1988,
title = {Space-time finite element methods for elastodynamics: Formulations and error estimates},
journal = {Computer Methods in Applied Mechanics and Engineering},
volume = {66},
number = {3},
pages = {339-363},
year = {1988},
issn = {0045-7825},
doi = {https://doi.org/10.1016/0045-7825(88)90006-0},
url = {https://www.sciencedirect.com/science/article/pii/0045782588900060},
author = {Thomas J.R. Hughes and Gregory M. Hulbert}
}

@article{KocherBause_JSC_2014,
  title={Variational Space–Time Methods for the Wave Equation},
  author={K{\"o}cher, U. and Bause, M.},
  journal={Journal of Scientific Computing},
  volume={61},
  pages={424–453},
  year={2014}
}

@inproceedings{Huynh_AIAA_2013,
  title={High-order space-time methods for conservation laws},
  author={Huynh, H. T.},
  booktitle={21st AIAA Computational Fluid Dynamics Conference},
  pages={2432},
  year={2013}
}

@article{Ehle_SJMA_1973,
  title={A-stable methods and {Pad{\,e}} approximations to the exponential},
  author={Ehle, B. L.},
  journal={SIAM J.Math. Anal.},
  volume={4},
  pages={671-680},
  year={1973}
}

@article{Bottasso_ANM_97,
author = {C. L. Bottasso},
year = {1997},
title = {A new look at finite elements in time: a variational interpretation of Runge-Kutta methods},
journal = {Applied Numerical Mathematics},
volume = {25},
number = {4},
pages = {355-368},
}

@article{Tang_Sun_AMC_12,
author = {W. Tang and Y. Sun},
year = {2012},
title = {Time finite element methods: A unified framework for numerical discretizations of ODEs},
journal = {Applied Mathematics and Computation},
volume = {219},
number = {4},
pages = {2158-2179},
}

@book{Hesthaven_Warburton_08,
author = {J. S. Hesthaven and T. Warburton},
year = {2008},
title = {Nodal Discontinuous Galerkin methods: algorithms, analysis, and applications},
address = {New York},
publisher = {Springer},
}

@article{Huynh_EtAl_CF_14,
author = {H. T. Huynh and Z. J. Wang and P. E. Vincent},
year = {2014},
title = {High-order methods for computational fluid dynamics: A brief review of compact differential formulations on unstructured grids},
journal = {Computers \& Fluids},
volume = {98},
pages = {209-220},
}

@article{Behr_IJNMF_2008,
title = {Simplex space–time meshes in finite element simulations},
journal = {Int. J. Numer. Methods Fluids},
volume = {57},
pages = {1421-1434},
year = {2008},
author = {Marek Behr}
}

@article{TEZDUYAR_EtAl_CMAME_1992,
title = {A new strategy for finite element computations involving moving boundaries and interfaces—The deforming-spatial-domain/space-time procedure: I. The concept and the preliminary numerical tests},
journal = {Computer Methods in Applied Mechanics and Engineering},
volume = {94},
number = {3},
pages = {339-351},
year = {1992},
issn = {0045-7825},
doi = {https://doi.org/10.1016/0045-7825(92)90059-S},
url = {https://www.sciencedirect.com/science/article/pii/004578259290059S},
author = {T.E. Tezduyar and M. Behr and J. Liou}
}

@inproceedings{Wukie_EtAl_AIAA_2023,
  title={High-Fidelity CFD Verification Workshop 2024: Mesh Motion},
  author={Nathan A. Wukie and Krzysztof Fidkowski and Per-Olof Persson and Zhi J. Wang},
  booktitle={AIAA SCITECH 2023 Forum},
  pages={1243},
  year={2023}
}

@article{Gottlieb_EtAl_SIAM_2001,
author = {Gottlieb, Sigal and Shu, Chi-Wang and Tadmor, Eitan},
title = {Strong Stability-Preserving High-Order Time Discretization Methods},
journal = {SIAM Review},
volume = {43},
number = {1},
pages = {89-112},
year = {2001},
doi = {10.1137/S003614450036757X},
URL = {      https://doi.org/10.1137/S003614450036757X
},
eprint = {      https://doi.org/10.1137/S003614450036757X
}
}

@article{YOU_EtAl_CPC_23,
title = {Deneb: An open-source high-performance multi-physical flow solver based on high-order DRM-DG method},
journal = {Computer Physics Communications},
volume = {286},
pages = {108672},
year = {2023},
issn = {0010-4655},
doi = {https://doi.org/10.1016/j.cpc.2023.108672},
url = {https://www.sciencedirect.com/science/article/pii/S0010465523000176},
author = {Hojun You and Juhyun Kim and Chongam Kim}
}

@article{RUSANOV1962,
title = {The calculation of the interaction of non-stationary shock waves and obstacles},
journal = {USSR Computational Mathematics and Mathematical Physics},
volume = {1},
number = {2},
pages = {304-320},
year = {1962},
issn = {0041-5553},
doi = {https://doi.org/10.1016/0041-5553(62)90062-9},
url = {https://www.sciencedirect.com/science/article/pii/0041555362900629},
author = {V.V Rusanov}
}

\end{document}